\documentclass[12pt,reqno]{amsart}
\usepackage{amsfonts}
\usepackage{amsmath}
\usepackage{amssymb}
\usepackage{mathtools}
\usepackage{fullpage}
\usepackage{amsthm}
\usepackage{enumitem}
\usepackage{graphicx}
\usepackage{relsize}
\usepackage{todonotes}
\usepackage{hyperref}
\usepackage{tikz}
\usepackage{cleveref}
\usepackage{thm-restate}
\usepackage[normalem]{ulem}
\usepackage[margin=1in]{geometry} 
\usetikzlibrary{decorations.markings,arrows,arrows.meta,cd} 
\tikzset{->-/.style={decoration={
  markings,
  mark=at position #1 with {\arrow{Classical TikZ Rightarrow[length=3pt]}}},postaction={decorate}}}

\newtheorem{theorem}{Theorem}[section]
\newtheorem{proposition}[theorem]{Proposition}
\newtheorem{lemma}[theorem]{Lemma} 
\newtheorem{corollary}[theorem]{Corollary}

\theoremstyle{definition}
\newtheorem{definition}[theorem]{Definition}

\newtheorem{conjecture}[theorem]{Conjecture}
\newtheorem{question}[theorem]{Question}

\newtheorem*{wordprob}{The Word Problem}

\newcommand{\la}{\langle}
\newcommand{\ra}{\rangle}
\newcommand{\A}{{\mathcal A}}
\newcommand{\B}{{\mathcal B}}
\newcommand{\N}{{\mathbb N}}
\newcommand{\R}{{\mathbb R}}
\renewcommand{\S}{{\mathbb S}}
\newcommand{\Z}{{\mathbb Z}}
\newcommand{\prq}{\preceq}
\newcommand{\Lk}{\operatorname{Lk}}
\newcommand{\id}{\operatorname{id}}
\newcommand{\FA}{\operatorname{FA}}
\newcommand{\Vertices}{\operatorname{Vertices}}
\newcommand{\Edges}{\operatorname{Edges}}

\newcommand{\HArea}{\operatorname{HArea}}
\newcommand{\Area}{\operatorname{Area}}
\newcommand{\DArea}{\operatorname{Area}}
\newcommand{\CAT}{\operatorname{CAT}}
\newcommand{\mass}{\operatorname{mass}}
\newcommand{\supp}{\operatorname{supp}}
\newcommand{\Lip}{\operatorname{Lip}}

\DeclareMathOperator*{\bigast}{\raisebox{-0.5ex}{\scalebox{2.0}{$\ast$}}}

\makeatletter
\@namedef{subjclassname@2020}{\textup{2020} Mathematics Subject Classification}
\makeatother

\begin{document}

\date{\today}	
	
\title{Homological Dehn functions of groups of type $FP_2$}
	
\author{Noel Brady}
\address{Department of Mathematics\\
University of Oklahoma \\
Norman\\   OK 73019\\ USA} \email{nbrady@ou.edu}

\author{Robert Kropholler}
\address{Mathematics M\"unster\\
Fachbereich Mathematik und Informatik der Universit\"at M\"unster\\
Orl\'eans-Ring 10 \\48149 M\"unster\\ Germany}
\email{robertkropholler@gmail.com}

\author{Ignat Soroko}
\address{Department of Mathematics\\
Louisiana State University\\
Baton Rouge\\ LA 70803\\ USA}
\email{ignat.soroko@gmail.com}
	

	\begin{abstract}
	    We prove foundational results for homological Dehn functions of groups of type $FP_2$ such as superadditivity and the invariance under quasi-isometry.
	    We then study the homological Dehn functions for Leary's groups $G_L(S)$ providing methods to obtain uncountably many groups with a given homological Dehn function. 
	    This allows us to show that there exist groups of type $FP_2$ with quartic homological Dehn function and unsolvable word problem.
	\end{abstract}
	
	\maketitle
	
	\section{Introduction}
	
	Geometric group theory can trace its roots back to the foundational work of Dehn \cite{Dehn1911}, in which he posed his famous three fundamental problems of group theory. 
	The first of these problems is:
	\begin{wordprob}
		Let $H$ be a finitely generated group with a generating set $S$. 
		Is there an algorithm that takes as input words in $S$ and outputs whether or not they are trivial?
	\end{wordprob}
	Dehn's seminal work showed that this problem is solvable for fundamental groups of surfaces. 
	However, in the 1950's Novikov~\cite{Nov} and Boone~\cite{Boo} showed that there exist finitely presented groups for which the word problem is unsolvable.
	
	For finitely presented groups the Dehn function $\delta_H$ determines exactly when such an algorithm exists.
	Moreover, there is an upper bound on the time complexity of such an algorithm in terms of $\delta_H$. 
	The Dehn function of finitely presented groups has received a great deal of interest over the last 100 years. 
	For a more complete background we refer the reader to \cite{BriChap} and the references therein.
	
	For infinitely presented groups the notion of Dehn function is not as useful as for finitely presented ones, since it can always be made linear. Hence it is desirable to find a suitable analog of the Dehn function for at least certain classes of infinitely presented groups.
	
	Groups of type $FP_2$ form a class of groups for which this can be done. While almost all of these groups are infinitely presented, they enjoy a remarkable property that their relation module is  finitely generated. This makes it natural to consider homological variants of the Dehn function, namely, the abelianised Dehn function $\delta_H^{ab}$ and the homological Dehn function~$\FA_H$. 

	Geometrically, the Dehn function can be viewed as a notion of area. 
	Namely, the Dehn function measures the area required to fill a loop of a given length with a disk,  
	the abelianised Dehn function measures the area required to fill a loop of a given length with a surface, and
	the homological Dehn function, measures the area required to fill a $1$-cycle with a $2$-chain. 
	For finitely presented groups all three functions have received a great deal of interest. 
	There are various bounds known, for instance, $\delta_H^{ab}\preceq \delta_H$ \cite{BMS}. 
	It is also known that for finitely presented groups the above inequality can be strict 
	\cite{ABDY}. 
	
	For groups of type $FP_2$ some aspects of homological Dehn functions were studied in~\cite{Ger,HM,MP,FM18}.
	In this article, we unify all these results into a single framework by considering homological finite presentations (\Cref{def:CayCx}). Using this notion we prove foundational results for homological Dehn functions such as their superadditivity (\Cref{prop:superadd}) and the invariance under quasi-isometry (\Cref{thm:qi}).
	
	A natural hope is that for groups of type $FP_2$, the homological Dehn function provides a notion of complexity for the word problem similar to that provided by the Dehn function of a finitely presented group. 
	We show that this is not the case by proving the following result.
	
	\begin{restatable*}{corollary}{thmsubrecursunsolv}
		There exists a group $H$ of type $FP_2$ with $\FA_H(n)\simeq n^4$ which has unsolvable word problem. 
	\end{restatable*}
	
	It is unclear whether one can prove the corresponding result for finitely presented groups and we leave this as an open question. 
	\begin{question}
		Suppose that $H$ is finitely presented. 
		If $\FA_H$ is bounded by a recursive function, does $H$ have decidable word problem?
	\end{question}
	
	Our proof of this result stems from the work of Leary. 
	In \cite{Lea}, it is shown that the class of groups of type $FP_2$ is uncountable, in stark contrast to the countable class of finitely presented groups. 
	The construction of Leary takes as input a flag complex $L$ and an arbitrary subset $S\subset \Z$ and outputs a group $G_L(S)$. 
	It is shown that $G_L(S)$ has type $FP_2$ if $H_1(L) = 0$.
	Thus, one can say that the property of being $FP_2$ is independent of $S$. 
	We prove a similar result for the homological Dehn functions of the groups $G_L(S)$. 
	Namely, combining Theorems \ref{thm:boundingfill} and \ref{thm:lowerbound} we obtain the following theorem. 
	\begin{theorem}\label{thm:1.2}
		Let $L$ be a finite connected flag complex with no local cut points and $H_1(L) = 0$. 
		Let $\FA_L$ be the homological Dehn function for $\pi_1(L)$ and let $H = G_L(S)$. 
		Then the following hold:
		\begin{enumerate}
			\item for any $S$, we have $\FA_H(n)\preceq n^4\FA_L(n)$;
			\item for any $S$, we have $\FA_H(n)\preceq n^2\FA_L(n^2)$; 
			\item if $S\neq \Z$, then $\FA_H(n)\succeq \FA_L(n)$. 
		\end{enumerate}
	\end{theorem}
	It is worth noting that the case $S = \Z$ has been previously studied in \cite{ABDDY}, where it is shown that the homological Dehn function for $G_L(\Z)$ is always bounded by $n^4$. 
	
	By taking $\pi_1(L)$ to be a perfect hyperbolic group we obtain uncountably many groups of type $FP_2$ with the homological Dehn function bounded above by $n^4$. 
	In fact, by utilizing specific flag complexes $L$, we can show that this bound is sharp. 
	Moreover, since the homological Dehn function in this case is independent of $S$ we can use the results from \cite{MOW} to prove the following theorem. 
	\begin{restatable*}{theorem}{thmquartic}\label{thm:qiquartic}
		There are uncountably many quasi-isometry classes of groups of the form $G_L(S)$ of type $FP$ with the homological Dehn function $n^4$.
	\end{restatable*}
	
	We obtain similar results for iterated exponentials. 
	\begin{restatable*}{theorem}{thmiteratedexp}
		There are uncountably many quasi-isometry classes of groups of the form $G_L(S)$ of type $FP$ with the homological Dehn function $\simeq \exp^{(m)}(n)$. 
	\end{restatable*}
	
	When we consider the larger class of groups of type $FP_2$, we can obtain similar results for many functions growing faster than $n^4$. 
	Namely, we obtain the following theorem. 
	
	\begin{restatable*}{theorem}{thmsnowflake}\label{thm:manyclasses}
		Let $H$ be a finitely presented group. 
		Suppose that $\FA_H\succeq n^4$. 
		Then there are uncountably many quasi-isometry classes of groups with homological Dehn function $\FA_H$. 
		Moreover, these groups can be constructed to satisfy the same homological finiteness properties as $H$. 
	\end{restatable*}
	
	Our work here suggests that the homological Dehn function of $G_L(S)$ is independent of $S$ as long as $S\neq \Z$. 
	This leads us to the following fundamental question. 
	
	\begin{restatable*}{question}{Qdependence}
		Let $S, T\subsetneq \Z$. 
		Is the filling function for $G_L(S)$ the same as that for $G_L(T)$?
	\end{restatable*}
	
	The uncountable set of choices for $S$ is key in showing that there are uncountably many groups of type $FP_2$. 
	If one were able to show that there is a flag complex $L$ for which the homological Dehn function of $G_L(S)$ depends on $S$ in a nontrivial way, one may hope to show that there are uncountably many functions that appear as homological Dehn functions. Denote the \emph{homological isoperimetric spectrum} as the set
	\[
	\operatorname{HIP} = \{\alpha\mid n^{\alpha} \text{ is the homological Dehn function of a group of type $FP_2$} \}.
	\]
	We leave the following as an open question that appears to be a gap in our understanding of groups of type $FP_2$. 
	\begin{restatable*}{question}{Quncountii}
		    Do we have an equality $\operatorname{HIP} = \{1\}\cup [2, \infty)$?
	\end{restatable*}
	
	The general plan of the paper is as follows. 
	In Section \ref{sec:homfunctions} we discuss the basics of homological filling functions and prove many key theorems that are required for the general subject. 
	In Section \ref{sec:leary} we give the necessary background on $M_\kappa$-polyhedral complexes as well as the theorems we require from \cite{Lea}.
	We also prove various theorems about controlled fillings in CAT(0) cube complexes, using ideas from~\cite{NR}. These results are used in \Cref{sec:pushbounds}, but may present an independent interest. 
	The upper and lower bounds for~\Cref{thm:1.2} are proved in \Cref{sec:pushbounds}, where also the examples required for \Cref{thm:uncountablymanyexp} are given. In \Cref{sec:n4bounds} we prove \Cref{thm:qiquartic} giving uncountably many groups of type $FP_2$ with quartic homological Dehn function. In \Cref{sec:wordproblem}, we establish an exact criterion when the word problem for $G_L(S)$ is solvable. Finally, in \Cref{sec:con}, we discuss extending our results to groups that are not of the form $G_L(S)$ and discuss a few intriguing open questions. 
	
	\subsection*{Acknowledgements} The authors acknowledge support from Simons Foundation award 430097. The second author was funded by the Deutsche Forschungsgemeinschaft (DFG, German Research Foundation) under Germany's Excellence Strategy EXC 2044--390685587, Mathematics M\"unster: Dynamics--Geometry--Structure. The third author was supported by the AMS--Simons Travel Grant. The authors are grateful to \mbox{Eduardo} Mart\'{\i}nez-Pedroza for his comments which improved the quality of this paper.

	\section{Homological Dehn functions}\label{sec:homfunctions}
	
	\subsection{Finiteness properties of groups}
	
	Throughout the paper we will be interested in finding groups with certain homological finiteness properties. 
	We are primarily concerned with the finiteness conditions $FP_n$ and $FP$. 
	For completeness recall the definitions here. 
	\begin{definition}
		A group $H$ is of type $FP_n$ if there is an exact sequence of the form:
		\begin{equation*}
		P_n\longrightarrow P_{n-1}\longrightarrow \,\dots\,\longrightarrow P_1\longrightarrow P_0\longrightarrow \Z\longrightarrow 0
		\end{equation*}
		where $P_i$ are finitely generated projective $\Z H$-modules. A group $H$ is of type $FP_\infty$ if it is of type $FP_n$ for all $n\in\N$.  
	\end{definition}
	
	\begin{definition}
		A group $H$ is of type $FP$, if there is an exact sequence of the form:
		\begin{equation*}
		0\longrightarrow P_n\longrightarrow P_{n-1}\longrightarrow \,\dots\,\longrightarrow P_1\longrightarrow P_0\longrightarrow \Z\longrightarrow 0
		\end{equation*}
		where $P_i$ are finitely generated projective $\Z H$-modules.  
	\end{definition}
	
	We will be particularly interested in the case of groups of type $FP_2$. In this case, there are many equivalent conditions to the above. 
	
	\begin{definition}
		Let $H$ be a finitely generated group given by the presentation $\la \A\,|\, R\ra$, i.e. there is a short exact sequence 
		\[
		1\longrightarrow N\longrightarrow  F(\A)\longrightarrow  H\longrightarrow  1
		\]
		where $F(\A)$ is a free group on a finite set of free generators $\A$ and $N=\la\!\la R\ra\!\ra$ is the normal closure of $R$ in $F(\A)$. The action by conjugation of $F(\A)$ on $N$ induces an action of $H$ on $N^{ab}=N/[N,N]$. 
		We call $N^{ab}$ with this action the {\em relation module} for $H$. 
	\end{definition}
	
	Strictly speaking the relation module is an invariant of a given presentation. However, we will see that finite generation of the relation module is an invariant of $H$ and not the choice of presentation. 
	
	Throughout this paper we will be interested in various types of complexes. All the complexes considered can be given the structure of combinatorial complexes. These are  CW complexes with  combinatorial restrictions on the attaching maps of cells. We refer the reader to~\cite[Appendix]{Hat02} for the background and the notation of CW complexes.
	
	\begin{definition}[Combinatorial complexes and maps] Combinatorial complexes and combinatorial maps between them are defined recursively on dimension. Zero dimensional cell complexes are defined to be combinatorial as are arbitrary maps between them. In general, a continuous map between CW complexes is said to be \emph{combinatorial} if it maps each open cell of the domain homeomorphically to an open cell of the target. Having defined combinatorial $k$-complexes and combinatorial maps between combinatorial $k$-complexes, one defines a combinatorial $(k+1)$-complex to be one whose $k$-skeleton is a combinatorial complex and whose attaching maps $\phi_e$ of $(k+1)$-cells $e$ are combinatorial maps for suitable combinatorial structures on $\partial D_e^{k+1} = S^k$. 
	\end{definition}
	
	\begin{definition}\label{def:CayCx}
		(Cayley complex relative to a pair, homological finite presentation)
		Let $H=\la \A\,|\, R\ra$ be a group presentation with $\A$ finite and let $R_0\subset R$. 
		The \emph{Cayley complex relative to a pair $(\A, R_0)$} is a combinatorial $2$-complex $X(\A,R_0)$ obtained from the Cayley graph $\Gamma(H,\A)$ in the following way. For each word $r\in R_0$ and each vertex $v$ of $\Gamma(H,\A)$, we attach a disk $D_r^v$ to the path in $\Gamma(H,\A)$ labeled with the word $r$ which starts at $v$. We set 
		\[
		X(\A,R_0)=\Gamma(H,\A)\cup\textstyle\bigcup_{r\in R_0,v\in H} D_r^v.
		\]
		We extend the action of $H$ over $X(\A,R_0)$ by $g\cdot D_r^v = D_r^{g\cdot v}$. (Note that if a relation is a proper power, i.e. $r = w^n$ with $n>1$, then there will be $n$ disks at $v$ with the same boundary label $r$. These disks are $D_r^{w^k\cdot v}$ for $0\leq k < n$.) This action of $H$ on $X(\A,R_0)$ is free.
		
		A \emph{homological finite presentation} for $H$ is a pair $\la \A \left|\right| R_0\ra$ such that $R_0$ is a finite subset of $R$ and $H_1(X(\A,R_0))=0$. 
		We call the $2$-complex $X(\A,R_0)$ in this case {\em the homological Cayley complex} for $\la \A\,||\, R_0\ra$. 
	\end{definition}
	
	\begin{proposition}\label{prop:freeimpliesCayley}
		Let $X$ be a connected $2$-dimensional combinatorial complex such that $H_1(X) = 0$. 
		Suppose that $H$ acts on $X$ freely, cellularly, cocompactly and with one orbit of vertices. 
		Then there is a homological finite presentation $\la \A \left|\right| R_0\ra$ for $H$ such that $X$ is the homological Cayley complex. 
	\end{proposition}
	\begin{proof}
		Since $H$ acts freely with a single orbit of vertices, it follows that $X^{(1)}$ is a Cayley graph for $H$ with respect to some generating set $\A$. This generating set is in one-to-one correspondence with orbits of edges for $H$ acting on $X^{(1)}$. 
		
		Since the action is cocompact there are only finitely many orbits of $2$-cells $[D_1], \dots, [D_n]$. 
		Moreover, each orbit of $2$-cells contains a $2$-cell based at the identity, we will assume this is $D_i$. 
		Let $r_i$ be the attaching map of $D_i$ read as a word in $\A$. 
		
		Now consider the Cayley complex relative to $\mathcal{P} = \la \A \left|\right| r_1, \dots, r_n\ra$. 
		It follows from the choices of $\A$ and $r_i$ that this relative Cayley complex is exactly $X$. 
		
		Since $H_1(X)$ was assumed trivial, we see that $X$ the homological Cayley complex for $\la \A \left|\right| R_0\ra$.
	\end{proof}
	
	The following relates the existence of a finite homological presentation to being of type~$FP_2$. 
	
	\begin{proposition}\label{prop:equiv}
		Let $H=\la \A\,|\,R\ra$ with $\A$ finite, and let  $N=\la\!\la R\ra\!\ra\le F(\A)$. The following are equivalent:
		\begin{enumerate}
			\item $H$ is of type $FP_2$;
			\item The relation module $N^{ab}$ is finitely generated as a $\Z H$-module;
			\item There exist a homological finite presentation $\la \A \left|\right| R_0\ra$ for $H$.
		\end{enumerate}
	\end{proposition}
	\begin{proof}
		The equivalence of the first two conditions is well known, see~\cite[VIII.5,\,Exerc.\,3b]{Bro} and also~\cite[\textsection2,\,Prop.\,3]{Tw}. 
		
		For the equivalence of the latter two conditions, note that the relation module is also isomorphic to $H_1(\Gamma(H,\A))$ where $\Gamma(H,\A)$ is the Cayley graph of $H$ with respect to the generating set $\A$.
		Thus, if it is finitely generated as a $\Z H$-module, then picking a finite generating set of loops gives us a finite homological presentation. 
		
		Conversely, if we have a homological finite  presentation, then the set $R_0$ gives a finite generating set for $H_1(\Gamma(H,\A))$.
	\end{proof}
	
	\subsection{Dehn functions and homological Dehn functions}\label{dehn}
	We will consider Dehn functions (homological or otherwise) up to the following equivalence relation.
	
	\begin{definition} ($\simeq$ equivalence)
		Let $f,g\colon[0,\infty)\to[0,\infty)$ be two functions, we say $f\preceq g$ if there exist constants $A,B>0$ and $C,D,E\ge0$ such that $f(n)\leq Ag(Bn+C)+Dn+E$ for all $n\geq0$.
		We say $f\simeq g$ if $f\prq g$ and $g\prq f$, where $f\prq g$. We extend this equivalence relation to functions $\N\to[0,\infty)$ by assuming them to be constant on each interval $[n,n+1)$. 
	\end{definition} 
	
	We are interested in  disk fillings of combinatorial loops in  combinatorial complexes. 
	
	\begin{definition}[Combinatorial loop] Let $X$ be a combinatorial complex. A \emph{combinatorial loop in $X$} consists of a combinatorial map $\gamma\colon S^1 \to X$ for some combinatorial structure on $S^1$.  Note that the domain $S^1$ may be replaced by an interval $[a,b]$ with the requirement that $\gamma(a) = \gamma(b)$. The \emph{length} of the combinatorial loop $\gamma$ is denoted by $|\gamma|$ and is defined to be the number of $1$-cells in the combinatorial structure on the domain $S^1$. 
	\end{definition}
	
	\begin{definition}[Disk diagram]
		A \emph{disk diagram} $\Delta$ is a finite combinatorial $2$-complex which is defined as in Section~1.1 of Part~II of \cite{BRS} to be  $\Delta = S^2\smallsetminus e_\infty$ for some open $2$-cell $e_\infty$ in a combinatorial cell structure on the $2$-sphere $S^2$. One can think of a  disk diagram $\Delta$ as living in the plane; in this case $e_\infty$ represents the exterior region of $\Delta$ together with the point at infinity. We call the attaching map $\phi_{e_\infty}: \partial D^2_{e_\infty} \to \Delta^{(1)} \hookrightarrow \Delta$ the \emph{boundary circle} of $\Delta$. 
	\end{definition}
	
	\begin{definition}[Disk filling of a combinatorial loop]
		Let $X$ be a combinatorial complex and $\gamma\colon S^1 \to X$ be a combinatorial loop.   A \emph{disk filling} of $\gamma$ consists of a  disk diagram $\Delta$ and a combinatorial map $\pi: \Delta \to X$ such that $\gamma$ factors through the attaching map of $e_\infty$; that is, $\gamma = \pi \circ \phi_{e_\infty}$.  	\end{definition}
	
	We  use annuli to realise free homotopies between combinatorial loops. 
	
	\begin{definition}[Annular diagram]
		In an analogous fashion, an \emph{annular diagram} $A$ is defined be a finite combinatorial $2$-complex   $A = S^2\smallsetminus(e_\infty \cup e_0)$ for some open $2$-cells $e_\infty$ and $e_0$ in a combinatorial cell structure on $S^2$.  As is the case with  disk diagrams, we refer to  the attaching maps $\phi_{e_\infty}\colon \partial D^2_{e_\infty} \to A^{(1)} \hookrightarrow A$  and  $\phi_{e_0}\colon \partial D^2_{e_0} \to A^{(1)} \hookrightarrow A$ as \emph{boundary circles} of~$A$.
	\end{definition}
	
	\begin{definition}[Annular diagram realizing a free homotopy of loops]
		Let $X$ be a combinatorial complex and  let $\gamma\colon S^1 \to X$ and $\alpha\colon S^1 \to X$  be two combinatorial loops. We say that the  annular diagram $A$ \emph{realises a free homotopy between $\gamma$ and $\alpha$} if  there is a combinatorial map $\pi\colon A \to X$ such that $\gamma$ (resp.\ $\alpha$) factors through the attaching map of $e_\infty$ (resp.\ $e_0$); that is, $\gamma = \pi \circ \phi_{e_\infty}$ and $\alpha = \pi \circ \phi_{e_0}$.  
	\end{definition}
	
	When we use the pushing fillings machinery of \cite{ABDDY}, we need bounds on the number of $1$-cells in links of vertices in disk fillings of combinatorial loops. 
	
	\begin{definition}[Size of link in a disk or annular diagram]
		Note that a  disk diagram $\Delta$ is obtained from its $1$-skeleton by attaching  (a finite number of)  $2$-cells $e$ via combinatorial attaching maps $\phi_{e}: S^1_e =\partial D^2_e \to \Delta^{(1)}$. Corresponding to  each vertex $w \in \partial D^2_{e}$ the corner of the $2$-cell $e$ at $w$ contributes a segment to the link of the vertex $\phi_e(w) \in \Delta$.   We define the \emph{size of the link} of a vertex $v \in \Delta^{(0)}$ (denoted by $|\Lk(v, \Delta)|$) to be the total number of such $1$-cells:
		\[
		|\Lk(v, \Delta)| \; =\; |\{w \in (\partial D_e^2)^{(0)} \, | \, e {\hbox{ is a $2$-cell of $\Delta$, }} \phi_{e}(w) = v\}| .
		\]
		Similarly, one defines the size, $|\Lk(v, A)|$, of the link $\Lk(v, A)$ of a vertex in an annular diagram $A$. 
	\end{definition}
	
	\begin{definition} \label{def:area}
		(Area, filling function, Dehn function)
		Let $X$ be a simply connected combinatorial complex. 
		For each loop $\gamma$ in $X^{(1)}$ there exist a  disk diagram $\Delta$ with boundary $\gamma$. 
		Given such a  disk diagram let $|\Delta|$ denote the number of $2$-cells in the diagram. 
		Then the \emph{area} of $\gamma$ is 
		\[
		\Area_X(\gamma)\vcentcolon=\min\big\{\text{$|\Delta|\mid \Delta$ is a  disk diagram with boundary $\gamma$}\big\}.
		\]
		We define the \emph{filling function} of $X$ as
		\[
		\delta_X(n)\vcentcolon=\sup\big\{\Area_X(\gamma)\mid \gamma\text{ is a  loop in }X^{(1)}\text{ of length $\leq n$}\big\},
		\]
		
		We define the {\em Dehn function} $\delta_H$ of a finitely presented group $H$ to be the filling function of the universal cover of a finite presentation $2$-complex. 
	\end{definition}
	
	If $H$ is a finitely presented group, then the Dehn functions for two finite presentation complexes are the same up to $\simeq$ equivalence,  see~\cite[1.3.3]{BriChap}.
	Thus, for finitely presented groups the above definition gives a well-defined group invariant. 
	
	If one were to take an infinite presentation complex, then a similar function could be defined. 
	However, it is no longer invariant of the group $H$. 
	
	For groups of type $FP_2$, one can use homological notions to define filling functions. 
	There are two competing notions well studied in the literature. 
	There are the abelianised Dehn function $\delta_H^{ab}$ and the homological Dehn function $\FA_H$. 
	We define these notions for spaces with $H_1$ trivial. 
	Throughout, we will focus our interests on the homological filling function as the superadditivity property (\Cref{prop:superadd}) makes certain computations easier. 
	We leave it as an exercise to the reader to verify many of the key results of this section for the abelianised filling function.
	
	\begin{definition}\label{def:homarea}
		(Homological area, Homological filling function, Size) Let $X$ be a combinatorial complex with $H_1(X,\Z)=0$. For each $1$-cycle $\gamma$ in $X^{(1)}$ there exist a $2$-chain $c=\sum_i a_i\sigma_i$ in $X$, where $a_i\in\Z$, and $\sigma_i$ are $2$-cells, such that $\gamma=\partial c$. Then the \emph{homological area} of $\gamma$ is 
		\[
		\HArea_X(\gamma)\vcentcolon=\min\big\{\textstyle\sum_i|a_i|\mid\gamma=\partial c,\, c=\textstyle\sum_i a_i\sigma_i\big\}.
		\]
		We define the \emph{homological filling function} of $X$ as
		\[
		\FA_X(n)\vcentcolon=\sup\big\{\HArea_X(\gamma)\mid\gamma\text{ is a  $1$-cycle in }X^{(1)}\text{ with }|\gamma|\le n\big\},
		\]
		where $|\gamma|$ is the \emph{size} of $\gamma$, given by the $\ell^1$--norm on the cellular chain group $C_1(X)$. 
		
		This quantity can be realised geometrically as follows. Since $\gamma$ is a $1$-cycle it can be represented as a cellular map $\sqcup_i S^1\to X$. Then $|\gamma|$ is the minimal number of $1$-cells of $\sqcup_i S^1$ across all representation $\sqcup_i S^1\to X$ of $\gamma$. 
		
		We define the \emph{abelianised filling function} of $X$ as
		\[
		\delta_X^{ab}(n)\vcentcolon=\sup\big\{\HArea_X(\gamma)\mid \gamma\text{ is a  loop in }X^{(1)}\text{ with }|\gamma|\le n\big\}.
		\]
	\end{definition}

	\begin{lemma}\label{prop:compact}
		Let $X$ be a compact combinatorial complex with $H_1(X) = 0$, then ${\FA_X(n)\simeq n}$.
	\end{lemma}
	\begin{proof}
		Up to equivalence every function is bounded below by a linear function, thus we focus on obtaining the upper bound. 
		
		Let $k$ be the number of vertices in $X$. 
		We claim that $w = \sum_i w_i$ where $w_i$ are loops with $|w_i| \leq k$ and $|w| = \sum_i|w_i|$.
		As $w$ is a $1$-cycle it can be written as a sum of loops $v_i$ and $|w| = \sum_i|v_i|$. 
		If each $v_i$ satisfies $|v_i|\leq k$, then we can take $w_i = v_i$. 
		Thus we can reduce to the case that $w$ is a loop, i.e.\ an image of $S^1$ and $|w|\geq k$.
		
		Pick a minimal cellular representative $S^1\to X$ for $w$. 
		Since $|w|>k$ we see that $f$ is not an embedding. 
		Thus there are two vertices of $S^1$ which map to the same vertex of $X$. 
		Let $x_1, x_2$ be two such vertices of $S^1$. 
		Let $u_1, u_2$ be the closures of the two intervals of $S^1\smallsetminus\{x_1, x_2\}$. 
		Then $v_i = f(u_i)$ are loops such that $w = v_1 + v_2$.
		The number of edges remains unchanged in this process so we see that $|w| = |v_1| + |v_2|$.
		We can repeat this process with the $v_i$, until we arrive at loops $w_i$ such that $|w_i|\leq k$.
		By construction $w = \sum_i w_i$ and there is no cancellation in this sum. 
		
		To complete the proof, write $w$ as a sum of loops $w_i$ such that $|w_i|\leq k$ and $|w| = \sum_i |w_i|$.
		Note that this means $|w|\geq \sum_i 1$.
		We can now compute $\HArea(w)$ using the following inequalities. 
		\[
		\HArea(w)\leq \textstyle\sum_i\HArea(w_i)\leq \sum_i \FA_X(|w_i|)\leq \sum_i\FA_X(k)\leq |w|\FA_X(k).
		\]
		Hence $w$ has a filling of size bounded by $|w|\FA_X(k)$ and hence $\FA_X(n)\preceq n$. 
	\end{proof}

	\begin{definition}(Superadditive)
		We say that a function $f\colon\N\to[0,\infty)$ is {\em superadditive} if $f(m+n) \geq f(m) + f(n)$ for all $m,n\ge0$. If $h\colon \N\to[0,\infty)$ is any function, the smallest superadditive function that is greater than or equal to $h$ is called the \emph{superadditive closure} of $h$ and is denoted $\bar h$.
	\end{definition}
	
	The following lemma is an easy exercise.
	
	\begin{lemma}\label{lem:closure}
		For an arbitrary function $f\colon\N\to[0,\infty)$ its superadditive closure equals
		\[
		\bar f(n)=\max\big\{f(n_1)+f(n_2)+\dots+f(n_r)\big\},
		\]
		where $\max$ is taken for all $r\ge1$ and all $n_1,\dots,n_r\in\N$ such that $n_1+\dots+n_r=n$.\qed
	\end{lemma}

	\begin{proposition}\label{prop:superaddforspaces}
		Let $X$ be a connected combinatorial complex which is locally finite and admits a cellular cocompact group action.
		Suppose that $H_1(X) = 0$.  
		Then $\FA_X$ is $\simeq$ equivalent to a superadditive function. 
	\end{proposition}
	\begin{proof}
		Let $H$ be the group acting cocompactly on $X$. 
		If $X$ is compact, then the filling function is linear by \Cref{prop:compact} and hence superadditive.

		We will use the following metric on the $1$-skeleton of $X$. 
		Given $v, w\in X^{(1)}$, define $d(v, w)$ to be the length of the shortest path in $X^{(1)}$ from $v$ to $w$. 
		Since $H$ acts cellularly, we see that $d(g\cdot v, g\cdot w) = d(v, w)$ for all vertices $v, w\in X$ and all $g\in H$. 
		
		Now suppose $X$ is not compact.
		Since $X$ is locally finite, we see that $X^{(1)}$ is unbounded. 
		We will show that given $1$-cycles $w_1$ and $w_2$, we can find a $g\in H$ such that the following hold:
		\begin{enumerate}
			\item $|w_1 + g\cdot w_2| = |w_1| + |w_2|$; 
			\item $\HArea(w_1 + g\cdot w_2)\geq \HArea(w_1) + \HArea(w_2)$.
		\end{enumerate}
		Now taking the supremum over all $1$-cycles $w_1$ such that $|w_1|\leq m$ and $w_2$ such that $|w_2|\leq n$ we obtain the result. 
		
		Let $v$ be a base vertex in the space $X$. 
		We will prove the above properties in two claims. 
		Since $X$ is unbounded and $H$ acts cocompactly for any fixed $R$ there is an element $g\in H$ such that $d(v, g\cdot v)\geq R$. 
		
		\textit{Claim~1. There exists $g\in H$ such that $|w_1+g\cdot w_2|=|w_1|+|w_2|$.} 
		To see this let $R>0$ be the radius of the ball $B(v,R)$ in $X^{(1)}$ such that both $1$-cycles $w_1$ and $w_2$ lie entirely in $B(v,R)$. 
		By the above we can find $g\in H$ such that $d(v, g\cdot v)> 2R$. 
		In this case, the $1$-cycles $w_1$ and $g\cdot w_2$ are disjoint in $X$. 
		Indeed, for any $x\in w_1$ and any $y\in g\cdot w_2$ the triangle inequality yields:
		$d(v,g\cdot v)\le d(v,x)+d(x,y)+d(y,g\cdot v)$. 
		Notice that $y=g\cdot x'$ for some $x'\in w_2$, and hence $d(y,g\cdot v)=d(g\cdot x',g\cdot v)=d(x',v)\le R$, which implies
		$d(x,y)\ge d(v,g\cdot v)-d(v,x)-d(x',v)> 2R-R-R=0$. 
		This means that $w_1$ and $g\cdot w_2$ are disjoint $1$-cycles, and hence the size of $w_1+g\cdot w_2$ is the sum of the sizes of $w_1$ and $w_2$.
		
		\textit{Claim~2. There exists $g\in H$ such that the homological area of $w_1 + g\cdot w_2$ is $\HArea(w_1) + \HArea(w_2)$.} 
		
		It is clear that there is a filling for $w_1+g\cdot w_2$ of size $\HArea(w_1) + \HArea(w_2)$. 
		Suppose that there were a more optimal filling $c = \sum_i a_i\sigma_i$.
		We say that $c'$ is a {\em subfilling of $c$} if $c' = \sum_i a_i'\sigma_i$ where $0\leq |a_i'|\leq |a_i|$ and $a_i, a_i'$ have the same sign. The following two cases are possible.
		
		Suppose that there is a subfilling $c' = \sum_i a_i'\sigma_i$ of $c$ such that $\partial c' = w_1$. 
		Let $c'' = c-c' = \sum_i a_i''\sigma_i$, this is a subfilling of $c$ with $\partial c'' = g\cdot w_2$. Then $|c'|\ge\HArea(w_1)$ and $|c''|\ge\HArea(g\cdot w_2)$, and we have:
		$|c|=\sum_i |a_i| = \sum_i|a_i'|+|a_i''| \geq \HArea(w_1) + \HArea(w_2)$.
		
		If there is no subfilling of $c$ with the boundary $w_1$, then we can find a sequence of $2$-cells $\sigma_1, \dots, \sigma_k$ belonging to $c$ such that $\partial \sigma_1$ contains an edge of $w_1$, $\partial\sigma_k$ contains an edge of $g\cdot w_2$ and $\sigma_{i}\cap \sigma_{i+1}$ is non-empty for all $i=1,\dots,k-1$. However, in this case we see that there is a vertex $x$ of $w_1$ and a vertex $y$ of $g\cdot w_2$ that are at a distance $d(x,y)\leq k\cdot L/2$, where $L$ is the largest length of a relator of $R_0$. 
		Arguing as in the proof of Claim~1 above, we see that 
		$d(x,y)\ge d(v,g\cdot v)-2R$, where $R$ is the radius of a ball centered at $v$ containing both $w_1$ and $w_2$. This gives us an inequality on $k$, and hence a lower bound on the size of $c$:
		\[
		|c|\ge k\ge \big(d(v,g\cdot v)-2R\big)\cdot 2/L.
		\]
		Thus if $d(v,g\cdot v)$ is large enough, any filling $c$ of $w_1+g\cdot w_2$ with the above property will have at least $\HArea(w_1) + \HArea(w_2)$ $2$-cells. This proves Claim~2.
	\end{proof}

In the following proposition $\bar \delta^{ab}_X$ denotes the superadditive closure of the abelianised Dehn function $\delta_X^{ab}$.

	\begin{proposition}\label{prop:loops}
		Let $X$ be a combinatorial complex which is locally finite with $H_1(X) = 0$. 
		Suppose that $X$ admits a cellular cocompact group action. 
		Then $\FA_X\simeq\bar\delta_X^{ab}$.
	\end{proposition}
	\begin{proof}
		By \Cref{prop:superaddforspaces}, we can assume up to equivalence that $\FA_X$ is superadditive and hence $\overline{\FA}_X \simeq \FA_X$. 
		From the definition it is clear that $\delta_X^{ab}\preceq \FA_X$. 
		Thus by taking superadditive closures we see $\bar{\delta}_X^{ab}\preceq \FA_X$.
		
		Now suppose that $w$ is a $1$-cycle, with $w = \sum_i w_i$, where $w_i$ are edge loops in $X$. 
		We may assume that there is no cancellation of $1$-cells in the sum $\sum_i w_i$, hence $|w|=\sum_i|w_i|$.
		Then we have the following inequalities: 
		\[
		\HArea_X(w) \leq \textstyle\sum_i \HArea_X(w_i)\leq \sum_i \delta_X^{ab}(|w_i|)\leq \bar{\delta}_X^{ab}(\sum_i|w_i|) = \bar{\delta}_X^{ab}(|w|).
		\]
		This implies that $\FA_X(n)\le\bar{\delta}_X^{ab}(n)$, which means that the two functions are equivalent.
	\end{proof}		

An analog of the following proposition was proved in~\cite[Th.\,1.1]{FM18} with the local finiteness condition replaced by the requirement that the action of the group is proper (and cocompact). There is also an instructive example there~\cite[Ex.\,4.3]{FM18} which shows that if neither of these conditions is satisfied, the conclusion of the proposition does not hold.
	
	\begin{proposition}\label{prop:attain}
		Let $X$ be a connected combinatorial complex which is locally finite with $H_1(X)=0$ which admits a cellular cocompact group action.
		Then the supremum in the definition of $\FA_X(n)$ is attained.
	\end{proposition}
	\begin{proof}
		Let $k$ be the maximum valence of a vertex of $X$, note that $k$ is bounded since there are only finitely many orbits of vertices under the group action. 
		Then there are at most $k^n$ paths of length $n$ based at any fixed vertex $v$ of $X$. (Indeed, given a path $\lambda$ of length $n-1$ based at $v$, we have at most $k$ choices of an edge to extend $\lambda$ to a path of length $n$.)
		In particular, there are only finitely many orbits of loops of a given length. 
		Thus we see that the supremum in the definition of $\delta_X^{ab}$ is attained. 
		It follows from \Cref{prop:loops} that to obtain $\FA_X$ it suffices to compute $\delta_X^{ab}$ and take the superadditive closure. 
		By~\Cref{lem:closure}, there exists a partition $n=n_1+n_2+\dots+n_r$ such that 
		\[
		\FA_X(n) = \delta_X^{ab}(n_1)+\delta_X^{ab}(n_2)+\dots+\delta_X^{ab}(n_r).
		\]
		Hence by finding loops that realise $\delta_X^{ab}(n_i)$ for each $i$ and arguing as in the proof of \Cref{prop:superaddforspaces}, we obtain a $1$-cycle $w$ of length $\leq n$ such that 
		\[
		\HArea(w) =\delta_X^{ab}(n_1)+\delta_X^{ab}(n_2)+\dots+\delta_X^{ab}(n_r) = \FA_X(n).\qedhere
		\]
	\end{proof}
	
	\begin{definition}\label{def:homdehn} (Homological Dehn function, abelianised Dehn function)
		Let $H$ be a group of type $FP_2$.
		Consider a homological finite presentation $\la \A \left|\right| R_0\ra$ for $H$ existing by  \Cref{prop:equiv}.
		Let $X(\A,R_0)$ be the Cayley complex relative to $\la \A \left|\right| R_0\ra$.
		The \emph{homological Dehn function} of the triple $\la H,\A,R_0\ra$ is defined as the homological filling function of the complex $X(\A,R_0)$: $\FA_{\la H,\A,R_0\ra}(n)\vcentcolon=\FA_{X(\A,R_0)}(n)$.
		
		The \emph{abelianised Dehn function} of the triple $\la H,\A,R_0\ra$ is defined as the abelianised filling function of the complex $X(\A,R_0)$: $\delta^{ab}_{\la H,\A,R_0\ra}(n)\vcentcolon=\delta^{ab}_{X(\A,R_0)}(n)$.
	\end{definition}
	
	In \cite{BMS}, the abelianised area of a word $w$ is defined as the minimal $n$ such that $w = \prod_{i=1}^n v_ir_i^{\pm 1}v_i^{-1}$, in the relation module, where the $r_i$ are taken from a generating set for the relation module.
	Since in~\cite{BMS} only finite presentations are considered, the generating set comes from relations of a finite presentation. 
	Since we are considering groups of type $FP_2$, we will pick our finite generating set to come from a homological finite presentation $\langle \A\left|\right| R_0\rangle$.  
	Thus, we have a surjection from a free $\Z H$-module $F$ with basis $R_0$ to the relation module and the abelianised area is the length of a minimal representative in $F$ considered as a free $\Z$-module.
	However, we can view $F$ as the $2$-chains in the homological Cayley complex for $H$. 
	Thus, we see that our definition coincides exactly with that of \cite{BMS}.
	
	As previously stated, we will focus our attention on the homological Dehn function. 
	In \Cref{prop:loops}, we give a close relationship between these two functions. 
	
	Propositions~\ref{prop:indeppres} and \ref{prop:spaces} below can also be obtained from Theorem~4.4 and Proposition~4.10 of~\cite{MP}, where the relative homological Dehn functions are studied. We include them here for completeness and ease of the reader. 
	
	\begin{proposition}\label{prop:indeppres}
		Up to $\simeq$ equivalence, the homological Dehn function is independent of a homological finite presentation for $H$. 
	\end{proposition}
	\begin{proof}
		We begin with the case that the homological finite presentations are on the same set of generators.
		Let $P_0 = \la \A \left|\right| R_0\ra$ and $P_1 = \la \A\left|\right| R_1\ra$ be homological finite presentations on the same generating set $\A$. 
		Let $\gamma$ be a $1$-cycle in the Cayley graph of $H$. 
		In the homological Cayley complex for $P_0$ let $\sum_i a_i\sigma_i$ be a filling for $\gamma$. Recall that $\sigma_i=g_i.r_i$ for some $g_i\in H$ and some relator $r_i\in R_0$ (viewed as a $2$-cell). Thus the boundary $\partial \sigma_i$ is a loop, hence a $1$-cycle, in the Cayley graph for $H$, corresponding to one of the relators in $R_0$. 
		Let $\beta_i$ be a filling for $\partial \sigma_i$ in the homological Cayley complex for $P_1$. 
		Note that if $\sigma_i$ and $\sigma_j$ are two translates of the same relator $r\in R_0$, the homological areas with respect to $X(\A,R_0)$ of $\beta_i=\partial\sigma_i$ and $\beta_j=\partial\sigma_j$ are the same.
		
		We can obtain a filling for $\gamma$, in the homological Cayley complex for $P_1$, of the form $\sum_i a_i\beta_i$. 
		Thus the filling in the homological Cayley complex for $P_1$ is bounded by 
		\[
		\max_i \{\HArea_{X(\A,R_1)}(\partial\beta_i)\}\cdot \HArea_{X(\A,R_0)}(\gamma).
		\]
		Thus we obtain equivalent filling functions if the homological finite presentations have the same generating set. 
		
		Now we must consider the general case. 
		Suppose that $\la \A\left|\right| R_0\ra$, $\la \A_1\left|\right| S_0\ra$ are homological finite presentations for $H$.
		Since $\A$ is a generating set, we have that for each $b\in\A_1$ there exists a word $v(\A)\in F(\A)$ such that $b=v(\A)$ in $H$.
		Let $R_1$ be the set of relations of the form $b=v(\A)$ for each $b\in\A_1$. We are going to construct another complex $X_1$ which is obtained from $X(\A, R_0)$ by adding one orbit of new edges for each element of $\A_1$ and one orbit of new $2$-cells for each relation in $R_1$. More specifically, for each $b\in\A_1$ and every vertex $x\in X(\A,R_0)^{(0)}=H$ we consider a new $1$-cell $e_b^x$ and a new $2$-cell $D_b^x$. Let $b=v(\A)$ be a relator in $R_1$. We attach $e_b^x$ to $X$ with initial and terminal vertices $x$ and $x\cdot v(\A)$ respectively, and label it with letter $b\in\A_1$. And we attach the $2$-cell $D_b^x$ to the closed path 	which reads $b^{-1}v(\A)$. Hence $X_1$ is set  to 
		\[
		X_1\vcentcolon=X(\A,R_0)\cup\textstyle\bigcup_{x\in H,b\in\A_1}e_b^x\cup\textstyle\bigcup_{x\in H,b\in\A_1} D_b^x.
		\]
		We extend the action of $H$ over $X_1$ by declaring $g\cdot e_b^x=e_{b}^{g\cdot x}$ and $g\cdot D_b^x=D_b^{g\cdot x}$ to obtain a free vertex transitive and cocompact action of $H$ on $X_1$.
		We claim that $H_1(X_1) = 0$. 
		Indeed, there exists a retraction $r\colon X_1\to X(\A, R_0)$, given as follows. 
		The retraction $r$ is the identity on $X(\A, R_0)$. 
		Each edge not in $X(\A, R_0)$ has the form $e_b^x$ for some $b\in \A_1$.
		We send this edge to the word labeled $v(\A)$ with the same endpoints where $b = v(\A)$ is a relation in $R_1$. 
		We can extend $r$ over the $2$-cells $D_b^x$ since their boundary now corresponds to the null-homotopic word $v(\A)^{-1}v(\A)$ in $X(\A,R_0)$. 
		
		The retraction $r$ is a deformation retraction. 
		To see this, note that every disk $D_b^x$ in $R_1$ has exactly one edge $e_b^x$ not in $X(\A, R_0)$ and this edge appears only on the boundary of $D_b^x$. 
		Thus we can realise the retraction $r$ by pushing the disks $D_b^x$ in from the edges $e_b^x$. 
		
		We conclude that $\la \A \sqcup \A_1\left|\right| R_0\sqcup R_1\ra$ is a homological finite presentation for $H$. 
		It is also clear by construction that $X_1$ is the homological Cayley complex for $\la \A \sqcup \A_1\left|\right| R_0\sqcup R_1\ra$.
		
		Let $w$ be a $1$-cycle in $X_1$. 
		For every edge in $w$ which is labeled by an element of $\A_1$, we can use a relator from $R_1$ to remove this edge. 
		Thus after applying $\leq |w|$ relators from $R_1$ we obtain a $1$-cycle $w'$ only using edges labeled by elements of $\A$. 
		Let $k$ be the maximum length of a word $v(\A)$ among all relators $b=v(\A)$ in $R_1$. 
		Then $|w'|\leq k|w|$. 
		
		Now we can fill $|w'|$ with a $2$-chain consisting of relators in $R_0$. 
		Thus $\HArea_{X_1}(w')\leq \FA_{X(\A, R_0)}(|w'|)$.
		And $\HArea_{X_1}(w)\leq \FA_{X(\A, R_0)}(k|w|) + |w|$. 
		Thus $\FA_{X_1}(n) \leq \FA_{X(\A, R_0)}(kn) + n$.
		We conclude that $\FA_{X_1}\preceq \FA_{X(\A, R_0)}$. 
		
		We must now obtain a similar lower bound for $\FA_{X_1}$. 
		Let $n\in \N$. 
		Let $w$ be a $1$-cycle in $X(\A, R_0)$ consisting entirely of edges labeled by $\A$ such that $|w|\leq n$ and $\HArea_{X(\A, R_0)}(w) = \FA_{X(\A, R_0)}(n)$, such a $1$-cycle exists by \Cref{prop:attain}.
		Let $c$ be a filling for $w$ in $X_1$. 
		Under the retraction $r$, $w$ is fixed and $c$ is sent to a filling for $w$ in $X(\A, R_0)$. 
		Since the minimal filling for $w$ in $X(\A, R_0)$ has size $\FA_{X(\A, R_0)}(n)$, we conclude that the image of $c$ must have size $\geq \FA_{X(\A, R_0)}(n)$. 
		Since the retraction $r$ can only decrease the size of $c$, we obtain that $\HArea_{X_1}(w)\geq \FA_{X(\A, R_0)}(n)$. This means that for each $n$ we found a $1$-cycle $w$ with $|w|\leq n$ on which the function $\HArea_{X_1}$ is at least as big as $\FA_{X(\A, R_0)}(n)$. We conclude that $\FA_{X_1} \succeq \FA_{X(\A, R_0)}$. 
		
		This proves that $\FA_{X_1}\simeq \FA_{X(\A, R_0)}$. 
		
		Now for every $a\in\A$ there exist a word $u(\A_1)\in F(\A_1)$ such that $a=u(\A_1)$ in $H$.
		Once again, we obtain a homological finite presentation $\la \A \sqcup \A_1\left|\right| S_0\sqcup S_1\ra$, where $S_1$ is the set of relations of the form 
		$a=u(\A_1)$ for each $a\in\A$.
		Let $Y_1$ be the homological Cayley complex for $\la \A \sqcup \A_1\left|\right| S_0\sqcup S_1\ra$. 
		In the same way as above we see that $\FA_{Y_1}\simeq \FA_{X(\A_1, S_0)}$. 
		
		We now have two homological finite presentations with the same set of generators. 
		Thus applying the first half of the proof we obtain $\FA_{X_1}\simeq \FA_{Y_1}$. 
		Finally we conclude that $\FA_{X(\A, R_0)}\simeq\FA_{X_1}\simeq \FA_{Y_1}\simeq \FA_{X(\A_1, S_0)}$. 
		This completes the proof.
	\end{proof}

	Since $\FA_{\la H,\A,R_0\ra}(n)$ is independent of choices of $\A$, $R$ and a finite $R_0\subset R$, we denote $\FA_{\la H,\A,R_0\ra}(n)$ as $\FA_H(n)$ from now on.

	\begin{proposition}\label{prop:spaces}
		Let $X$ be a connected combinatorial complex with $H_1(X) = 0$. 
		Suppose that $X$ admits an action of a group $H$ which is free, proper, cellular and cocompact. 
		Then the homological filling function $\FA_X$ is $\simeq$ equivalent to the homological Dehn function $\FA_H$.
	\end{proposition}
	\begin{proof}
		Since the homological filling function is defined at the level of $2$-chains, we will work with $X^{(2)}$ and thus assume that $X$ is $2$-dimensional. 
		
		Let $W = X/H$. 
		Since $H$ is acting freely, properly, cellularly and cocompactly we see that $X$ is a covering space of $W$ and $W$ is a compact combinatorial complex.
		Let $T$ be a maximal tree in the $1$-skeleton $W^{(1)}$ of $W$.
		We can lift $T$ to an $H$-invariant forest $\widetilde T$ in $X$, note that each component is isomorphic to $T$. 
		
		By collapsing each connected component of $\widetilde T$, we obtain a space $Y$ which is a combinatorial complex with an action of $H$ on it. 
		Let $f\colon X\to Y$ be the collapsing map. 
		This action is free, cellular, cocompact and is transitive on vertices. 
		Thus, by \Cref{prop:freeimpliesCayley}, we see that $Y$ is a homological Cayley complex for some homological finite presentation for $H$. 
		Thus, the homological filling function of $Y$ is $\FA_H$. 
		
		We claim that the map $f$ induces a bijection $Z_1(X)\to Z_1(Y)$ between $1$-cycles of $X$ and $1$-cycles of $Y$. 
		To see that it is injective, notice that the map $Z_1(X)\to Z_1(Y)$ is the restriction of the map of $1$-chains $f_\sharp\colon C_1(X)\to C_1(Y)$. 
		The $1$-chain group $C_1(X)$ is the free abelian group on the edges of $X$, i.e.\  $C_1(X)=\Z[E(X)]$, which can be decomposed as a direct sum $\Z[E(X)]=\Z[E(\widetilde T)]\oplus \Z[E(X\smallsetminus \widetilde T)]$ of the free abelian group on edges of $X$ belonging to $\widetilde T$ and the free abelian group on edges of $X$ not in $\widetilde T$. The map $f_\sharp\colon C_1(X) \to C_1(Y)$ can then be viewed as the composition: $\Z[E(X)]\to \Z[E(X\smallsetminus \widetilde T)]\to \Z[E(Y)]$, where the first map is a projection and the second is an isomorphism. 
		Thus if $\gamma$ is an element of $Z_1(X)$ that maps to 0 in $Z_1(Y)$, we see that $\gamma$ lies in the kernel $\Z[E(\tilde T)]$ of the projection. 
		Thus, $\gamma$ is a $1$-cycle in the forest $\widetilde T$, however all such $1$-cycles are trivial and hence $\gamma = 0$. This shows that the map $Z_1(X)\to Z_1(Y)$ is injective. 
		
		To prove surjectivity, let $\gamma'$ be a $1$-cycle in $Y$. 
		We can construct a $1$-cycle $\gamma$ in $X$ which maps onto $\gamma'$, in the following way. 
		If $\gamma'=\sum a_i\gamma_i'$ for some loops $\gamma'_i$, we see that the preimage of each edge of $\gamma_i'$ is an edge in $X$ with endpoints belonging to $\widetilde T$. 
		Moreover, if two edges of $\gamma'_i$ are adjacent to the same vertex in $Y$, their preimages are adjacent to the same component of $\widetilde T$ in $X$. 
		This means that given a loop $\gamma'_i\subset Y$ we can form a loop $\gamma_i\subset X$ which maps onto $\gamma'_i$ by connecting preimages of edges of $\gamma'_i$ by paths lying entirely within $\widetilde T$. 
		Thus, $Z_1(X)\to Z_1(Y)$ is surjective. 
		
		Moreover, if $K$ is the number of edges in $T$, we see that $|\gamma_i|\le (K+1)|\gamma_i'|$. 
		We conclude that $|\gamma'|\leq |\gamma|\leq (K+1)|\gamma'|$. 
		
		The map $f$ also induces a bijection between $2$-cells of $X$ and $2$-cells of $Y$. 
		Given a $2$-cell $\sigma$ of $X$ we denote the corresponding $2$-cell of $Y$ by $\sigma'$. 
		
		Let $\gamma$ be a $1$-cycle in $X$ and $\gamma'$ its image in $Y$. 
		If $c$ is a filling for $\gamma$, then the image $c'$ of $c$ is a filling of $\gamma'$ containing the same number of $2$-cells as $c$. 
		Thus $\HArea_X(\gamma) \geq \HArea_Y(\gamma')$. 
		
		Let $d' = \sum_ia_i\sigma_i'$ be a filling for $\gamma'$.
		Then $d = \sum_ia_i\sigma_i$ is a $2$-chain in $X$. 
		We have the equality $f_\sharp(\partial d) = \partial f_\sharp(d) = \partial d' = \gamma'$.
		Since $f_\sharp$ gives a bijection $Z_1(X)\to Z_1(Y)$ we see that $\partial d = \gamma$. 
		Since $|d| = |d'|$ we see that $\HArea_X(\gamma) \leq \HArea_Y(\gamma')$. 
		Thus, we obtain the equality $\HArea_X(\gamma) = \HArea_Y(\gamma')$.
		
		We now obtain the following inequalities: 
		\begin{multline*}
		\FA_X(n) = \max\{\,\HArea_X(\gamma)\mid|\gamma|\leq n\,\} = \max\{\,\HArea_Y(f_\sharp(\gamma))\mid|\gamma|\leq n\,\} \\
		\leq \max\{\,\HArea_Y(\alpha)\mid |\alpha|\leq n\,\} = \FA_Y(n).
		\end{multline*}
		
		We also have the following inequalities: 
		\begin{multline*}
		\FA_Y(n) = \max\{\,\HArea_Y(\gamma')\mid |\gamma'|\leq n\,\} = \max\{\,\HArea_X(\gamma)\mid |f_\sharp(\gamma)|\leq n\,\} \\
		\leq  \max\{\,\HArea_X(\beta)\mid |\beta|\leq (K+1)n\,\}= \FA_X\bigl((K+1)n\bigr).
		\end{multline*}
		These two inequalities prove that $\FA_X\simeq\FA_Y=\FA_H$.
	\end{proof}

	The next proposition follows immediately from \Cref{prop:superaddforspaces}. 
	\begin{proposition}\label{prop:superadd}
		Every homological Dehn function is $\simeq$ equivalent to a superadditive function. \qed
	\end{proposition}

	The corresponding statement about Dehn functions is a conjecture of Guba and Sapir:
	\begin{conjecture}[\protect{\cite[Conjecture\,1]{GS}}]\label{conj:gs}
		The Dehn function of a finitely presented group is $\simeq$ equivalent to a superadditive function.
	\end{conjecture}
	The same question is open for the abelianised Dehn function. 
	Guba and Sapir say that if this conjecture is wrong, then we may be using a wrong definition of the Dehn function. 
	Notably, a modification of the Dehn function is used in~\cite{Ger}, one can show that the definition there is superadditive using the same proof as \Cref{prop:superaddforspaces}. 
	
	Since clearly we have the inequality $\delta_H^{ab}\preceq \delta_H$ 
	for finitely presented groups, we obtain the following result. 
	\begin{proposition}\label{prop:boundbyDehn}
		Let $H$ be a finitely presented group and let $\bar{\delta}_H$ be the superadditive closure of the Dehn function $\delta_H$ of $H$. 
		Then $\FA_H(n) \preceq \bar{\delta}_H(n)$.
	\end{proposition}
		\begin{proof}
			Recall, that by \Cref{lem:closure}, the superadditive closure of a function $f\colon \N\to [0,\infty)$ is given by $\bar{f}(n) = \max\{\sum_i f(n_i) \mid n_i
			\ge1,\sum_i n_i  = n\}$. 
			
			Since $H$ is finitely presented, we can assume that the homological Cayley complex $X$ is the universal cover of the presentation complex for $H$. 
			Let $\gamma$ be a loop in $X$. 
			Then from Definitions \ref{def:area} and \ref{def:homarea} we immediately see $\HArea_X(\gamma)\leq \Area_X(\gamma)$.

			Let $w$ be a $1$-cycle in $X$. 
			Pick a minimal representative $\sqcup_i S^1\to X$ for $w$.  
			We refer to the map on the $i$-th copy of $S^1$ by $w_i$. 
			By minimality, we see that $|w| = \sum_i |w_i|$. 
			
			We now have the following inequalities: 
			\[
			\HArea_X(w)
			\leq \textstyle\sum_i\HArea_X(w_i)
			\leq \textstyle\sum_i\DArea_X(w_i)
			\leq \textstyle\sum_i \delta(|w_i|)
			\leq \bar{\delta}\left(\textstyle\sum_i|w_i|\right) = \bar{\delta}(|w|),
			\]
			and by taking supremum, we get $\FA_H(n)\le \bar\delta_H(n)$ for an arbitrary $n$.
		\end{proof}
	
	We suspect that the above holds without using superadditive closures, namely, 
	\begin{conjecture}
		Let $H$ be a finitely presented group. 
		Then $\FA_H\preceq \delta_H$. 
	\end{conjecture}
	It is clear that this conjecture is implied by 
	\Cref{conj:gs}.
	It is not clear to us whether the two are equivalent. 
	
	\medskip
	In certain cases we can also use loops to obtain a lower bound on $\FA_H$. 
	This will be key to the proof of several lower bounds in the following sections. 
	
	\begin{proposition}\label{thm:embeddeddisk}
		Let $H$ be a finitely presented group. 
		Let $X$ be the universal cover of a presentation $2$-complex for $H$. Suppose that $X$ satisfies $H_2(X)=0$. 
		If $\gamma$ is a loop in $X$ which is the boundary of an embedded  diagram $\Delta$, then $|\Delta|\leq \FA_H(|\gamma|)$.
	\end{proposition}
	\begin{proof}
		The map $\Delta$ defines a $2$-chain in $C_2(X)$ with boundary $\gamma$. 
		Thus $\HArea_X(\gamma)\leq|\Delta|$.
		
		Let $c$ be another $2$-chain with boundary $\gamma$. 
		Then $\Delta - c$ is a $2$-cycle. 
		Since $H_2(X)=0$ we see that $\Delta-c$ is a boundary. 
		However, there are no $3$-cells in $X$ and so all the boundaries are trivial. 
		Hence $\Delta - c = 0$ and $c = \Delta$. 
		Thus $|\Delta| = \HArea(\gamma) \leq \FA_H(|\gamma|)$. 
	\end{proof}

	\medskip
	To state the next proposition, we need the notion of the maximum of two $\simeq$ equivalence classes of non-decreasing functions. The following lemma shows that it is well defined. Recall that, given two functions $f,g$, the function $\max(f,g)$ is defined for each $n$ as 
	\[
	\max(f,g)(n)\vcentcolon=\max\{f(n),g(n)\}. 
	\]
	
	The following lemma is an easy exercise.
	\begin{lemma}
		Let $f,f',g,g'$ be non-decreasing functions $\N\to[0,\infty)$. If $f\simeq f'$ and $g\simeq g'$, then $\max(f,g)\simeq\max(f',g')$.\qed 
	\end{lemma}
	%
	
	\begin{definition}
		(Retraction)
		Let $H, H'$ be groups.
		Let $\iota\colon H'\to H$ be an injective homomorphism.
		We say that $H'$ is a {\em retract} of $H$, if there is a homomorphism $\phi\colon H\to H'$ such that $\phi\circ \iota = \id_{H'}$. In this case $\phi$ is surjective and is called a \emph{retraction} of $H$ onto $H'$. 
	\end{definition}
	
	\begin{proposition}\label{prop:retract}
		Let $H$ be a group of type $FP_2$, and $H'\subseteq H$. 
		Suppose that there exists a retraction $\phi\colon H\to H'$.
		Then $\FA_{H'}\preceq \FA_H$. 
	\end{proposition}
	\begin{proof}
		By \cite{Alonso}, we have that if $H'$ is a retract of a group of type $FP_2$, then $H'$ is of type $FP_2$ itself, and hence $\FA_{H'}$ is defined. 
		
		The idea of the proof is to obtain homological Cayley complexes for $H$ and $H'$ and a cellular retraction between them. 
		We achieve this as follows. 
		We claim that we can choose generating sets $\A$ and $\A'$ for $H$ and $H'$, respectively, in such a way that $\A'\subset \A$ and $\A\smallsetminus \A'\subset \ker(\phi)$. Indeed, choose an arbitrary generating set $\A'\subset H'$. The group $H$ decomposes as a semidirect product $H=H'\ltimes\ker(\phi)$, hence $H$ is generated by $\A'$ and $\ker(\phi)$. Since $H$ is finitely generated, there exists a finite generating set $\A_0\subset H$, and for each $a\in\A_0$ there exists a word $w_a$ in generators from $\A'\cup\ker(\phi)$ such that $a=w_a$ in $H$. Let $\B$ be the set of elements of $\ker(\phi)$ which participate in words $w_a$ as $a$ ranges over $\A_0$. Denote $\A=\A'\cup\B$. We see that $\A$ is finite, generates $H$, contains $\A'$ and that $\A\smallsetminus\A'=\B\subset\ker(\phi)$, as desired.
		
		Let $\A$ and $\A'$ be as above and let  $\la \A\left|\right| R_0\ra$ and $\la \A'\left|\right| R_1 \ra$ be homological
		finite presentations for $H$ and $H'$, respectively. 
		We are going to modify these presentations as follows. 
		Let $S$ denote the set of all words in generators $\A'$ which give trivial elements of $H'$ and have length $\le\max\{|r|\mid r\in R_0\}$. We claim that 
		$\la \A'\left|\right| R_1\cup S\ra$ is a homological finite presentation for $H'$. Indeed, the homological Cayley complex for $\la \A'\left|\right| R_1\cup S\ra$ is obtained from the one for $\la \A'\left|\right| R_1\ra$ by adding $2$-cells corresponding to words from $S$ which are trivial in $H'$, and hence this operation does not change the first homology of the complex. Similarly, we claim that $\la \A\left|\right| R_0\cup R_1\cup S\ra$ is a homological finite presentation for $H$, for the same reason as above: relators $R_1\cup S$ are words in generators $\A'\subset\A$, which represent trivial element of $H'$, and hence are trivial in $H$ as well.
		
		Let $X$ be the homological Cayley complex for $\la \A\left|\right| R_0\cup R_1\cup S\ra$ and $X'$ be the homological Cayley complex for $\la \A'\left|\right| R_1\cup S \ra$. 
		Since $\A'\subset \A$ and $R_1\cup S\subset R_0\cup R_1\cup S$ we have a cellular inclusion $X'\to X$. 
		We also obtain a map $X\to X'$ as follows. Recall that the $1$-skeletons of $X$ and $X'$ are naturally identified with the Cayley graphs $\Gamma(H,\A)$ and $\Gamma(H',\A')$, respectively.
		Since $\phi$ is a homomorphism such that $\phi(a)=a$ for all $a\in\A'$ and $\phi(a)=1$ for all $a\in\A\smallsetminus \A'$, it gives rise to a map of the Cayley graphs $\rho\colon X^{(1)}\to (X')^{(1)}$ in the following way. Each vertex $g\in X^{(0)}$ gets sent to the corresponding vertex $\phi(g)\in(X')^{(0)}$. If $e$ is an edge  of $\Gamma(H,\A)$ having a pair of vertices $(g,ga)$ as endpoints, then we define $\rho$ to map the edge $e$ homeomorphically onto the edge between vertices $(\phi(g),\phi(g)a)$ of $\Gamma(H',\A')$ if $a\in\A'$, and to map $e$ to the single vertex $\phi(g)$ if $a\in\A\smallsetminus\A'$. 
		One can see that the effect of applying $\rho$ is contracting all edges labeled by elements of $\A\smallsetminus\A'$, and sending edges labeled by elements from $\A'$ onto edges with the same label. In particular, $\rho$ does not increase the lengths of paths.
		
		We can extend $\rho$ over $2$-cells to obtain a cellular map of $X\to X'$ in the following way. 
		Observe that each $2$-cell $\sigma$ of $X$ has boundary $\partial \sigma$ labeled by a word from either $R_0$, or $R_1\cup S$. 
		In the latter case, all letters of the attaching word belong to $\A'$, and hence $\rho$ maps $\partial\sigma$ to a loop in $X'$ with the same label, and there is a corresponding $2$-cell $\sigma'$ in $X'$ having the same boundary $\partial\sigma$. This allows us to define $\rho|_\sigma\colon\sigma\to\sigma'$ via the identity map. If $\partial \sigma$ reads an element of $R_0$, we see that $\rho(\partial\sigma)$ is a loop in $(X')^{(1)}$ whose length is less than or equal to the length of $\partial\sigma$. Since $S$ contains all trivial words of $H'$ of length $\leq \max\{|r|\mid r\in R_0\}$, we see that the label of $\rho(\partial\sigma)$  belongs to $S$. 
		Thus there is a $2$-cell $\sigma'$ in $X'$ with boundary label $\rho(\partial\sigma)$. 
		We can extend $\rho$ by mapping $\sigma$ to $\sigma'$. 
		Thus we obtain a map of homological Cayley complexes that is a retraction $X\to X'$. 
		
		Let $\gamma$ be a $1$-cycle in $X'$ and let $c$ be a filling for $\gamma$ in $X$. 
		Then $\rho(c)$ is a filling for $\gamma$ in $X'$ and since each $2$-cell of $X$ maps to a single $2$-cell of $X'$, we have that $|c| = |\rho(c)|$. 
		Hence $\HArea_{X'}(\gamma) \le \HArea_{X}(\gamma)$, which implies that $\FA_{X'}(|\gamma|)\preceq \FA_X(|\gamma|)$. 
	\end{proof}
	
	\begin{proposition}\label{prop:freeproducts}
		Let $H_1, H_2$ be groups of type $FP_2$. Let $H = H_1\ast H_2$. Then $\FA_H$ is $\simeq$ equivalent to $\max\bigl(\FA_{H_1}, \FA_{H_2}\bigr)$.
	\end{proposition}
	\begin{proof}
		For full details on trees of spaces associated to graphs of groups we refer the reader to \cite{ScottWall}.
		Let $X_i$ be a homological Cayley complex for $H_i$. 
		Define $X$ as the following tree of spaces. 
		The underlying tree $T$ is a bipartite tree with vertex set $V_1 \sqcup V_2$. To each vertex in $V_i$ we associate $X_i$ and for each edge $e$ we associate a point. 
		Vertices in $V_i$ have valence $\operatorname{Card}(H_i)$, and for each $v_i\in V_i$, the edges adjacent to $v_i$ are in one-to-one correspondence with $X_i^{(0)}$.
		Let $e$ be the edge corresponding to $x\in X_i^{(0)}$, then the map $\{*\}\to X_i$ is given by $*\mapsto x$. 
		
		It is worth noting that the underlying tree is the Bass--Serre tree for the free product. 
		There is a free action of $H$ on $X$, extending the action of $H_i$ on $X_i$.
		The quotient $X/H$ is the union of $X_1/H_1 \cup X_2/H_2$ with an edge joining the unique vertex of $X_1/H_1$ to the unique vertex of $X_2/H_2$. 
		Thus $X$ has a free, proper, cellular and cocompact action of $H$. Denote $A_1$ the disjoint union of all copies of $X_1$ and all the open intervals corresponding to edges $e$, and define $A_2$ similarly. 
		Then from the Mayer--Vietoris sequence applied to the decomposition $X=A_1\cup A_2$, it follows that $H_1(X) = 0$. 
		Thus we see that $\FA_X$ is well defined and by \Cref{prop:spaces} we have $\FA_X\simeq \FA_H$. 
		
		To obtain the upper bound let $w$ be a $1$-cycle in $X$. 
		Then $w$ is a sum of loops $w_j$. 
		Let $f\colon S^1\to X$ be a loop. 
		If the image of $f$ is not contained in a single copy of $X_i$ for some $i= 1$ or $2$, then there is an edge of $X$ which the image of $f$ traverses twice in opposite directions. 
		Hence $f$ is not a minimal representative for its homology class. 
		Thus, by taking a minimal representative for $w$ we have that each loop $w_j$ belongs to a single copy of $X_1$ or a single copy of $X_2$. 
		
		We can now find a minimal representative for $w$ which is a sum $\sum_j w_j^1 + \sum_j w_j^2$, where $w_j^i$ are loops contained in copies of $X_i$.
		Let $n_1 = \sum_j |w_j^1|$ and $n_2 = \sum_j |w_j^2|$. 
		Then, we have the following inequalities: 
		\begin{multline*}\HArea_X(w)= \HArea_X(\textstyle\sum_j w_j^1 + \textstyle\sum_j w_j^2) 
		\leq \textstyle\sum_j\HArea_{X_1}( w_j^1) + \textstyle\sum_j\HArea_{X_2}(w_j^2)\\
		\leq \textstyle\sum_j\FA_{X_1}(|w_j^1|) +\textstyle\sum_j\FA_{X_2}(|w_j^2|)
		\leq \FA_{X_1}(n_1) + \FA_{X_2}(n_2).
		\end{multline*}
		In the last inequality we used the superadditivity of $\FA_{X_i}\simeq\FA_{H_i}$, which holds by \Cref{prop:superadd}.
		
		Since homological filling functions are non-decreasing and $|w| = n_1 + n_2$, we obtain a bound $\HArea_X(w) \leq \FA_{X_1}(|w|) + \FA_{X_2}(|w|)$. 
		By taking the maximum of the two numbers $\FA_{X_1}(|w|)$ and $\FA_{X_2}(|w|)$, we obtain an upper bound of the form 
		\[
		\HArea_X(w)\le\max_i\bigl\{\FA_{X_i}(|w|)+ \FA_{X_i}(|w|)\bigr\}\leq \max_i\bigl\{\FA_{X_i}(2|w|)\bigr\}.
		\]
		Thus we see that $\FA_X\preceq \max\bigl(\FA_{X_1},\FA_{X_2}\bigr)$. 
		
		To obtain the lower bound, we use \Cref{prop:retract}. 
		There are retractions $H\to H_1$ and $H\to H_2$. 
		Thus we see that $\FA_{H_1}\preceq \FA_{H}$ and $\FA_{H_2}\preceq \FA_{H}$, hence $ \max\bigl(\FA_{X_1},\FA_{X_2}\bigr)\preceq \FA_{H}$, which establishes the equivalence $\FA_X\simeq \max\bigl(\FA_{X_1},\FA_{X_2}\bigr)$.
	\end{proof}
	
	\medskip
	We now recall the definition of quasi-isometric groups.
	\begin{definition}\label{def:qi}
		(Quasi-isometry)
		Let $(X,d_X)$ and $(Y,d_Y)$ be metric spaces. 
		A map $f\colon X\to Y$ is a {\em quasi-isometry} if there exists a  constant $K\geq 1$ 
		such that the following conditions hold: 
		\begin{enumerate}
			\item for all $x,x'\in X$ we have,
			$\frac1K\cdot d_X(x,x')-K\le d_Y\big(f(x),f(x')\big)\le K\cdot d_X(x,x')+K.$
			\item For all $y\in Y$ there exists $x\in X$ such that $d_Y(f(x), y) \leq K$. 
		\end{enumerate}

		Two metric spaces $(X,d_X)$ and $(Y,d_Y)$ are \emph{quasi-isometric} if there exists a quasi-isometry $f\colon X\to Y$.
		Two groups $G$ and $H$ are \emph{quasi-isometric}, if they are quasi-isometric as metric spaces $(G,d_\A)$ and $(H,d_\B)$ with the word metrics with respect to some generating sets $\A\subset G$, $\B\subset H$.
	\end{definition}

	The Dehn function of a  finitely presented group is an invariant of the quasi-isometry class of the group, as was shown in~\cite{AlonsoFrench}. An analogous statement for the homological Dehn function of a finitely presented group was proved in \cite[Th.\,2.1]{Fle} and in \cite[Lem.\,1]{You}. The proof in \cite{You} relies on the work of \cite{AWP} for which a simply connected complex is required to extend the maps. 
	Here we appropriately modify the proof from \cite{AWP} to show quasi-isometry invariance of homological filling functions for groups of type $FP_2$. Note that the property $FP_n$ is a quasi-isometry invariant of a group, see~\cite{Alonso}.
	
	\begin{theorem}\label{thm:qi}
		The homological filling function is a quasi-isometry invariant for groups of type $FP_2$.
	\end{theorem}
	\begin{proof}
		By \Cref{prop:loops}, it suffices to check what happens for loops. 
		Let $G, H$ be quasi-isometric groups of type $FP_2$. 
		Let $\phi\colon G\to H$ be a quasi-isometry. 
		
		Let $\la \A\left|\right|R_0\ra$ and $\la \B\left|\right|S_0\ra$ be some homological finite presentations for $G$ and $H$, respectively.  
		Let $X$ be the homological Cayley complex for $\la \A\left|\right|R_0\ra$ and $Y_0$ be the homological Cayley complex for $\la \B\left|\right|S_0\ra$. 
		We can view $\phi$ as a map from $X^{(0)}\to Y_0^{(0)}$. 
		We can extend $\phi$ across edges in the following way. If $x,x'$ are two vertices in $X^{(0)}$ connected with an edge in $X$, then  $d_{\B}\big(\phi(x),\phi(x')\big)\le K\cdot 1+K=2K$. We choose a geodesic path in $Y_0^{(1)}$ connecting $\phi(x)$ and $\phi(x')$ and, since $Y_0^{(1)}$ is the Cayley graph with respect to the generating set $\B$, we see that this map sends each edge in $X^{(1)}$ to a path of length $\leq 2K$ in $Y_0^{(1)}$. Doing so for all edges of $X$, we obtain a cellular map $\phi\colon X^{(1)}\to Y_0^{(1)}$.

		Since $Y_0$ is not simply connected there is no reason why we can now extend this map over $2$-cells. 
		Thus we must make modifications to $Y_0$ as follows. 
		Let $L$ be the maximum length of a relation of $R_0$. 
		Then each $2$-cell of $X$ has boundary of length $\leq L$. 
		Thus each of these loops maps to a loop of length $\leq 2KL$ in $Y_0$. 
		Mimicking the construction from \Cref{def:CayCx}, we can add an $H$-orbit of disks to all loops of length $\leq 2KL$ in $Y_0$ to obtain a new complex $Y$. 
		Up to the action of $H$ there are only finitely many such loops in $Y_0$, thus $H$ acts on $Y$ freely, cellularly and cocompactly with one orbit of vertices. 
		Thus, by \Cref{prop:freeimpliesCayley}, we see that $Y$ is the homological Cayley complex for a finite homological presentation of $H$. 
		Therefore, $\FA_Y \simeq \FA_H$. 
		
		As previously mentioned, each disk of $X$ is attached along a loop of length $\leq L$ and this maps to a loop of length $\leq 2KL$ under the extension of $\phi$ to $X^{(1)}$. 
		All these loops are null-homotopic in $Y$ and thus we can extend the quasi-isometry $\phi$ to a cellular map $f\colon X\to Y$.
		Moreover, we can assume that the image of each $2$-cell of $X$ is a single $2$-cell of $Y$.  
		
		Now let $w = (h_0, h_1, \dots, h_n)$ be an edge loop in $Y$ of length $n$. 
		I.e.\ $h_i$ are vertices of $Y$ such that $h_i, h_{i+1}$ are adjacent for all $i=0,\dots,n-1$, and $h_n=h_0$. 
		For each $h_i$ we can find a vertex $x_i\in X$ such that $d(f(x_i), h_i)\leq K$. 
		Hence $d\big(f(x_i), f(x_{i+1})\big)\leq 2K+1$ and since $\phi$ is a $(K, K)$-quasi-isometry we see that $\frac1K\cdot d(x_i, x_{i+1})-K\leq d\big(f(x_i), f(x_{i+1})\big)$, and, combining with the previous inequality, that $d(x_i, x_{i+1})\leq 3K^2+K.$
		Let $\gamma_i$ be an edge path from $x_i$ to $x_{i+1}$ of length $|\gamma_i|\leq 3K^2 + K$, and let $\gamma$ be the concatenation of the edge paths $\gamma_i$. 
		We have: $|\gamma| \leq n(3K^2+K)$. 
		
		Let $\lambda$ be the image of $\gamma$ under $f$. 
		Thus, $\lambda$ is a loop in $Y$ of length $|\lambda|\leq 2K|\gamma|\leq  2Kn(3K^2+K)$. 
		We obtain a bound on $\HArea_Y(w)$ by first getting a filling for $\lambda$ and then a filling for $w - \lambda$. 
		
		We begin with $\lambda$. 
		The loop $\gamma$ has a filling in $X$ of size $\leq \FA_X\big(n(3K^2+K)\big)$. 
		We can map this filling to $Y$ by $f$. 
		Since each $2$-cell of $X$ maps to a single $2$-cell of $Y$, the map $f$ does not increase the size of the filling. 
		Thus we obtain a filling for $\lambda$ of size $\leq \FA_X\big(n(3K^2+K)\big)$. 
		
		Now consider $w - \lambda$. 
		We can represent this $1$-cycle $w-\lambda$ as a sum of  $n$ quadrilaterals with vertices $h_i, f(x_i), f(x_{i+1})$ and $h_{i+1}$. 
		We do this by picking for each $i$ a path from $h_i$ to $f(x_i)$ of length $\leq K$. 
		The path from $f(x_i)$ to $f(x_{i+1})$ is on $\lambda$ and has length $\leq 2K(3K^2 + K) = 6K^3 + 2K^2$. 
		Thus the boundary of this quadrilateral has length $6K^3 + 2K^2 + 2K +1$. 
		Let $C = \FA_Y(6K^3 + 2K^2 + 2K +1)$, note that this is a constant. 
		Then $w - \lambda$ has a filling of size $\leq nC$. 
		
		Adding these two fillings together we see that $\HArea_Y(w)\leq \FA_X(n(3K^2+K)) + nC$. 
		Thus by \Cref{prop:loops}, we obtain that $\FA_Y\preceq \FA_X$. 
		
		Since quasi-isometry is an equivalence relation, we can reverse the roles of $G$ and $H$ and  obtain an inequality in the other direction $\FA_X\preceq \FA_Y$, thus proving that $\FA_G\simeq \FA_H$. 
	\end{proof}

	\section{Polyhedral complexes and Leary's construction}\label{sec:leary}
	
	\subsection{Polyhedral complexes and Lipschitz fillings}\label{sec:complexes}
	
	We now recall the definition of $M_\kappa$-poly\-hed\-ral complexes from~\cite{BH}.
	These are defined using the model space $M_\kappa^n$ for $\kappa\in \R$. 
	We will only be interested in the cases that $\kappa = 0$ or $1$, in these cases $M_0^n = \R^n$ with the $\ell^2$-metric and $M_1^n$ is the unit sphere in $\R^{n+1}$ with the arc-length metric. 
	For full details see~\cite[I.7]{BH}.

	\begin{definition}
		Fix $\kappa\in\{0, 1\}$. 
		A {\em convex $M_\kappa$-polyhedral cell} $C\subset M_\kappa^n$ is the convex hull of a finite set of points $P\subset M_\kappa^n$, if $\kappa = 1$, then $P$ (and hence $C$) is required to lie in an open ball of radius $\frac{\pi}{2}$. 
		The {\em dimension} of $C$ is the dimension of the smallest $m$-plane containing it. The {\em interior} of $C$ is the interior of $C$ as a subset of this $m$-plane.
		
		Let $H$ be a hyperplane in $M_\kappa^n$. 
		If $C$ lies in one of the closed half-spaces bounded by $H$, and if $H\cap C\neq\varnothing$, then $F = H\cap C$ is called a {\em face} of $C$.
		The {\em dimension} of a face $F$ is the dimension of the smallest $m$-plane containing it.
		The {\em interior} of $F$ is the interior of $F$ in this plane. 
		The $0$-dimensional faces of $C$ are called its {\em vertices}. 
		The {\em support} of $x\in C$, denoted $\supp(x)$, is the unique face containing $x$ in its interior.
	\end{definition}
	
	The complexes we will study throughout will be built out of convex $M_\kappa$-polyhedra. 
	Such complexes are known as $M_\kappa$-polyhedral complexes. 
	Precisely, we have the following definition: 
	
	\begin{definition}
		Let $(C_\lambda\colon \lambda\in \Lambda)$ be a family of $M_\kappa$-polyhedral cells and 
		let $X = {\bigsqcup_{\lambda\in\Lambda}C_\lambda}$ be their disjoint union. 
		Let $\sim$ be an equivalence relation on $X$ and let $K = X/{\sim}$. Let $p\colon X \to K$ be the natural projection and define $p_\lambda\colon C_\lambda\to K$ as the restriction of $p$ to $C_\lambda$.
		
		The quotient $K$ is called an {\em $M_\kappa$-polyhedral complex} if:
		\begin{enumerate}
			\item  for every $\lambda\in\Lambda$, the restriction of $p_\lambda$ to the interior of each face of $C_\lambda$ is injective;
			
			\item for all $\lambda_1, \lambda_2\in\Lambda$ and $x_1 \in C_{\lambda_1}, x_2 \in C_{\lambda_2}$, if $p_{\lambda_1}(x_1) =   p_{\lambda_2}(x_2)$, then there is an isometry $h\colon \supp(x_1)\to \supp(x_2)$ such that $p_{\lambda_1}(x_1) =   p_{\lambda_2}(h(y))$ for all $y\in \supp(x_1)$.
		\end{enumerate}
		The set of isometry classes of the faces of the cells $C_\lambda$ is denoted $\operatorname{Shapes}(K)$.
	\end{definition}
	
	Let $K$ be a $M_\kappa$-polyhedral complex. 
	We construct a pseudometric on $K$ in the following way. A \emph{piecewise geodesic path from $x$ to $y$} in $K$ is a map $c\colon[a,b]\to K$ such that $c(a)=x$, $c(b)=y$ and there is a subdivision $a=t_0\le t_1\le \dots\le t_k=b$ and geodesic paths $c_i\colon[t_{i-1},t_i]\to C_{\lambda_i}$ such that for each $t\in[t_{i-1},t_i]$ we have $c(t)=p_{\lambda_i}(c_i(t))$. The length $l(c)$ of $c$ is defined as $l(c)\vcentcolon=\sum_{i=1}^k l(c_i)$ where $l(c_i)=d_{C_{\lambda_i}}(c_i(t_{i-1}),c_i(t_i))$. The \emph{intrinsic pseudometric} on $K$ is the function
	\[
	d(x,y)\vcentcolon=\inf\{l(c)\mid \text{$c$ is a piecewise geodesic path from $x$ to $y$}\}.
	\]
	If there is no piecewise geodesic path from $x$ to $y$, we set $d(x,y)\vcentcolon=\infty$.
	If $\operatorname{Shapes}(K)$ is finite, then $d$ is a metric and $(K,d)$ is a complete geodesic metric space~\cite[Th.\,I.7.50]{BH}.

	\medskip
	We say that an $M_0$-polyhedral complex is a {\em cube complex} if each polyhedral cell is $[0,1]^n\subset\R^n$ for some $n$. Such cells will be called \emph{$n$-cubes}, or \emph{cubes}, for brevity. If $n=2$ we will also call them \emph{squares}.
	
	\begin{definition}\label{def:links}
		Let $x$ be a vertex of a cube $C_\lambda$. 
		The \emph{link} $\Lk(x, C_\lambda)$ of $x$ in $C_\lambda$ is the set of initial vectors of geodesic segments in $C_\lambda$ starting at $x$ (one can say that the link consists of the tangent vectors at $x$ of unit length which ``point inside $C_\lambda$''). 
		We metrise the link by specifying the distance between two points as the angle between the corresponding vectors. 
		Here the angle is the euclidean angle. 
		If $C_\lambda$ is a cube of dimension $n$, then the link is an $(n-1)$-simplex metrised as a subset of the unit sphere $S^{n-1}$. Thus, the vertices of $\Lk(x, C_\lambda)$ are in one-to-one correspondence with the edges of $C_\lambda$ adjacent to $x$, and in general, the $k$-dimensional faces of $\Lk(x, C_\lambda)$ correspond to $(k+1)$-dimensional cubical faces of $C_\lambda$ containing $x$.  
		
		There is a topological embedding $\chi_\lambda\colon\Lk(x, C_\lambda)\to C_\lambda$, given by mapping the vertices $\Lk(x, C_\lambda)$ to points at distance $\frac{1}{4}$ from $x$ along the corresponding edges and extending affinely. 
		Let $v$ be a vertex of a cube complex $K$. 
		We define {\em the link of $v$ in $K$}, $\Lk(v, K)$, to be the quotient space $\Lk(v, K)\vcentcolon=\bigsqcup\Lk(x, C_\lambda)/{\sim}$, where the disjoint union is taken over all pairs $(x,\lambda)$ such that $x\in C_\lambda$ and $p_\lambda(x)=v$, and the 
		equivalence relation $\sim$ is given by the condition: 
		$z\sim z'$ if and only if $p_\lambda(\chi_\lambda(z)) = p_{\lambda'}(\chi_{\lambda'}(z'))$. 
		The maps $\chi_\lambda$ above induce a map $\chi\colon\Lk(v, K)\to K$. This turns $\Lk(v, K)$ into an $M_1$-polyhedral complex. If $K$ is finite dimensional, then $\Lk(v, K)$ endowed with the pseudometric induced by the angle metric, as described above, is a complete geodesic metric space \cite[Th.\,I.7.50]{BH}.
	\end{definition}
	
	A cube complex is called \emph{simple} if the link of every vertex is a simplicial complex. 
	A simplicial complex is \emph{flag} if any collection of $k+1$ pairwise adjacent vertices spans a $k$-simplex. 
	A cube complex is \emph{non-positively curved} if the link of each vertex is a flag simplicial complex. 
	By~\cite[Th.\,II.5.20]{BH} non-positively curved cube complexes are exactly cube complexes which are locally $\CAT(0)$.

	\subsection{Lipschitz filling functions}
	
	We now discuss an equivalent formulation of the homological filling discussed in \Cref{sec:homfunctions}.
	The material of this subsection will be used in the proof of \Cref{thm:lowerbound}. 
	
	The following definition of mass is adapted from~\cite{Gro83}.
	
	\begin{definition}(Lipschitz $n$-chain, mass)
		Let $X$ be a metric space and $\sigma\colon C\to X$ be a Lipschitz map of an $n$-dimensional convex $M_\kappa$-polyhedral cell $C$.
		We define the {\em mass} of $\sigma$, denoted $\mass(\sigma)$, to be the infimum of the total volumes of those Riemannian metrics on $C$ for  which the map $\sigma$ is $1$-Lipschitz.  
		Let $C_n^{\Lip}(X)$ be the free abelian group with basis the set of Lipschitz maps $C\to X$. We call elements of $C_n^{\Lip}(X)$ \emph{Lipschitz $n$-chains}.
		Let $c\in C_n^{\Lip}(X)$ be given by $c = \sum_\sigma a_\sigma\sigma$. 
		We define $\mass(c)\vcentcolon = \sum_\sigma |a_\sigma| \mass(\sigma)$.
	\end{definition}

	Given a convex $M_\kappa$-polyhedral cell $C$,  we can take its barycentric subdivision $C'$. By~\cite[I.7.44]{BH}, $C'$ is a $M_\kappa$-simplicial complex isometric to $C$. Thus we can view $\sigma \colon C\to K$ as a map of a simplicial complex to $K$ and thus we can view it as a singular chain in $K$. This gives an inclusion $C_n^{\Lip}(K)\to C_n(K)$.
	Since the boundary of a singular chain is the sum of restrictions to faces, we see that the boundary of a Lipschitz chain, under this inclusion, is the image of a Lipschitz chain.
	Thus we obtain a differential $C_n^{\Lip}(K)\to C_{n-1}^{\Lip}(K)$.
	One can check that the inclusion $C_n^{\Lip}(K)\to C_n(K)$ is a chain homotopy equivalence.

	\begin{definition}(Homological Lipschitz area, homological Lipschitz filling function)
		Let $X$ be an $M_\kappa$-polyhedral complex with $H_1(X) = 0$. 
		Let $\gamma$ be a cellular $1$-cycle. Define the {\em homological Lipschitz area} of $\gamma$ as
		\[	
		\HArea_X^{\Lip}(\gamma) \vcentcolon= \inf\{ \mass(c)\mid \partial c = \gamma, c\in C_2^{\Lip}(X)\}.
		\]
		Define the {\em homological Lipschitz filling function} of $X$ as
		\[
		\FA_X^{\Lip}(n) \vcentcolon= \sup\bigl\{\HArea_X^{\Lip}(\gamma)\mid \gamma\text{ is a cellular $1$-cycle in }X^{(1)},\,|\gamma|\le n\bigr\}.
		\]
	\end{definition}
	
	Before proving the main result of this subsection, we require the following hugely simplified version of the Federer--Fleming Deformation Theorem \cite[Th.\,5.5]{FF}. Compare with \cite[Th.\,2.1]{ABDDY} or \cite[Th.\,10.3.3]{Ep}. 
	The key point of this theorem is that we can replace a Lipschitz chain $c$ with a cellular chain $P(c)$ of similar size. 
	
	\begin{theorem}[\cite{FF}]\label{thm:FF}
		Let $X$ be an $M_\kappa$-polyhedral complex, with $\operatorname{Shapes}(X)$ finite.
		Let $\gamma$ be a cellular $1$-cycle in $X$ and let $c$ be a Lipschitz $2$-chain in $X$ such that $\partial c = \gamma$.
		Then there is a constant $K\ge0$ depending only on $X$ and a cellular $2$-chain $P(c)$ in $X$ such that:
		\begin{enumerate}
			\item $|P(c)|\leq K\mass(c)$; 
			\item $\partial P(c) = \gamma$. \qed
		\end{enumerate} 
	\end{theorem}
	
	\begin{theorem}\label{thm:lipcellequivalent}
		Let $X$ be an $M_\kappa$-polyhedral complex with $H_1(X) = 0$ and\/ $\operatorname{Shapes}(X)$ finite. 
		Then $\FA_X\simeq \FA_X^{\Lip}$. 
	\end{theorem}
	\begin{proof}

		The characteristic map of a $M_\kappa$-polyhedral cell is a $1$-Lipschitz map. Thus we have a homomorphism from the cellular chain complex to the Lipschitz chain complex. It is injective and so we can consider the cellular chains as a subspace of the Lipschitz chains. Since $\operatorname{Shapes}(X)$ is finite there exists a $D\in\R$ such that $\mass(\sigma) < D$ for every characteristic map $\sigma$. Thus the inclusion satisfies $\mass(c) < D |c|$ for all cellular chains $c$.
		
		Let $\gamma$ be a $1$-cycle. 
		By the above we have the inequality $\HArea_X^{\Lip}(\gamma)\leq D\cdot\HArea_X(\gamma)$ and $\FA_X^{\Lip}(n)\leq D\cdot\FA_X(n)$. 
		Thus, $\FA_X^{\Lip}\preceq \FA_X$. 
		
		Thus we are left to prove the other inequality. 
		Let $c$ be a Lipschitz chain with boundary~$\gamma$. 
		By \Cref{thm:FF}, we can find a cellular $2$-chain $P(c)$ with boundary $\gamma$ and  $|P(c)|\leq K\cdot\mass(c)$.
		Thus, $\HArea_X(\gamma)\leq K\cdot\HArea_X^{\Lip}(\gamma)$ and $\FA_X(n)\leq K\cdot\FA_X^{\Lip}(n)$, 
		We conclude that $\FA_X\preceq \FA_X^{\Lip}$, which completes the proof. 
	\end{proof}
	
	The following will also be useful in comparing the mass of fillings in different spaces. 
	
	\begin{lemma}\label{lem:lipfillingconstant}
		Let $f\colon X\to Y$ be a $k$-Lipschitz map between metric spaces $X$ and $Y$. 
		Let $c$ be a Lipschitz $n$-chain in $X$. 
		Then $f(c)$ is a Lipschitz chain in $Y$ satisfying $\mass(f(c))\leq k^n\mass(c)$.
	\end{lemma}
	\begin{proof}
		We will prove it in the case that $\sigma\colon C\to X$ is a Lipschitz map of an $n$-dimensional polyhedral cell $C$. 
		Let $d$ be a Riemannian metric on $C$ such that $\sigma$ is $1$-Lipschitz. Then $f\circ\sigma$ is $k$-Lipschitz with respect to the same metric. 
		
		If $k\geq 1$, then $f\circ\sigma$ is $1$-Lipschitz thus the metric $d$ can be used in the infimum defining $\mass$ and we see that $\mass(f\circ\sigma)\leq \mass(\sigma)\leq k^n\mass(\sigma)$. 
		
		If $k<1$, then we can dilate the metric $d$ to obtain a new metric $d'\vcentcolon=\frac{1}{k}d$. With respect to $d'$ the map $f\circ\sigma$ is $1$-Lipschitz. 
		Thus we see that $d'$ can be used for the infimum in the definition of $\mass$. 
		The volume of $C$ with respect to $d'$ is $k^n$ times the volume of $C$ with respect to $d$. 
		Thus, we obtain the inequality $\mass(f\circ\sigma)\leq k^n\mass(\sigma)$. 
		
		Since the mass of a chain is the sum of the masses of cells we obtain the desired conclusion. 
	\end{proof}

	\subsection{Morse theory on cube complexes}\label{subsec:morse}
	
	We adapt the notion of a Morse function from \cite{BeBr} to cube complexes:
	
	A map $f\colon X\to\R$ defined on a cube complex $X$ is a \emph{Morse function} if 
	\begin{itemize}
		\item for every cell $e$ of $X$, with the characteristic map $\chi_e\colon [0,1]^m\to e$, the composition $f\chi_e\colon [0,1]^m\to\R$ extends to an affine map $\R^m\to \R$ and $f\chi_e$ is constant only when $\dim e=0$;
		\item the image of the $0$-skeleton of $X$ is discrete in $\R$.
	\end{itemize}
	
	Suppose $X$ is a cube complex, $f\colon X\to \R$ a Morse function.
	The {\em ascending link} of a vertex $v$, denoted $\Lk_{\uparrow}(v, X)$ is the subcomplex of $\Lk(v, X)$ corresponding to cubes $e$ such that $f|_e$ attains its minimum at $v$. The {\em descending link}, $\Lk_{\downarrow}(v, X)$, is defined similarly replacing minimum with maximum. 
	
	The \emph{right-angled Artin group}, or RAAG, associated to a flag complex $L$, is a group $A_L$ given by the presentation:
	\[
	A_L=\la v_i\in\Vertices(L)  \mid [v_i,v_j]=1 \text{ if } (v_i,v_j)\in \Edges(L)\ra.
	\] 
	
	Given a flag complex $L$, the {\em Salvetti complex} $S_{L}$ is defined as follows. 
	For each vertex $v_i$ in $L$, let $S^1_{v_i} = S^1$ be a copy of the circle endowed with a structure of a combinatorial complex with a single $0$-cell and a single $1$-cell. For each simplex $\sigma = [v_0, \dots, v_n]$ of $L$ there is an associated $(n+1)$-dimensional torus $T_{\sigma} = S^1_{v_0}\times\dots\times S^1_{v_n}$. If $\tau \subset \sigma$ is another simplex of $L$, then there is a natural cellular inclusion $T_{\tau}\hookrightarrow T_{\sigma}$. We define
	\[
	S_{L} = \bigg(\coprod_{\sigma\subset L} T_{\sigma}\bigg)\Big/\sim
	\]
	where the equivalence is generated by the inclusions $T_{\tau}\hookrightarrow T_{\sigma}$. 
	
	Given a flag complex $L$, the {\em spherical double} $\S(L)$ is defined as follows. 
	For each vertex $v_i$ in $L$, let $S^0_{v_i} = \{v_i^+, v_i^-\}$ be a copy of $S^0$. 
	For each simplex $\sigma = [v_0, \dots, v_n]$ of $L$ there is an associated join $\S_{\sigma} = S^0_{v_0}\ast\dots\ast S^0_{v_n}$. If $\tau \subset\sigma$ is another simplex of $L$, then there is a natural simplicial inclusion $\S_{\tau}\hookrightarrow \S_{\sigma}$. We define
	\[
	\S(L) = \bigg(\coprod_{\sigma\subset L} \S_{\sigma}\bigg)\Big/\sim
	\]
	where the equivalence is generated by the inclusions $\S_{\tau}\hookrightarrow \S_{\sigma}$.
	
	The link of the single vertex of $S_L$ is $\S(L)$.
	This is a flag simplicial complex by~\cite[Lem.\,5.8]{BeBr}, and hence $S_L$ is a non-positively curved cube complex. 
	It follows that the universal cover $\widetilde S_L$ is a $\CAT(0)$ cube complex. 
	
	We define a homomorphism $\phi\colon A_L\to\Z$ by sending each generator to $1$. Denote $BB_L=\ker\phi$. (It is called the Bestvina--Brady kernel in the literature.)
	Let $X_L$ be the universal cover of $S_L$. 
	We can define a map $f\colon X_L\to \R$ as follows. 
	Since $A_L$ acts transitively on the vertices of $X_L$ we have a one-to-one correspondence between vertices of $X_L$ and elements of $A_L$. 
	Given this we define $f(v) = \phi(g_v)$, where $g_v$ is the group element corresponding to $v$. 
	We can extend this over higher skeleta via affine extensions $\R^n\to \R$. 
	The map $f$ descends to a map which we also call $f\colon X_L/BB_L\to \R$. 
	This is a Morse function. 
	The ascending link of $f$ is a copy of $L$ spanned by all vertices $v_i^-$ for $v_i\in\Vertices(L)$, and similarly, the descending link of $f$ is a copy of $L$ spanned by all vertices $v_i^+$.

	\subsection{Leary's construction}\label{learygroups}
	
	In~\cite{Lea}, Ian Leary developed machinery 
	that takes as input a flag complex $L$ with no local cut points and a set $S\subset \Z$. 
	From this data a group $G_L(S)$ is constructed whose  finiteness properties are controlled by the topology of $L$.
	In this subsection, we recall the necessary facts from \cite{Lea}, that we will require for our proofs. 
	Throughout the paper $L$ will always denote a connected compact flag simplicial complex with no local cut points. 
	By Lemma~6.3 of \cite{Lea} every connected compact simplicial complex is homotopy equivalent to a connected compact flag simplicial complex with no local cut points.
	
	\begin{theorem}[\protect{\cite[Cor.\,10.2,\,15.4]{Lea}}]\label{thm:lea} 
		Let $L$ be a simplicial flag complex with no local cut points, and let $\widetilde L$ be its universal cover. Then if $L$ and $\widetilde{L}$ are $n$-acyclic, then $G_L(S)$ is of type $FP_{n+1}$. If $L$ and $\widetilde{L}$ are acyclic, then $G_L(S)$ is of type $FP$. 
		Moreover, if $\pi_1(L)$ is nontrivial, then the groups $G_L(S)$ form uncountably many isomorphism classes, as $S\subset\Z$ varies. \qed
	\end{theorem}
	
	For each integer $n$ there is a single vertex $v$ of $X_L/BB_L$ such that $f(v) = n$. 
	Let $V_S = \{v\in (X_L/BB_L)^{(0)} \mid f(v)\notin S\}$. 
	Then $G_L(S) \vcentcolon= \pi_1\big((X_L/BB_L)\smallsetminus V_S\big)$. 
	
	The group $G_L(S)$ also acts on a $\CAT(0)$ cube complex $X_L^{(S)}$, which is defined as follows. 
	Take the universal cover of $(X_L/BB_L)\smallsetminus V_S$, lift the metric and then take the metric completion, the resulting cube complex is $X_L^{(S)}$. 
	There is a map $b\colon X_L^{(S)}\to X_L/BB_L$ which is a covering map away from the vertices.
	By considering the composition $X_L^{(S)}\to X_L/BB_L\to \R$, we obtain a Morse function  $f^{(S)}\colon X_L^{(S)}\to \R$. 
	
	The action of $G_L(S)$ is not cocompact since $X_L^{(S)}/G_L(S) = X_L/BB_L$ is an infinite cube complex. 
	Let $v$ be a vertex of $X_L^{(S)}$. 
	If $f^{(S)}(v)\in S$, then the link of $v$ in $X_L^{(S)}$ is $\S(L)$. 
	If $f^{(S)}(v)\notin S$, then the link of $v$ in $X_L^{(S)}$ is $\widetilde{\S(L)}$.
	In the case that $L$ is connected, not a single point and has no local cut points, it was shown in \cite[Prop.\,6.9,\,Cor.\,7.2]{Lea} 
	that $\widetilde{\S(L)} = \S(\widetilde{L})$.
	If $L$ has infinite fundamental group, then the action of $G_L(S)$ on $X_L^{(S)}$ will not be proper.
	Let $Z_L^{(S)} = {\big(f^{(S)}\big)}^{-1}\big(\frac{1}{2}\big)\subset X_L^{(S)}$. 
	Then the action of $G_L(S)$ on $Z_L^{(S)}$ is free, proper and cocompact~\cite[Cor.\,9.2,\,9.3]{Lea}. (Note that in~\cite{Lea} the space $Z_L^{(S)}$ is denoted $X_{1/2}^{(S)}$.)
	
	\begin{theorem}[\protect{\cite[Prop.\,10.1]{Lea}}]
		If $L$ and $\widetilde{L}$ are $n$-acyclic, then so is $Z_L^{(S)}$. 
		If $L$ and $\widetilde{L}$ are acyclic, then so is $Z_L^{(S)}$. \qed
	\end{theorem} 
	
	It will be useful to give $Z_L^{(S)}$ the structure of a $M_0$-polyhedral complex. 
	This can be done by defining the $k$-cells of $Z_L^{(S)}$ to be the intersection of $Z_L^{(S)}$ with the $(k+1)$-cubes of $X_L^{(S)}$. In what follows it will also be convenient to modify the combinatorial structure of $X_L^{(S)}$ by subdividing it  so that $Z_L^{(S)}$ becomes a combinatorial subcomplex of $X_L^{(S)}$.  
	Since $G_L(S)$ acts on $X_L^{(S)}$ by permuting cubes, we obtain that the restricted action on $Z_L^{(S)}$ is a cellular action.

	\subsection{Geodesic fellow-travelling in CAT(0) cube complexes} 
	
	Let $X$ be a finite dimensional (but not necessarily locally finite) CAT(0) cube complex. In this subsection, we show how to construct geodesic edge paths in $X$ with strong control on how they fellow-travel.  The argument is adapted from that of \cite{NR} where the hypothesis of $X$ admitting a cocompact group action by isometries is only needed to establish regularity of certain languages.

	The following definitions are taken from section 3 of \cite{NR}. 
	
	A sequence of cubes $\{C_i\}_{i=1}^m$ in $X$ is called a \emph{cube-path} if each $C_i$ has a pair of  vertices $v_{i-1} \not= v_i$ such that  $C_i$ is the unique minimal dimension cube containing $\{v_{i-1},v_i\}$, and $C_{i} \cap C_{i+1} = \{v_i\}$. We say that the cube-path is from $v_0$ to $v_m$. Note that the union $\bigcup_{i=1}^m C_i$ is a connected subcomplex of $X$. The \emph{length} of the cube-path $\{C_i\}_{i=1}^m$  is defined to be $m$. 
	
	A cube-path $\{C_i\}_{i=1}^m$ is said to be \emph{normal} if $St(C_i) \cap C_{i+1} = \{v_i\}$ for each $1 \leq i \leq m$. Recall that if $C$ is a cube in $X$, then the \emph{star of $C$}, $St(C)$, is defined to be the union of all cubes of $X$ which contain $C$ as a face. This collection includes $C$ itself.
	
	In \cite{NR} families of  hyperplanes are used to construct normal cube-paths between points. Recall that if $[0, 1]^m$ is a model $m$-cube, then a \emph{midcube} is a codimension one subspace obtained by setting one coordinate equal to $\frac{1}{2}$. 
	A \emph{hyperplane} $H$ in a CAT(0) cube complex $X$ is a connected subspace which is the union of a family of midcubes of cubes of $X$ with the property that for each cube $C$ of $X$, either $H \cap C = \varnothing$ or $H\cap C$ is a single midcube of $C$. The hyperplane $H$ is said to be \emph{dual} to the edge $e$ of $X$, if $H$ intersects $e$ in the midpoint (midcube) of $e$. 
	
	Here are some pertinent facts about hyperplanes and normal cube-paths in a finite dimensional CAT(0) cube complex X. The first three are proven in  \cite{Sageev}  using slightly different language. The last three are proven in \cite{NR} and the proofs hold for finite dimensional CAT(0) cube complexes. 
	
	\setlist[enumerate,1]{label={(H.\arabic*)}}
	
	\begin{enumerate}
		\item   A hyperplane $H$ in $X$ is itself a CAT(0) cube complex and the inclusion $H \to X$ is an isometric embedding. 
		\item Each hyperplane $H$ in $X$ separates $X$ into two connected components, called \emph{half\-spaces} and denoted by $H^+$ and $H^-$. The hyperplane $H$ is said to \emph{separate} vertices $v$ and $w$ of $X$ if one of $v$, $w$ is contained in $H^+$ and the other in $H^-$.  
		\item The union of all cubes of $X$ which intersect $H$ is a convex subcomplex $N(H)$ of $X$ which has the cubical structure of $H \times [0,1]$. We denote the subcomplex $H\times \{i\}$ of $N(H)$ by $\partial_i(N(H))$ for $i=0,1$.   
		\item \cite[Prop.\,3.3]{NR} Given vertices $v$ and $w$ of $X$, there exists a unique normal cube-path from $v$ to $w$. 
		\item \cite[Lem.\,3.2]{NR} Let $v$ and $w$ be vertices of $X$, let $\{C_i\}_{i=1}^m$ be the normal cube path from $v$ to $w$, and let $H$ be a hyperplane of $X$ which separates $v$ and $w$. Then $H$ meets $\bigcup_{i=1}^mC_i$ in a connected subset; namely, the midcube of  $C_j$ for some $1 \leq j \leq m$. 
		\item \cite[Proof of Prop.\,3.3]{NR} Suppose that each of the distinct hyperplanes $H_1, \ldots, H_k$ separates vertex $v$ from vertex $w$ and intersects the star neighborhood $St(v)$. There exists a $k$-dimensional cube $C$ in $St(v)$ so that  $H_i\cap C$ are the $k$ midcubes of $C$. 
	\end{enumerate}
	
	It follows from (H.5) that the set of hyperplanes in $X$ which separate $v$ from $w$ is in bijective correspondence with the collection of midcubes of the normal cube-path from $v$ to $w$, and so has cardinality equal to $\sum_{i=1}^m{\rm dim}(C_i)$. It follows from (H.2) and the definition of hyperplane that any combinatorial path from $v$ to $w$ has to have length bounded below by  $\sum_{i=1}^m{\rm dim}(C_i)$. 
	In particular, if $\rho$ is a geodesic edge path in 
	$\bigcup_{i=1}^mC_i$ from $v$ to $w$, then $\rho$ is a geodesic in $X^{(1)}$. We say that such geodesics $\rho$ are \emph{carried} by the  normal cube-path from $v$ to $w$. 
	
	\setlist[enumerate,1]{label={(\arabic*)}}
	
	Normal cube-paths are constructed recursively in \cite{NR}. We provide a sketch of the construction here, since the framework will be used in the proof of Lemma~\ref{lem:NibloReeves} below. 
	
	\begin{enumerate}
		\item Given vertices $v, w$ of $X$, the family ${\mathcal H}(v,w)$ of hyperplanes of $X$ which separate $v$ from $w$ is finite. Indeed, if $\beta$ is a combinatorial edge path from $v$ to $w$, then every hyperplane in ${\mathcal H}(v,w)$ intersects $\beta$ in midpoints of edges. Since distinct hyperplanes meet $\beta$ in distinct midpoints, we have 
		$|{\mathcal H}(v,w)| \leq |\beta| < \infty$.  
		\item Let $\{H_1, \ldots , H_k\} \subset {\mathcal H}(v,w)$ be the collection of hyperplanes of ${\mathcal H}(v,w)$ which intersect $St(v)$. By (H.6) there exists a $k$-dimensional cube $C_1$ in $St(v)$ so that  $H_i\cap C_1$ are the $k$ midcubes of $C_1$. Let $v_1$ be the unique vertex of $C_1$ which is separated from $v$ by each of the $H_i$. Verify that ${\mathcal H}(v,w) \smallsetminus \{H_1, \ldots, H_k\} = {\mathcal H}(v_1, w)$. 
		\item Repeat step 2 above for the points $v_1$ and $w$ to obtain a cube $C_2$ in $St(v_1)$ and prove  that the normality condition $C_2 \cap St(C_1) = \{v_1\}$ holds. 
		\item Continue inductively until ${\mathcal H}(v,w)$ is exhausted.  
	\end{enumerate}

	\begin{lemma}\label{lem:NibloReeves}
		Let $X$ be a finite dimensional CAT(0) cube complex, let $v, w, w'$ be vertices in $X$ with $w$ adjacent to $w'$, and let $\rho_{w'}$ be a geodesic edge path from $v$ to $w'$ which is carried by the normal cube-path from $v$ to $w'$. Then there exists a geodesic edge path $\rho_w$ from $v$ to $w$ which is carried by the normal cube-path from $v$ to $w$, and the smallest full subcomplex of $X$ containing the images of $\rho_{w}$ and $\rho_{w'}$ is an embedded disk diagram of one of the four types presented in~\Cref{fig:fellow}.
		\begin{figure}
			\def\svgwidth{4in}
			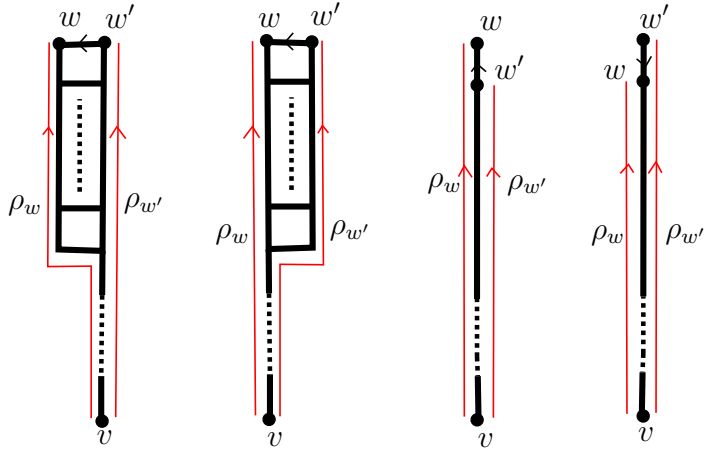
			\caption{\label{fig:fellow}Fellow-travelling for geodesics carried by normal cube-paths. (Geodesic paths $\rho_{w'}$ and $\rho_w$ are indicated in red.)}
		\end{figure}
	\end{lemma}
	
	\begin{proof}
		Let $H$ denote the hyperplane dual to the edge $[w,w']$ and choose halfspace labels so that $w \in H^+$ and $w' \in H^-$. Suppose that $v \in H^-$ (the case $v\in H^+$ is handled similarly). By (H.4) there are unique normal cube-paths  $\{C_i\}_{i=1}^m$ from $v =v_0$ to $w=v_m$ and $\{C'_i\}_{i=1}^{m'}$ from $v=v'_0$ to $w'=v'_{m'}$.  
		
		Now $\bigcup_{i=1}^m C_i$ is a connected subcomplex of $X$ containing $v$ and $w$ and $H$ separates $X$ with $v$ and $w$ in different components. Therefore,  $H \cap \bigcup_{i=1}^m C_i \not= \varnothing$ and, by (H.5), this intersection is the midcube of some $C_j$. 
		Let ${\mathcal H}(p,q)$ denote the collection of all hyperplanes of $X$ which separate vertices $p$, $q$. Note that ${\mathcal H}(v,w) = \{H\} \cup {\mathcal H}(v,w')$.
		Therefore, by the way normal cube-paths are constructed, $C'_i = C_i$  for $1 \leq i < j$ and $v'_{i} = v_{i}$ for $0 \leq i < j$. In particular, $v'_{j-1} =v_{j-1}$. Now the subset of hyperplanes in ${\mathcal H}(v_{j-1},w)$ which meet $St(v_{j-1})$ is the union of $\{H\}$ and the subset of hyperplanes in ${\mathcal H}(v_{j-1},w')$ which meet $St(v_{j-1})$. This implies that $C'_j$ is a codimension 1 face of $C_j$ and $[v'_j, v_j]$ is an edge of $C_j$ which is dual to the hyperplane $H$.  
		
		There are two cases to consider: $j<m$ and $j=m$. In the case $j=m$ there are 2 subcases depending on the dimension of $C_m$. If $C_m$ has dimension 1, then $H$ is the unique hyperplane separating $v_{m-1}$ from $w$. Therefore, $w' = v_{m-1}$,  $m' =m-1$ and the geodesic $\rho_{w}$ is defined to be the concatenation of  the geodesic $\rho_{w'}$ and the geodesic $[w',w]$.  The full subcomplex containing the images of these geodesic paths is indicated in the 3rd image of Figure~\ref{fig:fellow}. If $C_m$ has dimension 2 or more, then $C'_m$ is a codimension 1 face of $C_m$. Note that $m' =m$ in this case.  Now $\rho_w$ can be  defined as the concatenation of the geodesic $\rho_{w'}$ and the geodesic $[w', w] = [v'_m, v_m]$ and the full subcomplex is again described by the 3rd image in Figure~\ref{fig:fellow}.  Note that there is some variation possible here. One could define $\rho_w$ as the concatenation of the subpath of $\rho_{w'}$ from $v$ to $v'_{m-1} =v_{m-1}$ and any geodesic edge path in $C_m$ from $v_{m-1}$ to $v_m =w$ which projects onto the subpath of $\rho_{w'}$ from $v_{m-1}$ to $v'_m = w'$ under the codimension 1 projection $C_m \to C'_m$.  In this case the full subcomplex containing the images of $\rho_w$ and $\rho_{w'}$ can be either the first or the third image in Figure~\ref{fig:fellow}. 
		
		Finally, we consider the case $j<m$. Recall that  $C'_j$ is a codimension 1 face of $C_j$ and $[v'_j, v_j]$ is an edge of $C_j$ which is dual to the hyperplane $H$. In particular, the dimension of $C_j$ is 2 or more. The intersections $[v'_j, v_j] \cap H$ and $[w',w] \cap H$ are vertices in the cubical structure on the hyperplane $H$. Let $\{C^H_i\}_{i=1}^k$ be the unique normal cube-path in the hyperplane $H$ from $[v'_j, v_j] \cap H$ to $[w',w] \cap H$. 
		
		\noindent
		\emph{Claim:} The normal cube-path $\{C^H_i\}_{i=1}^k$ in the hyperplane $H$ determines a sequence of  cubes $\{C^H_i \times [0,1]\}$  in the cubical neighborhood $N(H)$. This sequence in turn determines normal cube-paths $\{ C^H_i \times 0\}$ from $v'_j$ to $w'$ and $\{ C^H_i \times 1\}$ from $v_j$ to $w$.  
		
		It follows from this  claim that $m' = m$ and that the dimension of $C'_i$ equals the dimension of $C_i$ for $j+1 \leq i \leq m$. Furthermore, given a geodesic path $\rho_{w'}$ from $v$ to $w'$ carried by $\{C'_i\}_{i=1}^m$, one defines $\rho_w$ as follows. Use the subpath of $\rho_{w'}$ from $v$ to $v'_j$ to define a geodesic from $v$ to $v_j$ as in the case $j=m$. Denote the complementary subpath of $\rho_{w'}$ (from $v'_j$ to $w'$) by $\rho_{w'}^c$. Note that $\rho_{w'}^c$ is carried by the normal cube-path from  $v'_j$ to $w'$ and so has image in $\partial_0N(H)$. Under the projection  $N(H) = N \times [0,1] \to N\times 1 = \partial_1N(H)$ the path $\rho_{w'}^c$  determines a geodesic edge path from $v_j$ to $w$ which is carried by the normal cube-path from $v_j$ to $w$.  Denote this by $\rho_w^c$. The full subcomplex of $X$ containing the images of $\rho_{w'}^c$ and $\rho_w^c$  is isomorphic (as a cube complex) to the product of $[0,1]$ with the image of $\rho_{w'}^c$. Finally, the geodesic path $\rho_w$ is defined as the concatenation of the geodesic path from $v$ to $v_j$ with the path $\rho_w^c$. By construction, the full subcomplex containing the images of $\rho_w$ and $\rho_{w'}$ is as shown in the first image in Figure~\ref{fig:fellow}. 
		
		Finally, we establish the claim. Recall that ${\mathcal H}(v,w) = {\mathcal H}(v,w') \cup \{H\}$ and that ${\mathcal H}(v,v_j) = {\mathcal H}(v,v'_j) \cup \{H\}$. Therefore, ${\mathcal H}(v_j, w) = {\mathcal H}(v'_j,w')$.  Given $K \in {\mathcal H}(v_j, w)$, note that $[v_j, v'_j] \cup H \cup [w',w]$ is a connected set containing points $v_j$ and $w$ in distinct halfspaces of $K$. Therefore, $K \cap ([v_j, v'_j] \cup H \cup [w',w]) \not= \varnothing$. But $K \cap [v_j,v'_j] = \varnothing = K \cap [w,w']$  (since these edges are dual to $H$ and $K \not= H$), and so $K \cap H \not=\varnothing$. The intersection $K\cap H = L$ is a hyperplane in the cube complex $H$. Note that $L$ separates the vertices $[v_j, v'_j] \cap H$ and $[w,w'] \cap H$ (if it did not then K would not separate $v_j$ and $w$ which contradicts $K \in {\mathcal H}(v_j, w)$). 
		
		Conversely, given a hyperplane $L'$ in $H$ which separates the vertices $[v_j, v'_j] \cap H$ and $[w,w'] \cap H$, note that $L' \times [0,1]$ is a hyperplane in the cubical neighborhood $N(H) = H \times [0,1]$ of $H$ in $X$. Extend $L'\times [0,1]$ to a unique hyperplane $K'$ in $X$. 
		Note that $K'$ separates $v_j$ from $w$ since the concatenation of the geodesic from $v_j$ to $[v_j,v'_j] \cap H$, the geodesic from  $[v_j,v'_j] \cap H$ to $[w,w'] \cap H$, and the geodesic from $[w,w'] \cap H$ to $w$ is a connected path which crosses $K'$ exactly once.
		This establishes a bijection between ${\mathcal H}(v_j, w) = {\mathcal H}(v'_j,w')$ and the set of hyperplanes in 
		$H$ which separate $[v_j, v'_j] \cap H$ and $[w,w'] \cap H$. 
		
		It is clear from the construction of the hyperplane $K'$ in $X$ from the hyperplane $L'$ in $H$ by extending the product $L'\times [0,1]$ that if the hyperplane $L'$ in $H$ meets the star neighborhood of $([v_j,v'_j]\cap H)$ in the cube complex $H$, then the hyperplane $K'$ meets the star neighborhood of $v_j$ in $X$. This sets up a bijective  correspondence (cube-by-cube) between the normal cube-path  in $H$ from  $[v_j, v'_j] \cap H$ to  $[w,w'] \cap H$ and the normal cube path in $X$ from $v_j$ to $w$ (and also from $v'_j$ to $w'$). This proves the claim. 
	\end{proof}

	\subsection{\texorpdfstring{Filling loops in  $Z_L^{(S)}$ by disks in  $X_L^{(S)}$ with controlled geometry}{Controlled fillings of loops}}

	Throughout, let $G = G_L(S)$, $X=X_L^{(S)}$, $f=f^{(S)}$, and  $Z = Z_L^{(S)}$. 	
	
	Let $\gamma$ be an embedded combinatorial loop in $Z$. The  goal of this subsection is to construct a  disk filling  of  $\gamma$ in $
	X$ whose geometry and link combinatorics are controlled in terms of the length $|\gamma|$. 
	
	In order to describe how to build a controlled disk filling $\Delta$ of $\gamma$ in $X$, we will need a way of describing particular subcomplexes of $X$. Recall that we modified the combinatorial structure of the cube complex $X$ by subdividing the cubes intersecting $Z$ so that $Z$ becomes a combinatorial subcomplex of $X$. For this discussion we will  use the notation $X_c$ to denote the original cubical structure on $X$ before the subdivision, 
	and will use $X_c^{(i)}$ to denote the $i$-skeleton of $X_c$. Note that $X_c^{(2)}$ is a subcomplex of $X^{(2)}$ but not equal to $X^{(2)}$ since it does not include any $2$-cells of $Z$.

	\begin{proposition}\label{prop:controldehn}	
		Let $\gamma$ be an embedded combinatorial loop in $Z$ of length $n$. There exists a  disk filling $\pi\colon \Delta \to X$  of $\gamma$ with the following properties.  
		\begin{enumerate}
			\item $\pi(\Delta)$ is contained in $X^{(2)}_c$; 
			\item $\pi^{-1}(X_c^{(0)})$ consists of at most $n^2/2$ vertices of $\Delta$; the other vertices of $\Delta$ are sent to $f^{-1}(1/2)$; 
			\item if $v  \in \pi^{-1}(X_c^{(0)})$, then $|\Lk(v, \Delta)| \leq 3n$; 
			\item $\pi(\Delta)$ is contained in the sublevel set $f^{-1}([-n/4, 1 + n/4])$. 
		\end{enumerate}
	\end{proposition} 
	
	\begin{proof}
		The  disk filling $\pi\colon \Delta \to X$ is built from two pieces.  By Lemma~\ref{lem:annulus} below, there exists  a combinatorial loop $\alpha\colon S^1 \to X_c$ which is freely homotopic to $\gamma$, is contained in $f^{-1}([0,1])$, and satisfies $|\alpha| \leq |\gamma|$. Furthermore, there is an annular diagram $A_{\gamma}$ and a map $\pi_{\gamma}\colon A_{\gamma} \to X$ which realises a free homotopy between $\gamma$ and the barycentric subdivision $\alpha'$ of $\alpha$. Note that $\alpha'\colon S^1 \to X$ is a combinatorial loop of length $|\alpha'|=2|\alpha|$. 
		
		Lemma~\ref{lem:fillalpha} below establishes the existence of a  disk filling $\pi_\alpha\colon \Delta_\alpha \to X_c^{(2)}$ of $\alpha$ with suitable control on its geometry and combinatorics. The cell structure on $X_c^{(2)}  $ pulls back via the map $\pi_\alpha$ to yield a subdivision $\Delta^\ast_\alpha$  of $\Delta_\alpha$ and a  disk filling $\Delta^\ast_\alpha \to X$ (which we  also denote by $\pi_\alpha$) of the barycentric subdivision $\alpha'$ of $\alpha$. Note that the new vertices in the cell structure of $\Delta^\ast_\alpha$ all map to $f^{-1}(1/2)$. 
		
		One of the boundary circles of $A_{\gamma}$ maps to the image of  $\alpha'$ in $X$ and this map factors through the map $\Delta^\ast_\alpha \to X$. Now, take $\Delta$ to be $\Delta^\ast_\alpha$ with $A_{\gamma}$ attached via this factor map and set $\pi$ to be the result of pasting together the maps $\pi_\alpha$ and $\pi_{\gamma}$. 
		
		Property (1) follows from property~(2) of Lemma~\ref{lem:annulus}, the fact that the image of $\Delta_\alpha$ is in $X_c^{(2)}$,  and the construction of $\Delta_\gamma$.  
		Property (2) follows from property~(3) of Lemma~\ref{lem:annulus},  from property~(1) of Lemma~\ref{lem:fillalpha}, and how  $\Delta^\ast_\alpha$  is glued to $A_{\gamma}$ to obtain $\Delta_\gamma$.  Property (3) follows form properties (3) of Lemma~\ref{lem:annulus} and  property~(2) of Lemma~\ref{lem:fillalpha}. Finally, property (4) follows from property (2) of Lemma~\ref{lem:annulus} and property (4) of Lemma~\ref{lem:fillalpha}.
	\end{proof}

	The next lemma shows how to construct the annular portion of the  disk filling of $\gamma$ in Proposition~\ref{prop:controldehn}. To understand the construction we need to describe how $2$-cells of $X_c$ are subdivided in $X_c$. Let $\sigma$ be a square in $X_c^{(2)}$. Then when $1/2\in f(\sigma)$, $\sigma$ is subdivided into a triangle, called a \emph{corner of a square} and denoted by $\tau$, and a pentagon. 
	There are two possibilities: either $f(\sigma) = [-1,1]$ and $\tau = \sigma \cap f^{-1}\bigl([1/2,1]\bigr)$, or $f(\sigma) = [0,2]$ and $\tau = \sigma \cap f^{-1}\bigl([0,1/2]\bigr)$, see~\Cref{fig:sigma}.
	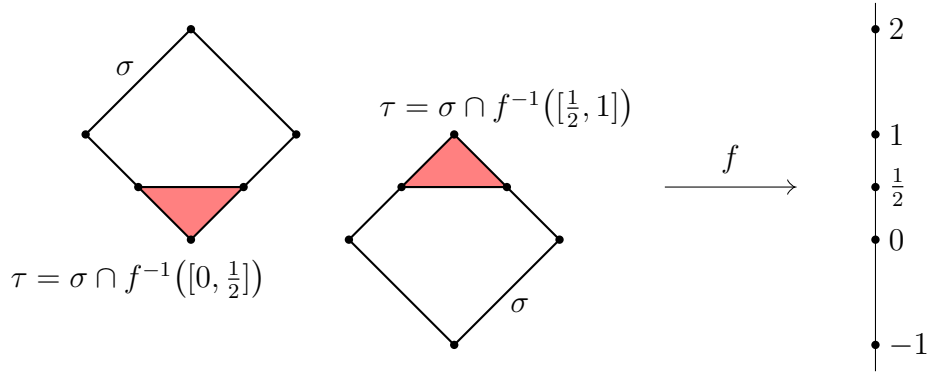
\begin{figure}
	\begin{center}
		\begin{tikzpicture}[scale=0.35]
			\begin{scope}
			\fill[red!50] (-2,-2)--(2,-2)--(0,-4)--(-2,-2);
			\fill (0,4) circle (4.5pt); \fill (0,-4) circle (4.5pt); \fill (-4,0) circle (4.5pt); \fill (4,0) circle (4.5pt);
			\fill (-2,-2) circle (4.5pt); \fill (2,-2) circle (4.5pt);
			\draw[thick] (0,4)--(-4,0)--(0,-4)--(4,0)--(0,4);
			\draw[thick] (-2,-2)--(2,-2);
			\draw (-2.5,2.5) node {$\sigma$};
			\draw (-2,-5.5) node {$\tau=\sigma\cap f^{-1}\bigl([0,\frac12]\bigr)$};
			\end{scope}
			
			\begin{scope}[xshift=10cm, yshift=-4cm]
			\fill[red!50] (-2,2)--(2,2)--(0,4)--(-2,2);
			\fill (0,4) circle (4.5pt); \fill (0,-4) circle (4.5pt); \fill (-4,0) circle (4.5pt); \fill (4,0) circle (4.5pt);
			\fill (-2,2) circle (4.5pt); \fill (2,2) circle (4.5pt);
			\draw[thick] (0,4)--(-4,0)--(0,-4)--(4,0)--(0,4);
			\draw[thick] (-2,2)--(2,2);
			\draw (2.5,-2.5) node {$\sigma$};
			\draw (2,5) node {$\tau=\sigma\cap f^{-1}\bigl([\frac12,1]\bigr)$};
			\end{scope}

			\begin{scope}[xshift=18cm]
			\draw[->-=0.99] (0,-2)--node[above=1pt] {$f$}(5,-2);
			\draw (8,-9)--(8,5);
			\fill (8,4) circle (4.5pt) node[right=1pt] {$2$};
			\fill (8,0) circle (4.5pt) node[right=1pt] {$1$};
			\fill (8,-2) circle (4.5pt) node[right=1pt] {$\frac12$};
			\fill (8,-4) circle (4.5pt) node[right=1pt] {$0$};
			\fill (8,-8) circle (4.5pt) node[right=1pt] {$-1$};
			\end{scope}

		\end{tikzpicture}
	\end{center}
	\caption{\label{fig:sigma}Two possibilities of subdivision of a $2$-cell of $X_c$.}
	\end{figure}

	\begin{lemma}\label{lem:annulus} 
		Let  $\gamma \subset Z$ be an embedded combinatorial loop of length $n$. There exists a combinatorial  loop $\alpha\colon S^1 \to  X_c^{(1)}$ and a combinatorial map of an  annular diagram $\pi_{\gamma}\colon A_{\gamma} \to X$ realizing a free homotopy between $\gamma$ and the barycentric subdivision $\alpha'$  of $\alpha$ such that:
		\begin{enumerate}
			\item $|\alpha| \leq |\gamma|$; 
			\item The $2$-cells of $A_{\gamma}$ are triangles which map homeomorphically to corners of squares in $X$. In particular, the image of $A_{\gamma}$ in $X$ is contained in the sublevel set  $f^{-1}([0,1])$;  
			\item  If $v \in \pi_{\gamma}^{-1}(X_c^{(0)})$, then   $v$ is contained in the boundary circle of $A_{\gamma}$ which maps to the image of $\alpha'$  in $X$ and 
			$|\Lk(v, A_{\gamma})| \leq n$.  
		\end{enumerate}
	\end{lemma}
	
	\begin{proof}
		Let $\gamma$ be a loop in $Z$ of length $n$ with consecutive vertices labeled by $t_0, t_1, \ldots, t_n = t_0$.  Corresponding to each edge $[t_i, t_{i+1}]$ of $\gamma$  there is a unique square $\sigma^2_i$ in $X_c^{(2)}$ so that $\sigma^2_i \cap Z = [t_i, t_{i+1}]$. The edge $[t_i, t_{i+1}]$ is an edge of the corner $\tau_i$ of the square $\sigma^2_i$. Let $v_i$ denote the unique vertex of $\sigma^2_i$ which is in $\tau_i$. 
		The corner $\tau_i$ provides a homotopy between the edge $[t_i,t_{i+1}]$ of $\alpha$ and the directed path $[t_i,v_i] \cup [v_i, t_{i+1}]$ composed of edges in the barycentric subdivision of  $X_c^{(1)}$. The loop $\gamma$ gives rise to a homotopic loop  $\bigcup_i ([t_i,v_i] \cup [v_i, t_{i+1}])$ of length $2n$. Note that the $f$-image of this loop is contained in $[0,1]$ since $f(\tau_i)$ is either $[0,1/2]$ or $[1/2,1]$ for each index $i$. 
		
		Now we tighten this loop to obtain the desired edge path loop $\alpha$. 
		The segment $[t_{i-1}, t_i] \cup [t_i,t_{i+1}]$ of $\alpha$ gives rise to the directed path 	$$[t_{i-1},v_{i-1}] \cup [v_{i-1}, t_i] \cup [t_i,v_i] \cup [v_i, t_{i+1}].$$ 
		Note that $t_i$ is the midpoint of a unique edge of $X_c$ and that $v_{i-1}$ and $v_i$ are vertices of this edge. Either $f(v_{i-1}) =f(v_i)$ in which case $v_{i-1} =v_i$ and the segment 	$[v_{i-1}, t_i] \cup [t_i,v_i]$ is a backtrack, or $f(v_{i-1}) \not=f(v_i)$ in which case $|f(v_{i-1}) - f(v_i)| =1$ and the segment $[v_{i-1}, t_i] \cup [t_i,v_i]$ is the barycentric subdivision of the edge $[v_{i-1}, v_i]$ in $X_c$. We tighten by eliminating backtracks (see~\Cref{fig:backtracks}). This yields an edge path loop (of length at most $2n$) which is the barycentric subdivision of an edge path loop $\alpha$ in $X_c$. Therefore, $2|\alpha| \leq 2n$ and so $|\alpha| \leq n$, establishing property (1). 
	    \begin{figure}	
		\begin{center}
		\def\svgwidth{3.5in}
\begingroup%
  \makeatletter%
  \providecommand\color[2][]{%
    \errmessage{(Inkscape) Color is used for the text in Inkscape, but the package 'color.sty' is not loaded}%
    \renewcommand\color[2][]{}%
  }%
  \providecommand\transparent[1]{%
    \errmessage{(Inkscape) Transparency is used (non-zero) for the text in Inkscape, but the package 'transparent.sty' is not loaded}%
    \renewcommand\transparent[1]{}%
  }%
  \providecommand\rotatebox[2]{#2}%
  \newcommand*\fsize{\dimexpr\f@size pt\relax}%
  \newcommand*\lineheight[1]{\fontsize{\fsize}{#1\fsize}\selectfont}%
  \ifx\svgwidth\undefined%
    \setlength{\unitlength}{516.20656149bp}%
    \ifx\svgscale\undefined%
      \relax%
    \else%
      \setlength{\unitlength}{\unitlength * \real{\svgscale}}%
    \fi%
  \else%
    \setlength{\unitlength}{\svgwidth}%
  \fi%
  \global\let\svgwidth\undefined%
  \global\let\svgscale\undefined%
  \makeatother%
  \begin{picture}(1,0.64573425)%
    \lineheight{1}%
    \setlength\tabcolsep{0pt}%
    \put(0,0){\includegraphics[width=\unitlength,page=1]{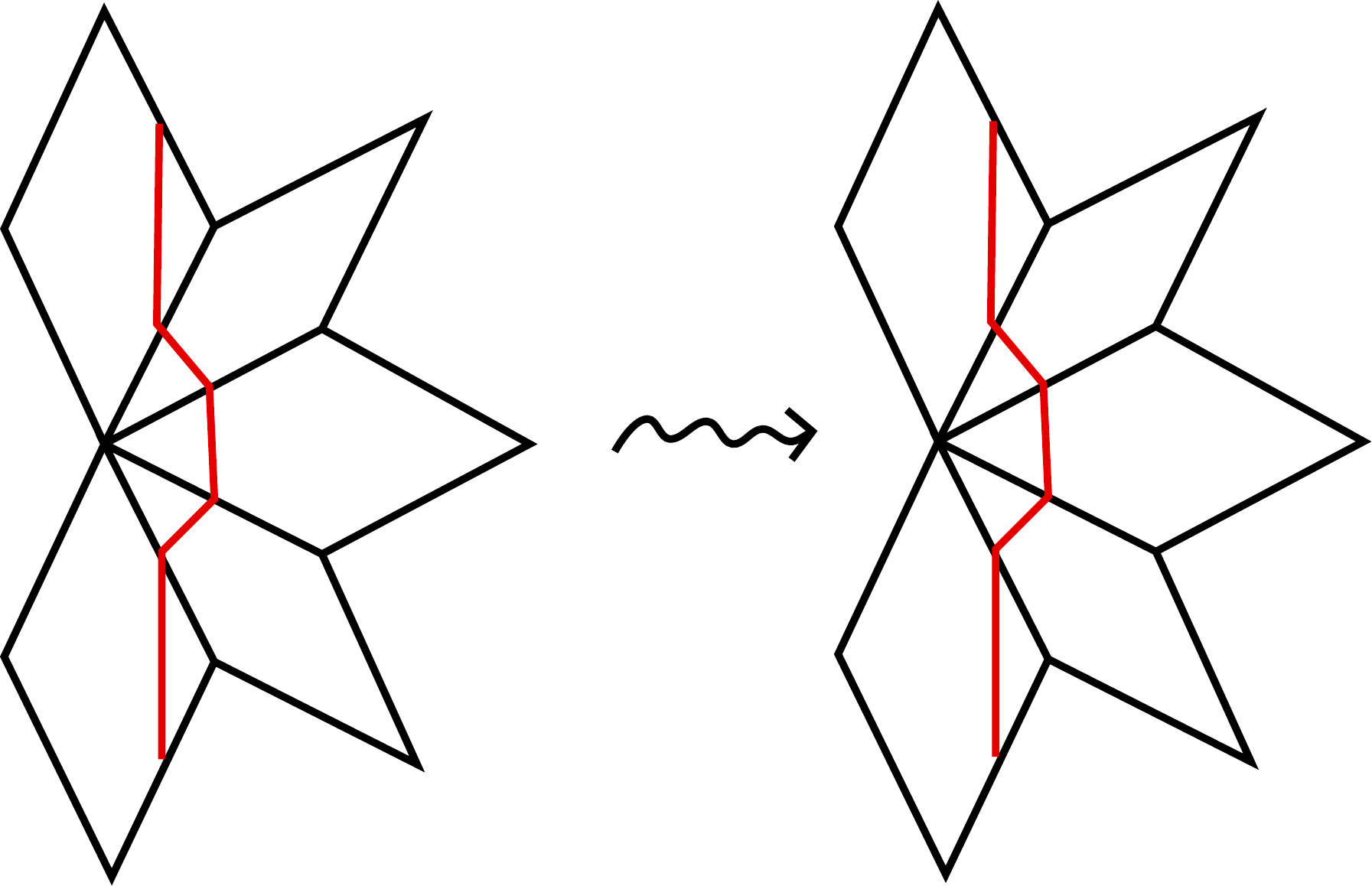}}%
    \put(0.44984184,0.38299022){\color[rgb]{0,0,0}\makebox(0,0)[lt]{\lineheight{1.25}\smash{\begin{tabular}[t]{l}tighten\end{tabular}}}}%
    \put(0,0){\includegraphics[width=\unitlength,page=2]{homotopy.pdf}}%
  \end{picture}%
\endgroup%

			\caption{\label{fig:backtracks}Tightening the edge path.}
		\end{center}
		\end{figure}
		
		Now we build the annular diagram. Suppose $w_{j-1}$, $w_j$, and $w_{j+1}$ are three consecutive vertices along the path $\alpha$. By construction there is a subsegment $\gamma_j$ of the path $\gamma$ whose endpoints  meet $\alpha$ in the barycenters of $[w_{j-1}, w_j]$ and $[w_j, w_{j+1}]$ and so that the union of corners of squares in $X$ determined by $\gamma_j$ is a disk which is topologically a cone on $\gamma_j$ with cone point $w_j$.  We can build copies $C(\gamma_j)$ of each of these cones in the plane and map them via combinatorial homeomorphisms onto the union of corresponding corners of squares in $X$.   Next attach the $C(\gamma_i)$ together in a necklace fashion in the plane by identifying $C(\gamma_j)$ with $C(\gamma_{j+1})$ along the unique vertex in each which maps to the barycenter of $[w_j, w_{j+1}]$. In this way we construct an annular diagram $A_{\gamma}$ and a combinatorial map $\pi_{\gamma}\colon A_{\gamma} \to X$ realizing a free homotopy between $\gamma$ and the barycentric subdivision $\alpha'$ of $\alpha$. The complement of this $A_{\gamma}$ in the plane consists of two connected components; a bounded and an unbounded component. The diagram $A_{\gamma}$  can be embedded in the plane in such a way that $\pi_{\gamma}$ restricted to the boundary of the bounded region (resp.\ unbounded region)  gives the combinatorial map $\alpha'$ (resp.\ $\gamma$).  Property (2) follows immediately from this construction. 
		
		By construction the only vertices of $A_{\gamma}$ which map to $X_c^{(0)}$ are the cone points of $C(\gamma_j)$ and the link of such a cone point in $A_{\gamma}$ has length equal to $|\gamma_j| \leq |\gamma| =n$.   This establishes property~(3). 
	\end{proof}

	Next we show that it is possible to  fill a combinatorial loop $\alpha$ in $X_c$ with a  disk $\pi_\alpha\colon \Delta_\alpha \to X_c$  whose geometry and link combinatorics are bounded in terms of the length of $\alpha$.

	\begin{lemma}\label{lem:fillalpha} 
		Let $\alpha$ be a combinatorial loop in $X_c$ of length $n$. There exists a   disk filling $\pi_\alpha: \Delta_\alpha \to X_c$ of $\alpha$ with the following properties: 
		\begin{enumerate}
			\item The disk $\Delta_\alpha$ has at most $n^2/2$ vertices and at most $n^2/2$ $2$-cells; 
			\item If $v \in \Delta_\alpha$ is a vertex,  then $|\Lk(v, \Delta_\alpha)| \leq 2n$; 
			\item The image  $\pi_\alpha(\Delta_\alpha)$ is contained in $f^{-1}([a-n/4, b+n/4])$ where  the image of $\alpha$ is contained in  $f^{-1}([a,b])$.
		\end{enumerate}
	\end{lemma}
	
	\begin{proof}
		Suppose that the combinatorial loop is given as   a combinatorial path $\alpha: [0,n] \to X_c^{(1)}$ with $\alpha(0) = \alpha(n)$. We build a  disk filling of $\alpha$ inductively as follows. 
		
		We will make use of combinatorial geodesics in $X_c^{(1)}$ between $\alpha(0)$ and an arbitrary vertex $\alpha(i)$. Denote the unique geodesic from $\alpha(0)$ to $\alpha(1)$ by $\rho_{\alpha(1)}$. A single $1$-cell in the plane which maps to the edge $[\alpha(0), \alpha(1)]$ is a  disk filling (with no $2$-cells) for the combinatorial loop $\alpha|_{[0,1]}\cdot \overline{\rho_{\alpha(1)}}$ (we denote by the overline the path traversed in the opposite direction). By the inductive hypothesis,   there is a geodesic path $\rho_{\alpha(i)}$ from $\alpha(0)$ to $\alpha(i)$ and a  disk filling of $\alpha_{[0,i]}\cdot \overline{\rho_{\alpha(i)}}$ whose  disk diagram   contains an embedded path $p_i$ in its boundary circle through which the geodesic $\rho_{\alpha(i)}$ factors. Lemma~\ref{lem:NibloReeves}   guarantees the existence of a geodesic path $\rho_{\alpha(i+1)}$ from $\alpha(0)$ to $\alpha(i+1)$ which fellow travels the geodesic $\rho_{\alpha(i)}$ and a  disk filling $\pi_i: \Delta_i \to X$ of the loop 
		$\rho_{\alpha(i)}\cdot \alpha_{[i, i+1]} \cdot \overline{\rho_{\alpha(i+1)}}$ with no interior vertices and containing embedded edge paths $q_i$ and $p_{i+1}$ in its boundary, with $\rho_{\alpha(i)}$ factoring  through $q_i$ and $\rho_{\alpha(i+1)}$ factoring  through $p_{i+1}$.  Take $\Delta_i$ to be a copy of the appropriate diagram from Figure~\ref{fig:fellow} drawn in the plane, with $\pi_i$ mapping $v$ to $\alpha(0)$, $w'$ to $\alpha(i)$, and $w$  to $\alpha(i+1)$.  Relabel the path $\rho_{w'}$ as $q_i$ and the path $\rho_w$ as $p_{i+1}$.  The inductive step is completed by gluing these two disk diagrams together by identifying the embedded paths $p_i$ and $q_i$ respecting the $\rho_{\alpha(i)}$ factorization. The result is a  disk filling of 
		$\alpha_{[0,i+1]}\cdot \overline{\rho_{\alpha(i+1)}}$  with an embedded path $p_{i+1}$ in its boundary through which $\rho_{\alpha(i+1)}$ factors. Note that $\Delta_i$ is a subdiagram of this  disk diagram. The induction ends with a  disk filling $\Delta_\alpha$ of the combinatorial  loop $\alpha$.
		
		By construction, all vertices of $\Delta_\alpha$ are contained in the union of the paths $p_i$. There are $n-1$ such paths and $|p_i| = |\rho_{\alpha(i)}| \leq n/2$ because the $\rho_{\alpha(i)}$ are geodesics connecting $\alpha(0)$ and $\alpha(i)$ which are at most $\min\{i, n-i\}$ apart. Therefore $\Delta_\alpha$ has at most $(n-1)n/2 \leq n^2/2$ vertices.
		
		Also every $2$-cell comes from subcomplex described in \Cref{lem:NibloReeves}. 
		Each such complex has at most $n/2$ $2$-cells and hence we see that $\Delta_\alpha$ has at most $(n-1)n/2 \leq n^2/2$ $2$-cells and property~(1) holds. 
		
		If $v$ is a vertex of $\Delta_\alpha$,  note that the contribution to $|\Lk(v, \Delta_\alpha)|$ coming from the  disk subdiagram $\Delta_i$ is $|\Lk(v,\Delta_i)| \leq 2$. Since there are  $n$ such subdiagrams in $\Delta_\alpha$, the total contribution to $|\Lk(v,\Delta_\alpha)|$ is at most $2n$. This establishes property~(2).  
		
		Property~(3) follows from the fact that the vertices of $\Delta_\alpha$ map to geodesic paths between $\alpha(0)$ and $\alpha(i)$. These paths have length at most $n/2$ and so  remain in an $n/4$-neighborhood of the image of $\alpha$. The rest of $\Delta_\alpha$ maps into the same neighborhood, since $\pi_\alpha\colon \Delta_\alpha \to X_c$ is combinatorial and $f$ is Morse and attains its max/min at vertices of cells. 
	\end{proof}

	\section{Pushing maps and bounds on homological Dehn functions}\label{sec:pushbounds}
	
	In \cite{ABDDY}, a construction of a map from $X_L$ to the level set of the Morse function $f$ was given. 
	This map was then used to study the Dehn functions of Bestvina--Brady groups. 
	The map had the property that it expanded distance linearly with height from the level set. 
	Here we will adapt these maps to $X_L^{(S)}$ and use them to study the homological filling functions of $G_L(S)$. 

    For any non-integer $a,b\in\R$ let $I = [a, b]$ and denote $X_I = \big(f^{(S)}\big)^{-1}\big(I\big)$. 
	Recall that $Z_L^{(S)}$ is exactly $X_I$ for $I = \{\frac{1}{2}\}$. 
	
	Let $V$ be the set of vertices in $X_L^{(S)}\smallsetminus X_I$. 
	Denote by $B$ the open radius $\frac{1}{4}$ neighbourhood of $V$ in the induced $\ell^1$ metric. 
	We are going to define a retraction $\mathcal{Q}\colon X_L^{(S)}\smallsetminus B\to X_I$ as a slight modification of the construction from Theorem\,4.2 of~\cite{ABDDY}. We proceed as follows.
	
	For all $x\in X_I$ we define $\mathcal{Q}(x) = x$.
	
	Throughout we will identify $\Lk(v, X_L^{(S)})$ with its image in $X_L^{(S)}$, via the map $\chi$ from \Cref{def:links}. 
	We start with the case of $\Lk(v, X_L^{(S)})$ where $f^{(S)}(v)>b$. 
	The case of $f^{(S)}(v)<a$ is analogous with descending links replaced by ascending links. Recall that the link of $v$ is either $\S(L)$ or $\S(\widetilde{L})$, and the descending link of $v$ is either $L$ or $\widetilde{L}$, respectively~\cite[Th.\,9.1]{Lea}. There is a canonical retraction $s\colon \Lk(v, X_L^{(S)})\to \Lk_{\downarrow}(v, X_L^{(S)})$, which can be described as follows. Recall that for any simplicial complex $K$, the $0$-skeleton  $\S(K)^{(0)}=\{v_i^+,v_i^-\}$ is the disjoint union of two copies of the $0$-skeleton $K^{(0)}=\{v_i\}$. Thus we can identify $L$ (resp. $\widetilde L$) with a subcomplex of $\S(L)$ (resp. $\S(\widetilde L)$) spanned by all vertices $v_i^+$. Then the retraction $s$ is defined by sending any pair of vertices $\{v_i^-,v_i^+\}$ to $v_i^+$ and extending simplicially.
	
	Now each point $x$ in the link $\Lk(v, X_L^{(S)})$ belongs to some simplex $\tau$ of the link. 
	There is a natural inclusion $\S(\tau)\to \S(L)$ (resp. $\S(\tau)\to \S(\widetilde L)$) which extends to an  inclusion of a flat subspace $F\to X_L^{(S)}$ (called a \emph{sheet} in~\cite[Sect.\,9]{Lea}).

	 The retraction $s\colon \Lk(v, X_L^{(S)})\to \Lk_{\downarrow}(v, X_L^{(S)})$ induces a retraction $s\colon F\to O$ where $O$ is an orthant of $F$ (called \emph{the downward part of the sheet} in \cite[Sect.\,12]{Lea}). 
	We now define $\mathcal{Q}(x)$ to be the point in $X_b$ which intersects the ray from $v$ through $s(x)$. 
	Using a similar procedure for vertices with $f^{(S)}(v) <a$ we have defined $\mathcal{Q}$ on $\partial B$. 
	
	Let $C$ be a cube of $X_L^{(S)}$. 
	There are three cases to consider; if $C\subset X_I$, we are done. 
	
	Secondly, suppose that $f^{(S)}(C)\subset (b, \infty)$. Having defined $\mathcal Q$ on $C\cap \partial B$, we define $\mathcal Q$ on other faces of  $C\smallsetminus B$ inductively as follows (see~\Cref{fig:pushing}). Let  $e$ be an edge of $C$ that connects two vertices $u,v$ of $C$, such that $f^{(S)}(u)>f^{(S)}(v)$. Then $e\smallsetminus B$ connects  some vertex $u_i^+$ of the descending link of $u$ with some vertex $v_j^-$ of the ascending link of $v$ (which we view as embedded subspaces in $X_L^{(S)}$). Consider the flat $F_C$ containing $C$ which is defined by a simplex at the unique highest vertex of $C$. In $F_C$ the ray from $v$ through $v_j^+$ lies (as a subset) entirely in the ray from $u$ through $u_i^+$. Since the image of $v_j^-$ under the retraction $s$ (computed at the link of $v$) is $v_j^+$, this proves that both $u_i^+$ and $v_j^-$ are mapped under $\mathcal Q$ to the same point  in $X_b$. This allows us to define $\mathcal Q(e\smallsetminus B)$ to be this common image point $\mathcal Q(u_i^+)=\mathcal Q(v_j^-)$.
	Now having defined  $\mathcal{Q}$ on all the faces of $C\smallsetminus B$ of dimension less than $i$, we define it on the $i$-dimensional faces by the following way. Notice that each $i$-dimensional face of $C\smallsetminus B$ is a convex polyhedron in the flat $F_C$ homeomorphic to a disk $D^{i}$, and we have already defined $\mathcal Q$ on its boundary, which is homeomorphic to $S^{i-1}$. We choose a basepoint $x_0\in S^{i-1}$ and notice that any point $x\in D^i$ can be uniquely presented as $x=(1-t)x_0+t\theta$ with $\theta\in S^{i-1}$ and $t\in[0,1]$. We define $\mathcal Q(x)=(1-t)\mathcal Q(x_0)+t\mathcal Q(\theta)$.
	
	A similar construction applies if $f^{(S)}(C)\subset (-\infty, a)$. 
	
	Finally, if $C$ intersects both $X_I$ and its complement non-trivially, then we can divide $C$ into three subsets $C\cap \big(f^{(S)}\big)^{-1}\big((-\infty, a)\big)$, $C\cap \big(f^{(S)}\big)^{-1}\big((b, \infty)\big)$ and $C\cap X_I$. 
	By the above, we can define $\mathcal{Q}$ on each of these subsets and piecing these together we have defined $\mathcal{Q}$ on~$C$. 
	
	Arguing as in~\cite[Lem.\,4.1]{ABDDY}, we conclude that there exists $k\ge1$ such that the map $\mathcal{Q}$ is $kt+k$ Lipschitz on $\big(f^{(S)}\big)^{-1}\big((b,b+t)\big)$ and $\big(f^{(S)}\big)^{-1}\big((a-t,a)\big)$, for $t\in (0, \infty)$. 

\medskip
	With this in mind, we are ready to prove our main results concerning various Dehn functions for $G_L(S)$. Recall that the key properties of  the space  $Z_L^{(S)}=\big(f^{(S)}\big)^{-1}(\frac12)$  were given at the end of subsection~\ref{subsec:morse}.
	
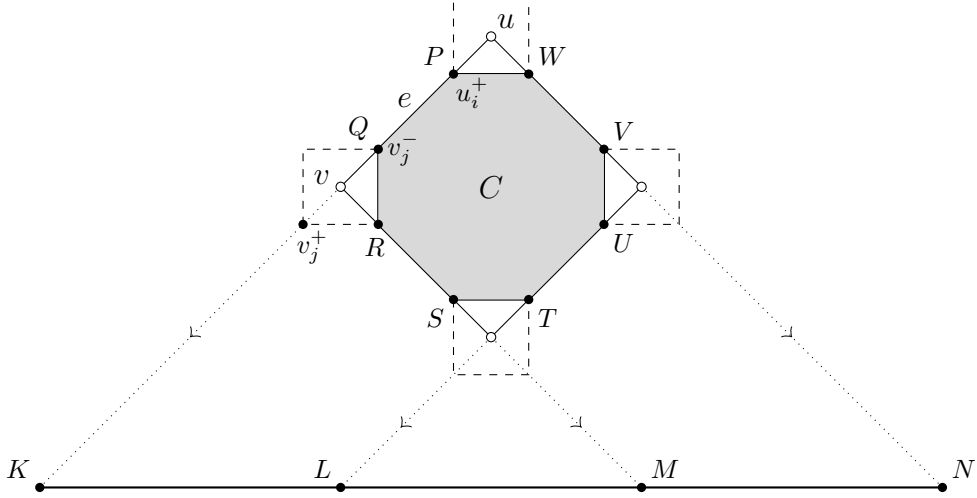
\begin{figure}
\begin{center}
\begin{tikzpicture}[scale=0.5]
\begin{scope}
\draw (0,4)--(-4,0)--(0,-4)--(4,0)--(0,4);

\begin{scope}[dotted]
\draw[->-=0.5] (-4,0)--(-12,-8); \draw[->-=0.5] (4,0)--(12,-8);
\draw[->-=0.6] (0,-4)--(-4,-8);  \draw[->-=0.6] (0,-4)--(4,-8);
\end{scope}
\draw[thick] (-12,-8)--(-4,-8)--(4,-8)--(12,-8);

\draw[thick] (-1,3)--(-3,1)--(-3,-1)--(-1,-3)--(1,-3)--(3,-1)--(3,1)--(1,3)--(-1,3);
\fill[black!15] (-1,3)--(-3,1)--(-3,-1)--(-1,-3)--(1,-3)--(3,-1)--(3,1)--(1,3)--(-1,3);
\draw (0,0) node {$C$};

\fill[white] (0,4) circle (3.5pt); \fill[white] (0,-4) circle (3.5pt); 
\draw (0,4) circle (3.5pt); \draw (0,-4) circle (3.5pt); 
\fill[white] (4,0) circle (3.5pt); \fill[white] (-4,0) circle (3.5pt);
\draw (4,0) circle (3.5pt); \draw (-4,0) circle (3.5pt);

{\footnotesize
\fill (-1,3) circle (3.5pt) [above left=2pt] node {$P$}; \draw (-0.5,2.5) node {$u_i^+$}; 
\fill (-1,-3) circle (3.5pt) [below left=2pt] node {$S$}; 
\fill (1,-3) circle (3.5pt) [below right=2pt] node {$T$}; \fill (1,3) circle (3.5pt) [above right=2pt] node {$W$}; 
\fill (-3,1) circle (3.5pt) [above left=2pt] node {$Q$} [right=2pt] node {$v_j^-$}; 
\fill (-3,-1) circle (3.5pt); \draw (-3.1,-1.6) node {$R$}; 
\fill (3,-1) circle (3.5pt) [below right=2pt] node {$U$}; \fill (3,1) circle (3.5pt) [above right=2pt] node {$V$}; 
\fill (-12,-8) circle (3.5pt) [above left=2pt] node {$K$}; \fill (-4,-8) circle (3.5pt) [above left=2pt] node {$L$}; 
\fill (4,-8) circle (3.5pt) [above right=2pt] node {$M$}; \fill (12,-8) circle (3.5pt) [above right=2pt] node {$N$};
\fill (-5,-1) circle(3.5pt); \draw (-4.75,-1.6) node {$v_j^+$};
}
\draw[dashed] (-1,3)--(-1,5)--(1,5)--(1,3); \draw[dashed] (-1,-3)--(-1,-5)--(1,-5)--(1,-3);
\draw[dashed] (-3,1)--(-5,1)--(-5,-1)--(-3,-1); \draw[dashed] (3,1)--(5,1)--(5,-1)--(3,-1);

\draw (-2.3,2.3) node {$e$};
\draw (0.4,4.4) node {$u$}; \draw (-4.5,0.25) node {$v$}; 
\end{scope}
\end{tikzpicture}
\end{center}
\caption{\label{fig:pushing}The retraction $\mathcal Q$ applied to a square $C$ with the neighbourhood $B$ removed. 
It maps $PQ\to K$, $RS\to M$, $UT\to L$, $WV\to N$, $QR$ onto $KM$, $ST$ onto $ML$, $UV$ onto $LN$, $PW$ and the whole gray octagon onto the segment $KN\subset X_I$.}
\end{figure}

	\begin{theorem}\label{thm:boundingfill}
		Let $L$ be a finite connected flag complex with no local cut points and $S\subset \Z$. 
		Suppose that $H_1(L) = 0$. 
		Let $\FA_H$ be the homological Dehn function for $H = \pi_1(L)$. 
		Then the homological filling function of $Z_L^{(S)}$ is bounded above by both $n^2\cdot\FA_H(n^2)$ and $n^4\cdot\FA_H(n)$. 
	\end{theorem}
	\begin{proof}
		Let $\gamma$ be a loop in $Z = Z_L^{(S)}$ of length $n$.
		We will construct a filling $c$ of $\gamma$ in $Z$ to show that $\HArea_Z(\gamma)$ is bounded above by both $n^2\cdot\FA_H(n^2)$ and $n^4\cdot\FA_H(n)$. 
		
		We have a disk filling $\pi\colon\Delta\to X_L^{(S)}$ for $\gamma$ in $X_L^{(S)}$ from \Cref{prop:controldehn}. 
		This filling satisfies the following properties:
		\begin{enumerate}
			\item $\pi(\Delta)$ is contained in $X^{(2)}_c  $, 
			\item $\pi^{-1}(X_c^{(0)})$ consists of at most $n^2/2$ vertices of $\Delta$; the other vertices of $\Delta$ are sent to $f^{-1}(1/2)$, 
			\item If $v  \in \pi^{-1}(X_c^{(0)})$, then $|\Lk(v, \Delta)| \leq 3n$, 
			\item $\pi(\Delta)$ is contained in the sublevel set $f^{-1}([-n/4, 1 + n/4])$. 
		\end{enumerate}
		
		We will replace this disk filling with a Lipschitz $2$-chain filling $\beta$ in $X_L^{(S)}\smallsetminus B$ as follows. There will be two types of $2$-cells of $\beta$. The first one comes from the link of each  vertex $v$ in  $\pi^{-1}(X_c^{(0)})$. To define it, let $\gamma_v$ be the $\Lk(v, \Delta)$ itself viewed as a loop.
		The image $\pi(\gamma_v)$ gives us a loop in $\Lk(\pi(v), X_L^{(S)})$.
		We can find a $2$-chain filling $\beta_v$ for this loop in $\Lk(\pi(v), X_L^{(S)})$ of mass bounded by $K\FA_H(|\gamma_v|)$ for some $K$. The sum of $\beta_v$ for all $v\in \pi^{-1}(X_c^{(0)})$ forms the first type of $2$-cells of $\beta$. The second type of $2$-cells of $\beta$ are the truncated $2$-cells of the original filling $\pi(\Delta)$.
		
		Thus we obtained a  Lipschitz filling $\beta$ for $\gamma$ in $X_L^{(S)}\smallsetminus B$. 
		Each $2$-cell of the original filling contributes mass bounded by a constant $K'$. 
		Thus we obtain a Lipschitz filling of mass bounded by $K'n^2 + \sum_v \mass(\beta_v)$, this is in turn bounded by $K'n^2 + \sum_v K\FA_H(|\gamma_v|)$. 
		
		We can push this filling to $Z$. 
		Since the filling $\Delta$ only contains vertices of heights bounded by $n/4<n$ , the pushing map $\mathcal{Q}$ is $kn + k$ Lipschitz. 
		Thus we see by \Cref{lem:lipfillingconstant} that we obtain a Lipschitz filling in $Z$ of mass bounded by $(kn+k)^2(K'n^2 + \sum_v K\FA_H(|\gamma_v|))$. 
		
		Note that since each $2$-cell of the diagram has four interior vertices we see that $\sum_v |\gamma_v|$ is bounded by $4$ times the number of $2$-cells in the diagram. 
		Thus $\sum_v |\gamma_v|\leq 4\cdot n^2/2=2n^2$. 
		Using superadditivity of $\FA_H$ we can obtain an upper bound of the form $(kn+k)^2(K'n^2 + K\FA_H(2n^2))$. 
		Since $\max(f, g)\preceq f+g\preceq 2\max(f, g)$ we obtain $f+g\simeq \max(f, g)$. 
		Since $\FA_H(n^2)\succeq n^2$, we obtain that $\FA_Z\preceq n^2\FA_H(n^2)$. 
		
		On the other hand, note that $\FA_H(|\gamma_v|)\leq \FA_H(3n)$ and there are at most $n^2/2$ vertices so the  sum $\sum_v \FA_H(|\gamma_v|)$ is bounded by $n^2\FA_H(3n)/2$. 
		Thus, $\FA_Z(n)$ is bounded by $(kn+k)^2(K'n^2 + Kn^2\FA_H(3n)/2)$.
		As above, we can now obtain the upper bound $\FA_Z(n)\preceq n^4\FA_H(n)$. 
		
		Since $\FA_Z(n)$ is bounded by both $n^2\FA_H(n^2)$ and $n^4\FA_H(n)$ we obtain the desired result. 
	\end{proof}
	
	\begin{corollary}\label{cor:hyplinks}
		Let $L$ be as in \Cref{thm:boundingfill}.
		If $\pi_1(L)$ is a hyperbolic group, then the homological Dehn function for $G_L(S)$ is bounded above by $n^4$. 
	\end{corollary}
	\begin{proof}
		The homological filling function of a hyperbolic space is bounded above by a linear function by \Cref{prop:boundbyDehn}. 
		Thus we can use \Cref{thm:boundingfill} with the upper bound given by $n^2\FA_H(n^2)$ to obtain the result.
	\end{proof}
	
	Later we will require the Dehn function of $G_L(S)$ when $S$ is finite. In this case we obtain the following. 
	
	\begin{theorem}\label{thm:Dehnbound}
		Let $L$ be a flag complex with no local cut points and $S$ be a finite set of integers.
		Let $\delta_H$ be the Dehn function for $H = \pi_1(L)$. 
		Then $G_L(S)$ is finitely presented and the Dehn function of $G_L(S)$ is bounded above by $n^4\cdot\delta_H(n)$. 
	\end{theorem}
	\begin{proof}
		We first find a simply connected space upon which $G_L(S)$ acts freely, properly, cellularly and cocompactly. 
		Let $V$ be the set of vertices $v\in X_L^{(S)}$ such that $\Lk(v, X_L^{(S)}) = \S(\widetilde{L})$.
		Let $D$ be the $1/4$ neighbourhood of the vertices in $V$ in the induced $\ell^1$-metric.  
		Let $Y_L^{(S)}  = X_L^{(S)}\smallsetminus D$. 
		Since $G_L^{(S)}$ acts on $X_L^{(S)}$ cellularly and freely away from $V$, we see that $G_L^{(S)}$ acts cellularly and freely on $Y_L^{(S)}$. 
		Moreover, $G_L(S)$ acts cocompactly on $Y_I = X_I\cap Y_L^{(S)}$ for any compact interval $I$. 
		
		Let $a, b$ be integers such that $S\subset [a-1/2, b+1/2] = I$.
		We will show that $Y_I$ is simply connected, in doing so we will give an appropriate upper bound on the filling function $\delta_{Y_I}$. 
		The theorem will then follow since the Dehn function of $G_L(S)$ is equivalent to the filling function of $Y$. 
		
		Let $\gamma$ be a loop in $Y_I$ of length $n$. 
		We can find a filling for $\gamma$ in $X_L^{(S)}$ as in \Cref{prop:controldehn}.
		We now proceed as in the proof of \Cref{thm:boundingfill}. 
		At the vertices in $V$ we can fill the loop $\gamma_v = \Lk(v, \Delta)$ in $\S(\widetilde{L})$ with a disk, this filling has size bounded by $\delta_H(|\gamma_v|)$. 
		Since all vertices not in $Y$ are in $V$, we can push this filling to $Y_I$ using the map $\mathcal{Q}$. 
		
		The pushed filling, as in \Cref{thm:boundingfill}, has size bounded by $(kn+k)^2(K'n^2 + Kn^2\delta_H(3n)/2)$ and thus we have $\delta_{Y_I}\preceq n^4\delta_H$. 
	\end{proof}

	One can also obtain the other bound from \Cref{thm:boundingfill}, but it requires using the superadditive closure of the Dehn function, i.e.\ the bound will look like $n^2\cdot\bar\delta_H(n^2)$. Since we will not use this result, we leave this added technicality to the reader. 
	
	Now that we have an upper bound it will be useful to gain a similar lower bound using the topology of $L$. 
	To do this we will study a Lipschitz map $Z\to \Lk(v, X)$ for appropriately chosen $v$. 
	
	\begin{lemma}\label{lem:Lipretract}
		Let $v$ be a vertex of $X = X_L^{(S)}$ such that $f^{(S)}(v)>\frac{1}{2}$ (resp. $<\frac{1}{2}$). Then there is a Lipschitz map $L\colon Z_L^{(S)}\to \Lk(v, X)$ such that $L\circ\mathcal{Q}(x) = x$ for all $x\in \Lk_{\downarrow}(v, X)$ (resp. $\Lk_{\uparrow}(v, X)$).
	\end{lemma}
	\begin{proof}
		By the definition of $Z_L^{(S)}$ we see that $v\notin Z_L^{(S)}$. 
		Since $\Lk(v, X)$ can be identified with the set of initial segments of geodesics at $v$ there is a map $r\colon X\smallsetminus\{v\}\to \Lk(v, X)$ given by sending a point $x$ to the initial segment of the unique geodesic from $v$ to $x$.
		
		We will show that when restricted to $Z = Z_L^{(S)}$, the map $r$ is locally Lipschitz. 
		Since the metric on $Z$ is defined using lengths of paths we see that locally Lipschitz implies Lipschitz. 
		
		Also since a path in $Z$ is a path in $X$, and the distance in $X$ is coming from minimizing lengths of paths, we see that $d_X(x, y)\leq d_Z(x, y)$. 
		Thus it suffices to prove the claim for the metric $d = d_X$ restricted to $Z$. 
		
		To show that the map is locally Lipschitz we will proceed as follows. 
		Since $v\notin Z$, there is a constant $c>0$ such that for any $x\in Z = Z_L^{(S)}$, we have $d(x, v)\geq c$.
		
		Now let $x, y$ be points of $Z$ such that $d(x, y)\leq \epsilon$ for some fixed $\epsilon$ to be specified later. 
		Let $d_x = d(x, v)$ and $d_y = d(y, v)$. 
		Consider the geodesic triangle $\Delta(v, x, y)\subset X$ and the comparison triangle $\bar{\Delta}(\bar{v}, \bar{x}, \bar{y})\subset \R^2$. 
		Let $\theta$ be the angle in $\Delta$ at $v$ and let $\bar{\theta}$ be the angle in $\bar{\Delta}$ at $\bar{v}$. 
		The CAT(0) inequality ensures that $\theta\leq\bar{\theta}$. 
		
		The euclidean law of cosines ensures that if $\frac{d(x, y)}{d_x}$ is sufficiently small we have $\bar{\theta}\leq \frac{\pi}{4}$, thus pick $\epsilon$ to achieve this bound.
		From this we can deduce that $d(r(x), r(y))\leq \theta$. (Recall that we defined the distance between two points in the link as the angle between the corresponding initial geodesic segments.) 
		Let $\bar{\alpha}$ be the angle of $\bar{\Delta}$ at $\bar{x}$. 
		The law of sines gives the equality $\frac{d(x, y)}{\sin(\bar{\theta})} = \frac{d_y}{\sin(\bar{\alpha})}$. 
		Thus we get the inequality $\sin(\bar{\theta})\leq \frac{d(x, y)}{d_y}\leq\frac{d(x, y)}{c}$. 
		Since $\sin$ is a Lipschitz function and increasing on the interval $[0, \frac{\pi}{4}]$, we obtain a $k$ such that $\theta\leq k\sin(\theta)\leq k\sin(\bar{\theta})\leq k\frac{d(x, y)}{c}$. 
		Thus, the map $r$ is locally Lipschitz when restricted to $Z$ and is thus globally Lipschitz. 	 
	\end{proof}

	\begin{theorem} \label{thm:lowerbound}
		Let L be a flag complex with no local cut points with $H_1(L)=0$.
		If $S$ is a proper subset of $\Z$, then the homological filling function of $Z_L^{(S)}$ is bounded below by the homological Dehn function for $H = \pi_1({L})$. 
	\end{theorem}
	\begin{proof}
		We have a Morse function $f^{(S)}\colon X_L^{(S)}\to \R$ and $G_L(S)$ acts on $X_L^{(S)}$ preserving the level sets. 
		Moreover, $G_L(S)$ acts properly freely and cocompactly on each non-integer level set. 
		This shows that $G_L(S)$ is isomorphic to $G_L(T)$ if $S = T+k$ for some $k\in \Z$. 
		This allows us to assume that $1\notin S$. 
		
		Recall that $Z = Z_L^{(S)}$ is the level set at height $\frac{1}{2}$. 
		Let $v$ be a vertex at height $1$. 
		Since $1\notin S$ we have that the descending link $U$ of $v$ is a copy of $\widetilde{L}$ and $\Lambda\vcentcolon=\Lk\bigl(v, X_L^{(S)}\bigr) = \S(\widetilde{L})$.
		Recall that by~\cite[Prop.\,6.9]{Lea}, $H$ is isomorphic to $\pi_1(\S(L))$, and by~\cite[Cor.\,7.2]{Lea}, $\S(\widetilde L)$ is the universal cover for $\S(L)$. 
		Hence both $U=\widetilde L$ and $\Lambda=\S(\widetilde L)$ have a proper, free and cocompact $H$ action and thus have homological functions equivalent to $\FA_H$. 
		Let $\gamma$ be a $1$-cycle in $U$, realizing $\FA_{\widetilde{L}}(n)$. 
		
		Since we have a map $\Lambda\to X_L^{(S)}$, we can view the cycle $\gamma$ as a subset of $X_L^{(S)}$ in the $\frac{1}{4}$ neighbourhood of $v$. 
		We can push the $1$-cycle $\gamma$ to $Z$ a cell at a time. 
		Each $1$-cell of $\gamma$ is in a square. 
		By taking instead the diagonal at height $\frac{1}{2}$, we obtain a $1$-cycle $\alpha$ in $Z$ such that $|\alpha| = |\gamma|$. 
		Pick a minimal filling $\beta$ for $\alpha$ in $Z_L^{(S)}$.
		
		Since $Z\subset X_L^{(S)}\smallsetminus \{v\}$ we have a Lipschitz map $r\colon Z\to \Lk(v, X_L^{(S)})$ from \Cref{lem:Lipretract}. 
		Let $k$ be the corresponding Lipschitz constant. %
		
		We see that under the map $r$ the $1$-cycle $\alpha$ is sent to the $1$-cycle $\gamma$ in $\Lk(v, X_L^{(S)})$. 
		
		Since $r$ is $k$-Lipschitz, \Cref{lem:lipfillingconstant} shows that $\mass(r(\beta))\leq k^2\mass(\beta)$. 
		
		By \Cref{thm:FF}, we can find a cellular $2$-chain $P(r(\beta))$ with boundary $\gamma$ and $|P(r(\beta))|\le K \mass(r(\beta))$ for some constant $K\ge0$ depending just on $\Lambda$. 
		Arguing as in the proof of  \Cref{thm:lipcellequivalent}, we get $\mass(\beta) \leq D|\beta|$ for some $D\in \R$ and thus we have:
		\begin{multline*}
		\FA_{\widetilde{L}}(n) = \HArea_{\Lambda}(\gamma)\leq |P(r(\beta))|\le K\mass(r(\beta))\leq K\cdot k^2\mass(\beta)\\
		= K\cdot k^2\cdot D|\beta|\leq K\cdot k^2\cdot D\FA_Z(|\alpha|)\leq K\cdot k^2\cdot D\FA_Z(n),
		\end{multline*}
		which shows that $\FA_H\preceq \FA_Z$.
	\end{proof}
	
	Combining \Cref{thm:boundingfill} and \Cref{thm:lowerbound} we obtain the following Corollary.
	
	\begin{corollary}\label{cor:sameasH}
		Let $H$ be a finitely presented perfect group with $\FA_H(n)\simeq n^4\FA_H(n)$.
		Let $L$ be a flag complex with no local cut points such that $\pi_1(L)  = H$.
		If $S\neq \Z$, then the homological Dehn function of $G_L(S)$ is $\simeq$ equivalent to $\FA_H$. \qed
	\end{corollary}
	
	Given $H$ satisfying the requirements of the above Corollary, we can vary the set $S$ and obtain uncountably many groups of type $FP_2$ with homological Dehn function $\FA_H$.
	In fact, by \cite[Th.\,1.1]{KLS}, we can obtain uncountably many quasi-isometry classes of such groups. 
	
	By improving the homological properties of $L$, we can create uncountably many groups of type $FP_n$ or $FP$ with such a homological filling function. 
	For instance, we can prove the following theorem:
	
	\begin{theorem}\label{thm:uncountablymanyexp}
		For each $k\ge1$, there exist uncountably many groups of the form $G_L(S)$ of type $FP$ whose homological filling function is the iterated exponential $\exp^{(k)}(n)$. 
	\end{theorem}
	Recall that the iterated exponential function is defined as $\exp^{(k)}(n)\vcentcolon=2^{\exp^{(k-1)}(n)}$, with $\exp^{(0)}(n)\vcentcolon=n$. To prove this theorem we will need several group theoretic constructions.
	
	Let $A$ be the group given by the following presentation
	\[
	A\vcentcolon=\la\, a,b\mid a[b,a][b^2,a^2]\dots[b^{8},a^{8}], \quad b[a^{9},b^{9}]\dots[a^{11},b^{11}]\,\ra,
	\]
	where $[a, b] = a^{-1}b^{-1}ab$.
	\begin{proposition}\label{prop:hkl}
		The group $A$ satisfies the $C'(1/6)$ small cancellation condition, and 
		its presentation complex is aspherical and acyclic.
	\end{proposition}
	\begin{proof}
		The longest piece in the above presentation is $a^{10}b^{10}$ which has length $20$. 
		The length of the first relation is $145$ and the length of the second relation is $121$, thus $A$ satisfies condition $C'(1/6)$ and hence the presentation $2$-complex is aspherical, see~\cite[Th.\,13.3]{Olsh}.
		
		To see that the presentation $2$-complex is acyclic, one can look at the cellular chain complex.
		First note that the first homology vanishes as can be seen by abelianizing the presentation.
		Since the presentation complex is a $2$-complex, the second homology is free abelian. 
		The Euler characteristic of this $2$-complex is: $\chi=\#\text{($0$-cells)}-\#\text{($1$-cells)}+\#\text{($2$-cells)}=1-2+2=1$. On the other hand, $\chi=b_0-b_1+b_2$, where $b_i$ is the $i$-th Betti number, and hence $b_2 = 1 - b_0 + b_1 = 1 - 1 + 0 = 0$, which means that the second homology vanishes. 
		Once again since the space is a $2$-complex we see that it is acyclic. 
	\end{proof}
	
	Now let
	\[
	B_m\vcentcolon=\la\, x_0,x_1,\dots, x_m\mid x_{i}^{-1}x_{i-1}x_i=x_{i-1}^2\text{ for } i=1,\dots,m\,\ra
	\]
	for some fixed $m$, and let $Q$ be the amalgamated free product of $A$ and $B_m$ identifying the cyclic subgroups $\la b\ra$ and $\la x_m\ra$:
	\[
	Q\vcentcolon=A\textstyle\bigast\limits_{b=x_m}B_m.
	\]

	\begin{lemma}\label{lem:acycasphere}
		The presentation complex for $Q$ is aspherical and acyclic.
	\end{lemma}
	\begin{proof}
		The group $B_m$ can be written iteratively as an HNN extension $B_k = (B_{k-1})\text{\large$\ast$}_{\Z}$, $k=1,\dots,m$, with $B_0 = \Z$. 
		Theorem 4.6 of \cite{aspherical} (taking into account a remark on p.~2 therein) tells us that the fundamental group of a graph of groups with each vertex group aspherical and each edge group free, is aspherical. Thus the presentation $2$-complex of $B_m$ is itself aspherical. 
		By \Cref{prop:hkl}, we have that the presentation $2$-complex of $A$ is aspherical. 
		Applying Theorem~4.6 of \cite{aspherical} again, we see that the presentation complex for $Q$ is aspherical. 
		
		To see that the complex is acyclic, we observe that its first homology vanishes and that the presentation is balanced (i.e.\ the number of relations equals the number of generators). Hence we 
		can apply the Euler characteristic argument as in the proof of~\Cref{prop:hkl} above to conclude that the presentation complex for $Q$ is acyclic. 
	\end{proof}
	
	\begin{lemma}\label{lem:homfilling}
		The homological Dehn function of $Q$ is $\simeq$ equivalent to  $\exp^{(m)}(n)$. 
	\end{lemma}
	\begin{proof}
		We first appeal to Proposition~3.5 of~\cite{brick} to see that the Dehn function of $Q$ is $\simeq \exp^{(m)}(n)$. 
		To apply this result it suffices to show that in $A$ we have $n = |b^n|$ and in $B_m$ we have $n = |x_m^n|$.

		For the first we can use Dehn's algorithm for small cancellation groups \cite[Sec.~VI]{LyndonSchupp}. 
		It states that if $w$ is not a geodesic word in a small cancellation group, then there is a relator $r$ such that $w$ contains a subword which is at least half of a conjugate of $r$. 
		We can see from the presentation of $A$ that $b^n$ does not contain half a relator for any $n$. 
		Thus, $b^n$ is a geodesic word and $|b^n| = n$. 
		
		For the inequality in $B_m$, we use the following homomorphism. 
		Define $\phi\colon B_m\to \Z$ by $x_i\mapsto 0$ if $i\neq m$ and $x_m\mapsto 1$. 
		Now suppose $|x_m^n| \leq n$, let $w$ be a geodesic for $x_m^n$. 
		Since $\phi$ reduces word length we see that $|\phi(w)|\leq |w| \leq n$. 
		However, we also see that $\phi(w) = n$ and so $n\leq |\phi(w)|$. 
		Thus $|w| = n$. 
		
		Since $\exp^{(m)}(n)$ is superadditive, \Cref{prop:boundbyDehn} shows that $\FA_H$ is bounded above by $\exp^{(m)}(n)$.
		We must now obtain a similar lower bound. 
		The lower bound on the Dehn function of $B_m$ comes from embedded diagrams in the universal cover of a presentation complex for $B_m$. 
		It can be shown \cite[Ex.~7.2.11]{BriChap}, that there is a sequence of  words $w_n$ such that $|w_n|\simeq n$ and $w_n$ bounds an embedded diagram with area $\simeq \exp^{(m)}(n)$. 
		Thus, appealing to the proof of  \Cref{thm:embeddeddisk}, we see that $\HArea(w_n)\simeq \exp^{(m)}(n)$ and hence $\exp^{(m)}(n)\preceq \FA_Q(n)$. 
	\end{proof}
	
	\Cref{thm:uncountablymanyexp} follows immediately from the following theorem. 
	
	\thmiteratedexp
	\begin{proof}
		Let $K$ be the presentation complex for $Q$ from above. 
		We can find a homotopy equivalent complex $L$ which is flag and has no local cut points. 
		Since $L$ is acyclic and aspherical by \Cref{lem:acycasphere}, $G_L(S)$ is of type $FP$. 
		By \Cref{thm:boundingfill} and \Cref{lem:homfilling}, the homological filling function of $G_L(S)$ is bounded above by $n^4\exp^{(m)}(n)\simeq \exp^{(m)}(n)$. 
		\Cref{thm:lowerbound} gives a lower bound on the homological filling function of $\exp^{(m)}(n)$ as long as $S\neq \Z$. 
		Thus, by varying $S\subsetneq\Z$ we obtain from \cite[Th.\,5.2]{KLS} that there are uncountably many quasi-isometry classes of such groups. 
	\end{proof}
	
	\begin{corollary}\label{cor:depends}
		The homological filling function is not independent of $S$. 
	\end{corollary}
	\begin{proof}
		In the above proof we see that if $S = \Z$, then the homological filling function is bounded by $n^4$. 
		However, when $S\neq \Z$, the homological filling function is an iterated exponential. 
	\end{proof}

	\section{Groups with quartic homological filling function}\label{sec:n4bounds}
	In this section we will give a construction of an aspherical, acyclic flag complex $L$ such that $G_L(S)$ has homological Dehn function $n^4$ for any set $S$.


	Let $K_A$ be the presentation $2$-complex of $A$, 
	and let $K$ be the second barycentric subdivision of $K_A$. 
	\begin{proposition}
		$K$ is a flag simplicial complex with no local cut points.
	\end{proposition}
	\begin{proof}
		We only need to check that the complex satisfies the no local cut points condition. By~\cite[Prop.\,6.1]{Lea}, this requires that the link of every vertex be connected and not equal to a point. 
		There are two types of vertices in the complex $K_A$, the original $0$-cell from the presentation $2$-complex and the vertices coming from barycenters. 
		
		The vertices from barycenters are all contained in $2$-cells. 
		If they are from the interior of the $2$-cell then their link is isomorphic to $S^1$. If they are barycenters of edges then their link is the suspension of a discrete non-empty set and is thus connected. 
		
		The link of the original $0$-cell $v$ is connected, non-empty and not a point. 
		Indeed, since all four of the words $ab, ab^{-1}, a^{-1}b$ and $a^{-1}b^{-1}$ appear as subwords of the relations we see that the link of this $0$-cell in $K_A$ is connected. 
		After taking the second barycentric subdivision the link of the $0$-cell corresponding to $v$ is isomorphic to the second barycentric subdivision of the link of $v$ and hence is connected, non-empty and not a point. 
	\end{proof}
	
	Consider a flag simplicial complex $F$ depicted in~\Cref{fig:flag}. It was shown in~\cite[I.2.5.2]{BRS} that the Bestvina--Brady kernel $BB_F$ of $F$ has the Dehn function $\simeq$ equivalent to $n^4$.
	Let $L$ be the simplicial complex obtained by identifying an edge of $K$ with the bold edge of $F$.
	
	\begin{figure}
		\centering
		\begin{tikzpicture}[scale=0.8]
		\filldraw[black] (3,0) circle (1pt);
		\filldraw[black] (-3,0) circle (2pt);
		\filldraw[black] (0,3) circle (2pt);
		\filldraw[black] (0,-3) circle (2pt);
		\filldraw[black] (0,-1) circle (2pt);
		\filldraw[black] (0,1) circle (2pt);
		\filldraw[black] (1,0) circle (2pt);
		\filldraw[black] (-1,0) circle (2pt);
		\draw[thick] (3,0) -- (-3,0)--(0, -3)--(3, 0)--(0, 3) -- (-3, 0);
		\draw[thick] (1, 0) -- (0, 1) -- (-1, 0) -- (0, -1) -- (1, 0);
		\draw[thick] (3, 0) -- (0, 1) -- (-3, 0) -- (0, -1) -- (3, 0);
		\draw[thick] (0, 3) -- (0, 1);
		\draw[thick] (0, -3) -- (0, -1);
		\draw[line width=0.8mm
		] (3,0) -- (0,-3);
		\end{tikzpicture}
		\caption{A complex $F$ for which the Bestvina--Brady kernel $BB_F$ has a quartic Dehn function.}
		\label{fig:flag}
	\end{figure}
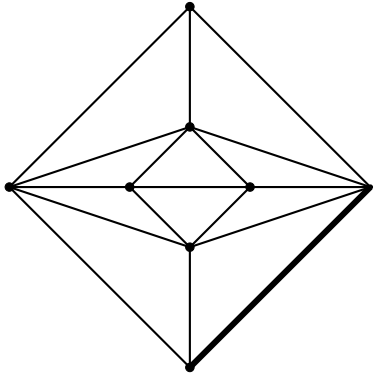
	
	The complex $L$ has no local cut points. 
	Also $\pi_1(L) = A$ and $L$ is aspherical and acyclic as it is homotopy equivalent to $K_A$. 
	We will now study the homological filling function of $G_L(S)$. 
	
	\begin{theorem}\label{thm:n4} Let $L$ be as above. Then the homological Dehn function of $G_L(S)$ is $\simeq$ equivalent to $n^4$ for any subset $S\subset \Z$. 
	\end{theorem}
	\begin{proof}
		As $A$ is a hyperbolic group, the universal cover of $L$ is a hyperbolic metric space.
		We obtain an upper bound of $n^4$ for the homological filling function by \Cref{cor:hyplinks}.
		
		To obtain the lower bound we use the simplicial complex $F$ from \Cref{fig:flag}. 
		From \cite{Lea}, given a flag complex with no local cut points $P$, we have a presentation for $G_P(S)$ of the form
		\[		
		\langle E(P)\mid R_1\cup R_2\cup R_3\rangle,
		\]
		where
		\begin{itemize}
		    \item $E(P)$ is the set of directed edges of $P$. 
		    \item $R_1$ is the set of words of the form $e\bar e$ for each $e\in E(P)$. 
		    \item $R_2$ is the set of words of the form $abc$ and $a^{-1}b^{-1}c^{-1}$for each directed triangle $(a,b,c)$ in $P$. 
		    \item $R_3$ is the set of words of the form $a_1^na_2^n\dots a_l^n$ where $n\in S$ and the edge loops  $(a_1, a_2, \dots, a_l)$ normally generate $\pi_1(P)$. 
		\end{itemize}

		Now let $e$ be the bold edge in \Cref{fig:flag} and $e'$ be the edge of $K$ which is identified with $e$ to create $L$. 
		Since $F$ is simply connected we can assume that all the loops for the relators in $R_3$ are contained in $K$. 
		
		Writing down the presentation for $G_K(S)\ast_{e' = e}G_F(S)$ and canceling the generator $e$ using the relation $e = e'$ we obtain exactly the presentation for $G_L(S)$ given above. 
		Thus we can see that $G_L(S) = G_K(S)\ast_{e' = e}G_F(S)$ for all $S$. 
		
		Since the flag complex $F$ is simply connected, we can take $R_3$ to be empty for $G_F(S)$. 
		Thus we see that $G_F(S)$ is independent of $S$ and by \cite{DicksLeary}, we can see that $G_F(S)$ is isomorphic to $BB_F$ for all $S$.
		Thus $G_L(S) = G_K(S)\ast_{e' = e}BB_F$, for all $S$. 
		
		As we mentioned above, by \cite[I.2.5.2]{BRS}, the Dehn function for $BB_F$ is $n^4$. 
		The proof of this uses embedded diagrams to obtain the lower bound and thus we can appeal to \Cref{thm:embeddeddisk} to see that the homological Dehn function of $BB_F$ is also $n^4$. 
		
		To obtain the lower bound on $G_L(S)$ we will construct a retraction $G_L(S)\to BB_F$ and appeal to \Cref{prop:retract}. 
		In \cite[Th.\,14.1]{Lea}, it is shown that if $S\subset T$ then there is an epimorphism $G_L(S)\to G_L(T)$. 
		Thus we obtain a homomorphism  $G_L(S)\to G_L(\Z) = BB_L$. 
		Now $BB_L$ is a subgroup of $A_L$ which is the kernel of a map to $\Z$. 
		We see that the right-angled Artin group $A_L$ has a splitting $A_L = A_K\ast_{\Z^2}A_F$, where the copy of $\Z^2$ corresponds to the bold edge in \Cref{fig:flag}. 
		We obtain a retraction $A_L\to A_F$, by mapping the generators corresponding to each vertex not in $F$ to the identity. 
		This restricts to a retraction $BB_L\to BB_F$. 
		Composing with $G_L(S)\to BB_L$ we obtain a retraction $G_L(S)\to BB_F$ and thus, by \Cref{prop:retract}, we conclude that the homological Dehn function of $G_L(S)$ is bounded below by $n^4$.
	\end{proof}
	
	\thmquartic

	\begin{proof}
		Taking $L$ as above, we see from \Cref{thm:lea} that $G_L(S)$ is of type $FP$ as $L$ is aspherical and acyclic. 
		Moreover, $G_L(S)$ has homological Dehn function $n^4$ for any set $S$. 
		Thus, by \cite[Th.\,5.2]{KLS}, we obtain uncountably many quasi-isometry classes of such groups. 
	\end{proof}

	\section{The word problem}\label{sec:wordproblem}
	Given a finitely generated group $G$ one can ask if there is an algorithm that takes as input words in the finite generating set and outputs whether or not they represent the trivial element. 
	For finitely presented groups such an algorithm exists if and only if the Dehn function is bounded by a recursive function. 
	For a detailed introduction to decision problems in group theory we refer the reader to \cite{Mil}.
	
	For groups of type $FP_2$, one may suspect that the word problem is solvable if and only if the homological Dehn function is bounded by a recursive function. 
	However, there are at most countably many groups with a solvable word problem. 
	Thus,  \Cref{thm:n4} shows that there are groups of type $FP_2$ (and even of type $FP$) whose homological filling function is $n^4$ but the word problem is unsolvable. 
	Thus we immediately obtain: 
	\thmsubrecursunsolv
	
	For completeness we prove a criterion  showing exactly when $G_L(S)$ has solvable word problem, see~\Cref{thm:glsword} below. 
	There are two cases where the word problem is obviously solvable. 
	By~\cite{DicksLeary}, in the case that $S = \Z$ or $\pi_1(L) = 1$, we have $G_L(S) = BB_L$ which is a subgroup of a right-angled Artin group.
	Hence, $G_L(S)$  has solvable word problem in these cases. 
	The next lemma shows that in general for solvability of the word problem for $G_L(S)$, it is necessary for $\pi_1(L)$ to have a solvable word problem.  
	
	\begin{lemma}\label{lem:pi1lembeds}
		If $S\neq \Z$, then $\pi_1(L)$ embeds in $G_L(S)$. 
	\end{lemma}
	\begin{proof}
		If $S\ne\Z$, there is a vertex $v\in X_L^{(S)}$ such that $f^{(S)}(v)\notin S$. By~\cite[Corollary~9.2]{Lea}, the  stabiliser of $v$ under the action of $G_L^{(S)}$ on $X_L^{(S)}$ is isomorphic to $\pi_1(L)$.
		In particular, $\pi_1(L)$ is a subgroup of $G_L(S)$.
	\end{proof}
	
	\begin{lemma}\label{lem:finitepressolv}
		If $\pi_1(L)$ has a solvable word problem and $S$ is finite, then $G_L(S)$ has a solvable word problem. 
	\end{lemma}
	\begin{proof}
		Since $S$ is finite, we see that $G_L(S)$ is finitely presented. 
		It suffices to show that the Dehn function of $G_L(S)$ is bounded by a recursive function. 
		Let $\delta$ be the Dehn function of $\pi_1(L)$. 
		As in \Cref{thm:Dehnbound}, we see that the Dehn function of $G_L(S)$ is bounded by $n^4\delta(n)$. 
		Since $\delta(n)$ is a recursive function, we see that $n^4 \delta(n)$ is a recursive function. 
		Thus, $G_L(S)$ has a solvable word problem. 
	\end{proof}
	
	We are now ready to do the case of general $S$. 
	
	\begin{theorem}\label{thm:glsword}
		Suppose that $S\neq \Z$ and  $\pi_1(L)\neq 1$. 
		Then the group $G_L(S)$ has a solvable word problem if and only if $\pi_1(L)$ has a solvable word problem and $S$ is recursive.
	\end{theorem}
	\begin{proof}
	    The solvability of the word problem passes to subgroups and hence by \Cref{lem:pi1lembeds}, we see that if $G_L(S)$ has solvable word problem so does $\pi_1(L)$. 
	    
	    Now since $\pi_1(L)$ is non-trivial we can find an edge path $(a_1, \dots, a_l)$ which gives a non-trivial element of $\pi_1(L)$. 
	    By \cite[Lem.\,15.3]{Lea}, the set of $n$ such that $a_1^na_2^n\dots a_l^n$ is trivial in $G_L(S)$ is exactly $S$. 
	    Hence if $G_L(S)$ has solvable word problem we can compute the set of all numbers $n$ such that $a_1^na_2^n\dots a_l^n$ is trivial and thus compute $S$. 
	    Hence, if $G_L(S)$ has a solvable word problem, then $S$ is a recursive set. 
	
		We now show the other direction. Suppose that $\pi_1(L)$ has a solvable word problem and $S$ is recursive. 
		Let $w$ be a word in $G_L(S)$. 
		Since $G_L(S)$ acts properly and cocompactly on $Z$, we see that $Z$ is quasi-isometric to the Cayley graph $\Gamma$ of $G_L(S)$. 
		Since $Z$ is connected we can create a continuous quasi-isometry $q\colon\Gamma\to Z$. 
		To see this, let $v$ be a vertex of $Z$, since $G_L(S)$ acts freely on $Z$ there is a one-to-one correspondence between the orbit of $v$ and $G_L(S)$. 
		This defines the map $q$ on vertices. 
		Since $Z$ is connected, we can extend over the edge $(g, gs)$ by picking a shortest length path $g\cdot v$ to $gs\cdot v$ in $Z$. 
		
		Thus, if $w$ is trivial in $G_L(S)$, then $q(w)$ gives a loop $\gamma$ in $Z$. 
		If $K$ is the quasi-isometry constant for $q$, then $|\gamma|\leq 2K|w|$. 
		By \Cref{prop:controldehn} this loop bounds a disk in $f^{-1}([-2K|w|, 2K|w|+1])$. 
		We can now map this disk to $(X_L/BB_L)\smallsetminus V_S$. 
		The disk lives in the subspace $(X_L/BB_L)\smallsetminus V_T$, where $T = S\cap [-2K|w|, 2K|w|+1]$. 
		Thus we see that if $w$ is trivial it is also trivial in $G_L(T)$.
		Since $S$ is recursive, we can compute $T$ algorithmically. 
		However, we can now appeal to \Cref{lem:finitepressolv} to see that $G_L(T)$ has solvable word problem. 
		Hence we can decide if $w$ is trivial. 
	\end{proof}

	\section{Concluding remarks}\label{sec:con}
	
	Throughout we have studied the groups $G_L(S)$ and their homological filling functions. Most of the results have shown that there is very little dependence on $S$. For instance \Cref{cor:depends} is the only case we have shown dependence on $S$ and in this case there is only one outlier, namely, when $S = \Z$.
	This leads us to the following question: 
	\Qdependence
	This lack of dependence on $S$ allows us to construct uncountably many groups with a given homological filling function. 
	We can prove the following theorem: 
	
	\begin{theorem}\label{thm:manyisoclasses}
		Let $H$ be a group of type $FP_2$. Suppose that $\FA_H\succeq n^4$. Then there are uncountably many groups with homological Dehn function $\FA_H$. 
		Moreover, these groups can be constructed to satisfy the same homological finiteness properties as $H$. 
	\end{theorem}
	\begin{proof}
		Let $G_L(S)$ be the uncountably many groups constructed from \Cref{thm:n4}. Let $H_L(S) =  G_L(S)\ast H$. Using the Grushko decomposition we can see that $H_L(S)\cong H_L(T)$ if and only if $G_L(S)\cong G_L(T)$, thus we obtain uncountably many isomorphism classes of groups. 
		By \Cref{prop:freeproducts} and the fact that $\FA_H\succeq n^4$ we see that $\FA_{H_L(S)} \simeq \FA_H$. 
		
		Considering the action of $H_L(S)$ on the Bass--Serre tree for the free splitting we can apply Proposition~1.1 of \cite{Bro87} and the fact that $G_L(S)$ is of type $FP$ to see that $H_L(S)$ is of type $FP_n$ if $H$ is of type $FP_n$. 
		The same reasoning applies if $H$ is of type $FP_\infty$ or $FP$. 
	\end{proof}
	
	If, in addition, $H$ is finitely presented, then we can also gain uncountably many quasi-isometry classes of such groups.

		\thmsnowflake
		
		To prove \Cref{thm:manyclasses} we will use the approach of \cite{MOW}, based on Grigorchuk's construction of the space of marked groups (for details see~\cite{Gri84,CG05,BK20}). 
		
		Given any $n\in \mathbb{N}$, let $F_n$ denote the free group with a fixed basis $\{a_1,\dots,a_n\}$.
		The {\em space of $n$-generated marked groups}, $\mathcal{G}_n$, is the set of normal subgroups of $F_n$ endowed with the topology induced from the product topology on the set $2^{F_n}$ of all subsets of $F_n$ (identified with the Cartesian product $\{0,1\}^{\aleph_0}$).
		
		Given a group $G$ and a tuple $X = (x_1, \dots, x_n)\in G^n$ such that $\{x_1, \dots, x_n\}$ generates $G$, we obtain a surjection $\phi_G\colon F_n\to G$ given by $a_i\mapsto x_i$ and we can associate to $(G,X)$ the point $\ker(\phi_G)\in \mathcal{G}_n$. 
		This gives a surjection from pairs $(G, X)$ to $\mathcal{G}_n$.
		When we refer to $(G, X)$ as a marked group we will mean its image in $\mathcal{G}_n$ under the above surjection. 
		
		The topology on $\mathcal{G}_n$ can also be described in terms of convergence as follows. 
		Let $(G, X)$ and $(H, Y)$ be $n$-generated marked groups. 
		Let $\phi_G\colon F_n\to G$ and $\phi_H\colon F_n\to H$ be the maps above.
		We say that $(G, X)$ and $(H, Y)$ are {\em $r$-locally isomorphic}, for some $r>0$ if $\ker(\phi_G)\cap B_{F_n}(r) = \ker(\phi_H)\cap B_{F_n}(r)$.
%
%
%
		A sequence of marked groups $(G_i,X_i)$ converges to $(G,X)$ in $\mathcal{G}_n$ if for every $r>0$ there exists $l>0$, such that $(G_i, X_i)$ and $(G, X)$ are $r$-locally isomorphic for all $i>l$.
		
		\begin{lemma}\label{lem:localfree}
			Let $(G, X)$ be a marked $m$-generated group and $(H_1, Y_1)$ and $(H_2, Y_2)$ be marked $n$-generated groups. 
			Then $(H_1, Y_1)$ and $(H_2, Y_2)$ are $r$-locally isomorphic if and only if $(G\ast H_1, X\cup Y_1)$ and $(G\ast H_2, X\cup Y_2)$ (with homomorphisms $\phi_{G\ast H_i}$ naturally defined by $\phi_G$ and $\phi_{H_i}$) are $r$-locally isomorphic. 
		\end{lemma} 
		\begin{proof}
		Identifying $F_{m+n}$ with $F_m\ast F_n$, we have the following commutative diagram:
		\begin{center}
		\begin{tikzcd}
		F_m\arrow[r,hook]{r}\arrow[d,"\phi_G"] & F_m \ast F_{n} \arrow[d,"\phi_{G\ast H_i}"]& F_n\arrow[l,hook']\arrow[d,"\phi_{H_i}"]\\
		G\arrow[r,hook] & G\ast H_i & H_i \arrow[l,hook'] 
		\end{tikzcd}
		\end{center}

			To prove the lemma in one direction, assume that $(G\ast H_1, X\cup Y_1)$ and $(G\ast H_2, X\cup Y_2)$ are $r$-locally isomorphic and suppose that $w$ is a word of length $\leq r$ in $F_n$ which is trivial under $\phi_{H_1}$. 
			Then $w$ also maps to the trivial element under the homomorphism $\phi_{G\ast H_1}$. Since $(G\ast H_1, X\cup Y_1)$ and $(G\ast H_2, X\cup Y_2)$ are $r$-locally isomorphic, the word $w$ maps to the trivial element under the homomorphism $\phi_{G\ast H_2}$ as well. 
			From the commutativity of the diagram we see that $w$ must also map to the trivial element under the homomorphism~$\phi_{H_2}$. 
			
			For the other direction, suppose that $(H_1, Y_1)$ and $(H_2, Y_2)$ are $r$-locally isomorphic. 
			We will prove that  $(G\ast H_1, X\cup Y_1)$ and $(G\ast H_2, X\cup Y_2)$ are $r$-locally isomorphic by induction on $r$.
			Note that any two marked groups are $0$-locally isomorphic. 
			Now suppose $w$ is a word of length $r$ in $F_{n+m}$ which maps to the trivial element in $G\ast H_1$. From the description of normal forms in free products we conclude that $w$ contains a subword $v$ of length $\leq r$ such that either $v$ is trivial under the homomorphism $\phi_{H_1}$ or $v$ is trivial under the homomorphism $\phi_G$. Writing $w=uvu'$ for some (possibly trivial) subwords $u,u'$, we can remove $v$ and obtain a word $w'=uu'$ of length $l<r$ which still maps to the trivial element in $G\ast H_1$.
			By induction, $w'\in B_{F_{n+m}}(l)\cap \ker(\phi_{G\ast H_2})$. 
			If $v$ was trivial under $\phi_G$, then by the commutative diagram above we see that $v$ is also trivial under the homomorphism $\phi_{G\ast H_2}$, and hence $w$ is trivial under $\phi_{G\ast H_2}$ as well.
			If $v$ was trivial under $\phi_{H_1}$, then since $(H_1, Y_1)$ and $(H_2, Y_2)$ are $r$-locally isomorphic we see that $v\in \ker(\phi_{H_2})\cap B_{F_n}(r)$. Then $v$ is trivial under the homomorphism $\phi_{G\ast H_2}$ as well, and hence $w=uvu'$ is trivial under $\phi_{G\ast H_2}$ again. 
		\end{proof}
		
		Recall, that a topological space $X$ is {\em perfect} if every point $x$ of $X$ is a limit of points in $X\smallsetminus\{x\}$. We denote as $2^\Z$ the set of all subsets of $\Z$. 
		
		\begin{lemma}\label{lem:perfectimage}
			Fix a finite flag complex $L$ with $\pi_1(L)\neq 1$. Let $n$ be the number of edges of $L$. Then the map $2^{\Z}\to \mathcal{G}_n$ given by $S\mapsto (G_L(S), E(L))$ is injective and has perfect image.
		\end{lemma}
		\begin{proof}
			Let $K_S$ be the kernel of the natural homomorphism $F_n\to G_L(S)$ where $F_n$ is the free group with basis in bijection with $E(L)$. 
			Let $S, T\subset \Z$. 
			Suppose $s\in S\smallsetminus T$. 
			Then \cite[Lem.\,14.4]{Lea} shows that we can find an element of $K_S\smallsetminus K_T$, thus showing that the map is injective. 
			
			To show that the image is perfect, pick $S\subset \Z$. 
			We consider two cases. First, suppose $S$ is infinite. Let $T_i = S\cap [-i, i]$. 
			Then there are surjections $G_L(T_i)\to G_L(S)$ and by \cite[Lem.\,3.2]{KLS}, the shortest element of the kernel has length $(i+1)\sqrt{2/(d+1)}$. We see that $(G_L(T_i), E(L))$ and $(G_L(S), E(L))$ are $r$-locally isomorphic for all $r<(i+1)\sqrt{2/(d+1)}$.
			
			In the case that $S$ is finite let $T_i = S\cup \{j\mid  j\leq -i$ or $j\geq i\}$. 
			We can then apply the same reasoning as above to see that $(G_L(T_i), E(L))$ and $(G_L(S), E(L))$ are $r$-locally isomorphic for all $r<(i+1)\sqrt{2/(d+1)}$.	
		\end{proof}

	Following the proof of \Cref{lem:perfectimage} and \Cref{lem:localfree} we obtain the following. 
	\begin{corollary}\label{cor:perfectsubsets}
		Fix a finite flag complex $L$ with $\pi_1(L)\neq 1$. Let $n$ be the number of edges of $L$. Let $G$ be a finitely presented group. Let $(G, X)\in \mathcal{G}_m$. Then the image of $2^{\Z}\to \mathcal{G}_{m+n}$ given by $S\mapsto (G\ast G_L(S), X\sqcup E(L))$ is perfect. 
		Moreover, this image has a dense subset of finitely presented groups corresponding to finite sets $S\subset \Z$.\qed
	\end{corollary}

		\begin{proof}[Proof of \Cref{thm:manyclasses}]
			Fix $L$ as in \Cref{thm:n4}. 
			Thus $G_L(S)$ is of type $FP$ and has homological Dehn function $n^4$ for every $S$.  
			Since $\FA_H\succeq n^4$ \Cref{prop:freeproducts} implies that the homological Dehn function of $H\ast G_L(S)$ is $\FA_H$ for all $S$. Then $H\ast G_L(S)$ has homological filling function $\FA_H$ by \Cref{prop:freeproducts}.  
			It follows from \Cref{cor:perfectsubsets} that $(H\ast G_L(S), X\sqcup E(L))$ is a perfect subset of $\mathcal{G}_k$ for some $k$ with a dense subset of finitely presented groups. 
			Thus we can apply Corollary 1.2 of \cite{MOW} to see that $H\ast G_L(S)$ form uncountably many quasi-isometry classes of groups. 
		\end{proof}

	\begin{corollary}\label{cor:denseset}
		There are uncountably many quasi-isometry classes of groups of type $FP$ with the homological Dehn function $n^{\alpha}$, where $\alpha = 2 \log_2({2p}/{q})$ with $p,q\ge1$ arbitrary integers such that ${p}/{q}\geq 2$. 	
		In particular, there is a countable dense subset of exponents $\alpha\in[4, \infty)$, containing all even integers $\alpha\ge4$, such that there are uncountably many groups of type $FP$ with the homological filling function $n^{\alpha}$.
	\end{corollary}
	\begin{proof}
		It was proved in \cite{snowflake} that there are groups with Dehn function $n^{\alpha}$ where $\alpha = 2 \log_2({2p}/{q})$, for arbitrary $p/q\ge1$.
		The methods used rely on embedded disks to show that the lower bound is achieved and thus we can use \Cref{thm:embeddeddisk} to see that the homological Dehn function is bounded below by $n^{\alpha}$. 
		Since $n^{\alpha}$ is superadditive, we see by \Cref{prop:boundbyDehn} that the homological filling function is also bounded above by $n^{\alpha}$. 
		The result now follows from \Cref{thm:manyclasses}, where to ensure that $n^\alpha\succeq n^4$, we must require $p/q\ge2$.
	\end{proof}

	\medskip
	It is also worth noting that the proof of~\Cref{cor:denseset} requires the exponent to be at least~$4$. However, we see no reason why this should be the general case and ask the following question: 
	\begin{question}	
		Are there uncountably many groups with the homological Dehn function $n^{\alpha}$ for a dense set of $\alpha\in [2, 4)$?
		Are there uncountably many groups with the homological Dehn function $n^2$? 
	\end{question}	
	Notice that answering the latter question would also answer the former one using the work of \cite{snowflake} in a similar way as we did  in the proof of~\Cref{cor:denseset}. 
	
	Finally, it would be of interest to find examples where the set $S$ plays a more prominent role in the calculation of the homological Dehn function. Doing so leads to the possibility of answering the following question: 
	
	\begin{restatable}{question}{Quncount}
		Are there uncountably many homological Dehn functions? 
	\end{restatable}

	Recall that the homological isoperimetric spectrum is the set 
	\[
	\operatorname{HIP} = \{\alpha\mid n^{\alpha} \text{ is the homological Dehn function of a group of type $FP_2$} \}.
	\]
	It would be of great interest to understand this set better.
	
	\smallskip
	It is known that for finitely presented groups the (analogously defined) isoperimetric spectrum has a gap between $\alpha=1$ and $\alpha=2$, both for Dehn functions and for homological Dehn functions. For Dehn functions this is a theorem of Gromov~\cite{Gro87} with detailed proofs given later in \cite{Ols91,Bow95,Pap}. For homological Dehn functions this follows from works~\cite[Thm.\,3.3]{Fle} and~\cite[Thm.\,A.1]{Ger98}. Remarkably, Fletcher's proof can be adapted to the homological Dehn functions of groups of type $FP_2$. Indeed, he works with surface diagrams for the so-called triangular presentations, in which all relators have length three. We can apply his arguments to the \emph{triangular homological finite presentations}, i.e.\ such homological finite presentations $\langle \mathcal{A}\mid\mid R_0\rangle$ in which each element of $R_0$ has length $3$. It is easy to see that by adding extra generators one can always ensure that any group of type $FP_2$ has a triangular finite homological presentation. Examining the proof of Theorem 3.3 of \cite{Fle}, we obtain:
\begin{theorem} 
Let $G$ be a group of type $FP_2$ with the homological Dehn function which is subquadratic. Then $G$ is hyperbolic. In particular, $\operatorname{HIP}\subset \{1\}\cup [2, \infty)$. \qed
\end{theorem}
Here a function $f\colon [0,\infty)\to[0,\infty)$ is called \emph{subquadratic}, if $\lim_{n\to\infty}f(n)/n^2=0$.
 
	By the argument as in the proof of \Cref{cor:denseset}, it follows from \cite{snowflake} that 
	\[
	\overline{\operatorname{HIP}} = \{1\}\cup [2, \infty).
	\]
	
	\vspace*{-1ex}
	\Quncountii

\end{document}